\documentclass[a4]{scrartcl}
\usepackage{xcolor}
\usepackage[colorlinks,allcolors=green!60!black]{hyperref}

\makeatletter
\@mparswitchfalse
\makeatother

\usepackage{xparse} 

\usepackage{todonotes,amsmath,amsthm,amsfonts,amssymb,import,xfrac,mathtools,dsfont,tikz,tikz-cd,pgfplots}

\pgfplotsset{compat=1.9}
\usetikzlibrary{
  pgfplots.groupplots,
  matrix
}
\newcommand{\tikloglogpic}[2]{%
  \begin{tikzpicture}
    \begin{loglogaxis}[#1]
      #2
    \end{loglogaxis}
\end{tikzpicture}}

\newcommand{\norm}[1]{\|#1\|}
\newcommand{\set}[1]{\{#1\}}
\newcommand{\abs}[1]{|#1|}
\newcommand{\ahalf}{{\frac{1}{2}}}

\newenvironment{keywords}%
   {\begin{trivlist}\item[]{\bfseries\sffamily Keywords:}\ }
   {\end{trivlist}}

\newtheorem{lemma}{Lemma}
\newtheorem{definition}[lemma]{Definition}
\newtheorem{theorem}[lemma]{Theorem}

\newtheorem{remark}[lemma]{Remark}
\newtheorem{assumption}[lemma]{Assumption}
\renewcommand{\div}{\operatorname{div}}

\newcommand\restrtmp[3][]{{
  #2 {#1|}_{#3}
}}
\newcommand\restr[3][-1]{{ 
  \ifcase#1\relax
    \restrtmp{#2}{#3}\or         
    \restrtmp[\big]{#2}{#3}\or   
    \restrtmp[\Big]{#2}{#3}\or   
    \restrtmp[\bigg]{#2}{#3}\or  
    \restrtmp[\Bigg]{#2}{#3}     
  \else
	\left.\kern-\nulldelimiterspace #2\vphantom{|} \right|_{#3}
  \fi
}}
\newcommand{\dif}{\mathop{}\!\mathrm{d}}

\newcommand{\euler}{\mathrm{e}}

\newcommand{\ik}{\mathfrak i_k}
\newcommand{\ak}{\mathfrak a_k}
\newcommand{\mk}{\mathfrak m_k}
\newcommand{\grid}{\mathds{I}}
\newcommand{\indicatorfunc}{\mathds{1}}

\newcommand{\oCdiff}{N}

\newcommand\sH[1]{H^{#1}}
\newcommand\sHz[1]{H_0^{#1}} 
\newcommand\sL[1]{L^{#1}}
\newcommand\sLd[1]{L^{#1}_\sigma}
\newcommand\sLa[1]{L_0^{#1}} 
\newcommand\sLJ[1]{L_{\!J}^{#1}} 
\newcommand\sLJp[1]{L_{\!J'}^{#1}} 
\newcommand\sW[1]{W^{#1}} 
\newcommand\sV[1]{V^{#1}} 
\newcommand\sQ[1]{Q^{#1}} 
\newcommand\sC[1]{C^{#1}} 
\newcommand\sCJ[1]{C_{\!J}^{#1}} 
\newcommand\sCJp[1]{C_{\!J'}^{#1}} 
\DeclareDocumentCommand{\sCcd}{o}{ 
\IfValueTF{#1}{C_{0, \sigma}^{#1}}{C_{0, \sigma}}
}
\newcommand\sP[1]{P^{#1}} 
\newcommand\sdG[1]{dG^{#1}} 
\newcommand\scG[1]{cG^{#1}} 

\newcommand\vars{\sigma} 
  
\DeclareDocumentCommand{\spL}{ s m }{
  \IfBooleanTF #1
  {\left(#2\right)}
  {(#2)}
}
\newcommand\spTS[1]{(\mkern-4mu(#1)\mkern-4mu)}
\newcommand\spLL[1]{\spTS{#1}} 

\DeclareDocumentCommand{\nS}{ s m }{
  \IfBooleanTF #1
	{\norm[2]{#2}}
	{\norm{#2}}
}

\newcommand\nTS[1]{\norm{#1}} 

\DeclareDocumentCommand{\nHH}{ o o m }{
\IfValueTF{#1}
  {\IfValueTF{#2} 
    {\nTS{#3}_{#1, #2}}
    {\nTS{#3}_{#1}}
  }
  {\nTS{#3}}
}

\DeclareMathOperator\Id{Id}

\newcommand\Cas[1]{C_{\hyperref[assumption:#1]{A\ref*{assumption:#1}}}}

\title{Second order pressure estimates for the Crank-Nicolson
  discretization of the incompressible Navier-Stokes Equations}

\author{Florian Sonner
\thanks{Universit\"at Erlangen-N\"urnberg, Erlangen, Germany (\url{florian.sonner@math.fau.de}).}
\and
Thomas Richter
\thanks{Universit\"at Magdeburg, Magdeburg and Interdisciplinary Center
  for Scientific Computing, Heidelberg, Germany,
(\url{thomas.richter@ovgu.de}).}}

\begin{document}

\maketitle

\begin{abstract}
	We provide optimal order pressure error estimates for the
        Crank-Nicolson semidiscretization  of the incompressible
        Navier-Stokes equations. Second order estimates for the
        velocity error are long known, we prove that the pressure
        error is of the same order if considered at interval
        midpoints, confirming previous numerical evidence. For
        simplicity we first give a proof under high regularity
        assumptions that include nonlocal compatibility conditions for
        the initial data, then use smoothing techniques for a proof
        under reduced assumptions based on standard local conditions
        only.

\end{abstract}

\section{Introduction}

We consider the Crank-Nicolson timestepping scheme for the Stokes and Navier-Stokes equations as described in the seminal paper by Heywood and Rannacher~\cite{HeywoodRannacher1990}, where optimal velocity error estimates were proven under weak regularity assumptions. Due to its implicit occurrence, the pressure error is only linear in time at the timesteps. Numerically it is well-known  that quadratic convergence can be recovered using the midpoint values of the pressure, see for example Reusken and Esser~\cite{ReuskenEsser2013}, Rank~\cite{Rang2008} or Hussain, Schieweck and Turek~\cite{HussainSchieweckTurek2013} who also consider higher order schemes. 

The main result of this paper is a proof of this result: If $p_k$ denotes the semidiscrete and $p$ the continuous pressure, we prove in section~\ref{sec:high-reg} the temporal $\sL2$-estimate
\[ \nTS{\ak p - p_k}_{\sL2(0, T; \sH1(\Omega))} \leq C k^2, \]
where $\ak p$ is an average over $p$, and the $\sL\infty$-estimate
\[ \nTS{\mk p - p_k}_{\sL\infty(0, T; \sH1(\Omega))} \leq C k^2 \]
where $\mk p$ denotes the midpoint values of $p$, for details on the notation see below. These results hold for solutions of both the Stokes and Navier-Stokes equations. Comparable are temporal $\sL2$-estimates by De Frutos, Garci\'ia-Archilla, John and Novo~\cite{deFrutosArchillaJohnNovo2016} for the Oseen problem in the fully discrete setting. In~\cite{deFrutosArchillaJohnNovo2018} these results are extended to the Navier-Stokes equations including optimal order estimates for the pressure in the interval midpoints. The evaluation of the adjoint variables in the midpoints by Meidner and Vexler~\cite{MeidnerVexler2011} for parabolic optimal control problems is of similar spirit. Rang~\cite{Rang2008} shows second order pressure convergence for a variant of the Crank-Nicolson timestepping scheme applied to the Navier-Stokes equations, where the pressure is treated like the velocity and split into explicit and implicit part. 

The previous results assume highly regular solutions which only exist if the initial data satisfies nonlocal compatibility conditions. Similar to Heywood and Rannacher~\cite{HeywoodRannacher1990} we reduce the assumptions on the data in section~\ref{sec:low-reg} and use the smoothing properties of the solution operator: Replacing the first timestep by an implicit Euler step, we have the $\sL2$-estimate
\[ \nTS{\tau_k^{\frac 3 2} (\ak p - p_k)}_{\sL2(t_1, T; \sH1(\Omega))} \leq C k^2 \]
where $\tau_k$ is a discretization of the continuous smoothing function $\tau(t) \coloneqq \min(1, t)$, and replacing the first two steps by implicit Euler steps we have the $\sL\infty$-estimate
\[ \nTS{\tau_k^2 (\mk p - p_k)}_{\sL\infty(t_2, T; \sH1(\Omega))} \leq C k^2.\]
These results are derived using energy techniques and technically involved, but based on the simple observation that the discrete smoothing creates jump terms which require a cascade of estimates due to the weak bounds available for the Crank-Nicolson scheme.

In section~\ref{sec:num} we present a numerical study illustrating the optimality of the error estimates and the necessity to consider both a weighted norm and initial Euler steps, if the initial data does not satisfy the compatibility conditions. 

We interpret the Crank-Nicolson scheme as a discontinuous Petrov-Galerkin in time method, extending a construction by Aziz and Monk~\cite{AzizMonk1989}: Velocities are approximated using continuous, piecewise-linear functions, the pressure with discontinuous, piecewise-constant functions. The test functions for both velocity and pressure are discontinuous, piecewise-constant in time. This mismatch between velocity test and ansatz spaces is central to the a priori analysis of the problem. Further analysis of Petrov-Galerkin in time methods can be found in Schieweck~\cite{Schieweck2010} for parabolic equations, of related interest are also results by Chrysafinos and Karatzas~\cite{ChrysafinosKaratzas2015} for discontinuous Galerkin methods applied to the Stokes equations, where best approximation results for the velocity error are derived.

\subsection{General Notation}
In the following let $\Omega \subseteq \mathds{R}^d$ with $d \in \set{2, 3}$ be a bounded domain with regularity described later on. Let $I \coloneqq [0, T)$ with $T > 0$ denote the finite time interval on which a solution is sought.	We write $\spL{\,\cdot\,, \,\cdot\,}$ for the scalar product on $\sL2(\Omega)$ but also for the duality product on a generic Banach space if no confusion is possible. For spatial norms we omit the domain $\Omega$ and write e.g.\ $\nS{\cdot}_{\sL{p}} \coloneqq \nS{\cdot}_{\sL{p}(\Omega)}$. Functions in $\sLa2(\Omega)$ have zero average and those in $\sLd2(\Omega)$ are solenoidal, i.e.\ have zero (weak) divergence. We write $P$ for the Helmholtz projection. 

We use the Bochner spaces $\sL{p}(0, T; X)$, $\sW{m, p}(0, T; X)$ and $\sH{m}(0, T; X)$ for a Banach space $X$, $m \in \mathbb N$ and $1 \leq p \leq \infty$. For $J \subset I$ we write $\spLL{\cdot, \cdot}_J$ for the scalar product on $\sL2(J, \sL2(\Omega))$, but also for the duality product on $\sL2(J, X)$ if no confusion is possible. We write $\nHH{\cdot}_J$ for the norm on $\sL2(J, \sL2(\Omega))$ and omit $J$ in both norm and scalar product if $J = I$. For general Bochner spaces we abbreviate $\nTS{\cdot}_{\sLJ{p} X} \coloneqq \nTS{\cdot}_{\sL{p}(J, X)}$ and again omit $J$ if $J = I$. 

The natural spaces of velocity regularity are induced by the Stokes operator $-P \Delta$, see \cite{sohr2001} for details: Since $\Omega$ will have at least a $C^2$-boundary, the Stokes operator $-P \Delta\colon \mathcal D(-P \Delta) \subset \sLd2(\Omega) \to \sLd2(\Omega)$ has domain $\mathcal D(-P\Delta) = \sLd2(\Omega) \cap \sHz1(\Omega) \cap \sH2(\Omega)$. Furthermore, $- P \Delta$ is positive, selfadjoint, compact and has a bounded inverse. In particular we may define $\sV{s} \coloneqq \mathcal D((-P \Delta)^{\frac s 2})$ for $s \geq 0$ with graph norm $\norm{\cdot}_{\sV{s}}$ and $\sV{-s} \coloneqq (\sV{s})'$. Then $\sV0 \cong \sLd2(\Omega)$ and $\sV{s} \cong \sLd2(\Omega) \cap \sHz1(\Omega) \cap \sH{s}(\Omega)$ for $s \in \mathbb N$. For $u, v \in \sV1$ there holds $\spL{\nabla u, \nabla v} = \spL{(-P\Delta)^{\frac 12} u, (- P \Delta)^{\frac 12} v}$. For brevity we use the notation $\sQ{s} \coloneqq \sLa2(\Omega) \cap \sH{s}(\Omega)$ for the pressure regularity spaces with the $\sH{s}(\Omega)$-norm.

We denote the identity operator by $\Id$ and the indicator function on $J$ by $\indicatorfunc_J$. For a real number $a \in \mathbb R$ we write $(a)^+ \coloneqq \max\set{0, a}$ and $(a)^- \coloneqq \min\set{0, a}$. We denote by $C > 0$ a generic constant independent of the timestep size which may change with each occurrence.

\subsection{Discrete Spaces and Operators}

Let $X$ be a real Hilbert space.

\begin{definition}
	Let $0 = t_0 < t_1 < \dots < t_N = T$. We call $\grid_k = \set{ I^1, \dots, I^N }$ with $I^n \coloneqq (t_{n-1}, t_n]$ for $n = 1, \dots, N$ a \emph{discretization of $I$ with nodes $(t_n)_{n = 0}^N$}. Let $k_n \coloneqq \abs{I^n}$, $k \coloneqq \max_{n = 1, \ldots, N} k_n$ and denote the midpoints by $t_{n-\ahalf} \coloneqq \frac 12 (t_{n-1} + t_n)$ for $n = 1, \dots, N$.
\end{definition}

For the remainder of the paper $\grid_k$ is fixed and satisfies, for the Stokes equations,\ldots

\begin{assumption}\label{as:time1}
	There exists $\kappa > 0$ such that for $n = 1, \dots, N-1$ there holds $\kappa^{-1} \le \frac{k_n}{k_{n+1}} \le \kappa$. The generic constants $C > 0$ may depend on $\kappa$.
\end{assumption}

\ldots and for the Navier-Stokes equations the stronger\ldots

\begin{assumption}\label{as:time2}
	There exists $\rho \geq 1$ such that $\max_{n = 1, \ldots, N} k_n \leq \rho \min_{n = 1, \ldots, N} k_n$. The generic constants $C > 0$ may depend on $\rho$.
\end{assumption}

\begin{definition}
	We define the \emph{continuous and discontinuous Galerkin spaces on $X$ of degree $r \in \mathbb N_0$} as
	\begin{alignat*}{2}
		\scG{r}(\grid_k, X) &\coloneqq \{ f_k \in \sC0([0, T], X) &\;\,\mid\,\;& \restr{f_k}{I^n} \in \sP{r}(I^n, X) \; \forall I^n \in \grid_k \},\\
		\sdG{r}(\grid_k, X) &\coloneqq \{ f_k : (0, T] \to X &\;\,\mid\;\,& \restr{f_k}{I^n} \in \sP{r}(I^n, X) \; \forall I^n \in \grid_k \},
	\end{alignat*}
	where $\sP{r}(I^n, X)$ is the space of polynomials of degree $r$ with values in $X$ on $I^n$. We equip $\sdG{r}(\grid_k, X)$ with the $\sL2(0, T; X)$-topology.
\end{definition}

\begin{definition}
	For $u\colon I \to X$ sufficiently regular, we define
	\begin{itemize}
		\item 
			the \emph{nodal interpolation operator} 
			\[ \ik u \in \scG{1}(\grid_k, X), \quad (\ik u)(t_n) \coloneqq u(t_n) \quad\text{for $n = 0, \dots, N$,} \]
		\item
			the \emph{$\sL2$-projection} onto $\sdG0(\grid_k, X)$, which corresponds to \emph{averaging}
			\[ \ak u \in \sdG{0}(\grid_k, X), \quad \restr{\ak u}{I^n} \coloneqq \frac{1}{k_n} \int_{I^n} u(t) \dif t, \quad \text{for $n = 1, \ldots, N$,} \]
		\item
			the \emph{constant continuation of midpoint values} 
			\[ \mk u \in \sdG{0}(\grid_k, X), \quad \restr{\mk u}{I^n} \coloneqq u(t_{n-\ahalf}) \quad \text{for $n = 1, \dots, N$.} \]
	\end{itemize}
\end{definition}

\begin{remark}
	The operator $\ik$ is used for the velocity and the operator $\mk$ for the pressure error estimate. The averaging operator $\ak$ occurs naturally in a-priori estimates for the Crank-Nicolson scheme, see remark~\ref{rm:dual-space}. 
\end{remark}

\begin{lemma}\label{th:nodal-int-properties}
	Let $2 \leq p \leq \infty$ and $m \in \set{1,2}$. For $u \in \sW{m,p}(0, T; X)$ there holds
	\[ \nTS{u - \ik u}_{\sL{p} X} \le C k^m \nTS{\partial_t^m u}_{\sL{p} X}. \]
\end{lemma}
\begin{proof}
	Follows by reference transformation techniques.
\end{proof}

\begin{lemma}\label{th:midpoint-quadratic-estimate}
	For $u \in \sW{2,\infty}(0, T; X)$ there holds
	\[ \nTS{\ak u - \mk u}_{\sL\infty X} \le C k^2 \nTS{\partial_{tt} u}_{\sL\infty X}. \]
\end{lemma}
\begin{proof}
	For each $I^n = (t_{n-1}, t_n] \in \grid_k$ we have, by Taylor expansion around $t_{n-\ahalf}$:
	\begin{align*}
		\int_{I^n} u(t) - u(t_{n-\frac 12}) \dif t &= \partial_t u(t_{n-\ahalf}) \int_{I^n} t - t_{n-\ahalf} \dif t + \int_{I^n} \int_{t_{n-\ahalf}}^t (t-s) \partial_{tt} u(s) \dif s \dif t.
	\end{align*}
	The first term on the right vanishes and we can hence conclude
	\begin{equation*}
		\max_{I^n \in \grid_k} \norm{\restr{(\ak u - \mk u)}{I^n}}_X \leq C k^2 \nTS{\partial_{tt} u}_{\sL\infty X}.\qedhere
	\end{equation*}
\end{proof}

\begin{lemma}\label{th:nodal-midpoint-estimates}
	For $2 \leq p \leq \infty$ and $u \in \sW{1,p}(0, T; X)$ there holds:
	\[ \nTS{u - \ak \ik u}_{\sL{p} X} \le C k \nTS{\partial_t u}_{\sL{p} X}. \]
\end{lemma}
\begin{proof}
	Follows from the reference transformation technique and the Bramble-Hilbert lemma, since $u - \ak \ik u$ vanishes for piecewise constant functions.
\end{proof}

\begin{remark}\label{rm:approx-weights}
	In the second part of this paper we apply the estimates of this section to $\sL{p}$-spaces equipped with a discrete weight function $\tau_k^\alpha \in \sdG0(\grid_k, \mathbb R_+)$. Since $\restr{\tau_k^\alpha}{I^n} \equiv \text{const}$ the weighting does not affect the error estimates, such that e.g.\ lemma~\ref{th:nodal-int-properties} reads
	\[ \nTS{\tau_k^\alpha (u - \ik u)}_{\sL{p} X} \leq C k^m \nTS{\tau_k^\alpha \partial_t^m u}_{\sL{p} X}. \]
\end{remark}

\section{Estimates with High Regularity}\label{sec:high-reg}

\subsection{Stokes Equations}

Let $u$ and $p$ denote a solution of the Stokes equations
\begin{equation}\label{eq:st}
	\partial_t u - \Delta u + \nabla p = f, \quad \div u = 0 \qquad \text{in $\Omega$},
\end{equation}
with $u(0) = u^0$, $\restr[0]{u}{\partial\Omega} = 0$ and $\int_\Omega p = 0$.
In this section, we require solutions of high regularity and hence the following assumptions on the problem data:
\begin{assumption}\label{assumption:stokes1}
  We assume that $\Omega$ has a $C^5$-boundary, the
  right hand side has regularity $f\in C^0(\bar I, \sH3(\Omega))\cap
  C^1(\bar I, \sH1(\Omega))\cap 
  \sH2(I, \sL2(\Omega))$ and that the initial data $u^0\in \sV5$ satisfies the
	compatibility conditions that $p^0$ and $q^0$ with
  \begin{equation}\label{comp:stokes1}
		\Delta p^0 = \nabla\cdot f(0), \quad \Delta q^0 = \nabla\cdot \partial_t f(0) \quad \text{ in }\Omega
  \end{equation}
  can be chosen such that
  \begin{equation}\label{comp:stokes2}
    \nabla p^0 = f(0)+\Delta u^0,\quad
    \nabla q^0 = \partial_t f(0)+\Delta f(0) + \Delta^2 u^0 - \Delta \nabla p^0 \quad \text{ on }\partial\Omega.
  \end{equation}
  We write
  \[
		\Cas{stokes1}:= C \left(
		\nTS{u^0}_{\sV5}
		+\nTS{f}_{C^0\sH3}
		+\nTS{f}_{C^1\sH1}
		+\nTS{f}_{\sH2\sL2} \right)
  \]
	where $C > 0$ is independent of the data and, by abuse of notation, may change with each occurrence of $\Cas{stokes1}$.
\end{assumption}
In~\cite[theorem 2.1 and theorem
  2.2]{Temam1982} it is shown that exactly under these conditions
there exists a solution to the Stokes equations that satisfies the
bound
\begin{equation}\label{reg:stokes}
	\nTS{u}_{\sC0 \sV5} + \nTS{u}_{\sC1 \sV3} + \nTS{u}_{\sC2\sV1} + \nTS{u}_{\sH2\sV2} + \nTS{p}_{\sH2\sQ1} \le
  \Cas{stokes1}. 
\end{equation}
The equations \eqref{comp:stokes1} and \eqref{comp:stokes2} require $p^0$ and $q^0$ to solve Poisson problems
with overdetermined boundary conditions. This makes assumption~\ref{assumption:stokes1} not only hard to check for given $u^0$ and $f$ but unlikely to hold. Nevertheless we will assume its validity 
throughout this section.  A more technical analysis without compatibility assumptions will be carried out in section~\ref{sec:low-reg}.

To formulate the timestepping scheme in dual spaces of $\sdG0$-functions we note:

\begin{remark}\label{rm:dual-space}
	Let $X$ be a Hilbert space. Since $\sdG0(\grid_k, X') \subset \sL2(I, X')$, we have both $\sL2(I, X') \cong \sL2(I, X)'$ and $\sdG0(\grid_k, X') \cong \sdG0(\grid_k, X)'$ and the diagram
	\begin{center}
		\vspace{1em}
		\begin{tikzcd}[column sep=6em,row sep=2.5em]
			\sdG0(\grid_k, X') \arrow[r, "\cong"', "\set{\phi_k\;\mapsto\;\spLL{\,\cdot\,, \phi_k}}"] \arrow[d, hook] \arrow[d, "\subset", hook] & \sdG0(\grid_k, X)'\\
			\sL2(I, X') \arrow[u, shift left=3, "\ak"] \arrow[r, "\cong", "\set{\phi\;\mapsto\;\spLL{\,\cdot\,, \phi}}"'] & \sL2(I, X)' \arrow[u, "{\left.\rule{0pt}{6pt}\,\cdot\, \right|_{\sdG0(\grid_k, X)}}"']
		\end{tikzcd}
		\vspace{1em}
	\end{center}
	commutes. We emphasize the occurrence of $\ak$ in this diagram. In particular, any $f \in \sL2(I, X')$ can be understood in $\sdG0(\grid_k, X)'$ through
	\begin{equation}\label{eq:dual-incl}
		\spLL{f, \phi_k} \coloneqq \int_I \spL{f(t), \phi_k(t)} \dif t = \int_I \spL{\ak f(t), \phi_k(t)} \dif t = \spLL{\ak f, \phi_k}
	\end{equation}
	for arbitrary $\phi_k \in \sdG0(\grid_k, X)$. For the dual norm we have
	\begin{equation*}
		\norm{f}_{\sdG0(\grid_k, X)'} \coloneqq \sup_{\phi_k \in \sdG0(\grid_k, X)} \frac{\int_I \spL{f(t), \phi_k(t)} \dif t}{\norm{\phi_k}_{\sL2 X}} = \norm{\ak f}_{\sL2 X'}.
	\end{equation*}	
\end{remark}

We can now give an abstract formulation of the Crank-Nicolson scheme as solution $u_k \in \scG{1}(\grid_k, \sV1)$ and $p_k \in \sdG{0}(\grid_k, \sQ0)$ of
\begin{equation}\label{eq:st:disc}
	\partial_t u_k - \Delta u_k + \nabla p_k = f \quad\text{in $\sdG0(\grid_k, \sHz1(\Omega))'$}
\end{equation}
with $u_k(0) = u^0$. The equation must be understood with left- and right-hand side interpreted as elements of $\sdG0(\grid_k, \sHz1(\Omega))'$ as in \eqref{eq:dual-incl}. Using $\ak$ we have more explicitly:
\begin{equation}\label{eq:st:disc-avg}
	\partial_t u_k - \Delta \ak u_k + \nabla p_k = \ak f,
\end{equation}
now as an equality of $\sdG0$-functions. With $u_k^n \coloneqq u_k(t_n)$ and $p_k^n \coloneqq \restr[0]{p_k}{I^n}$, the equation \eqref{eq:st:disc} / \eqref{eq:st:disc-avg} can be written as a classical timestepping scheme:
\begin{equation}\label{eq:st:disc-classical}
	u_k^n - u^{n-1}_k - \tfrac{k_n}{2} \Delta (u_k^n + u_k^{n-1}) + k_n \nabla p_k^n = \int_{I^n} f \dif t
\end{equation}
for $n = 1, \ldots, N$ and $u_k^0 \coloneqq u^0$. This is the usual Crank-Nicolson scheme with exact integration of the right-hand side. Equation \eqref{eq:st:disc} emphasizes the relation to the continuous problem with replaced time-discrete trial and test spaces and can be thought of as integral version of the pointwise-in-time scheme \eqref{eq:st:disc-classical}. Our error analysis is based on \eqref{eq:st:disc-avg} rather than \eqref{eq:st:disc-classical}, trading notational simplicity for, in our opinion, more transparent proofs since algebraic manipulations and choices of test functions can be identified as standard integration techniques, such as partial integration, and use of temporal projections.

\begin{remark}
	We emphasize the purely implicit occurrence of the pressure in \eqref{eq:st:disc-classical}. The character of the analysis changes significantly if $\frac{k_n}{2} \nabla (p_k^n + p_k^{n-1})$ is used as pressure term, which corresponds to \eqref{eq:st:disc} with $p_k \in \scG1(\grid_k, \sQ0)$ and leads more naturally to second-order pressure estimates. This approach is investigated by Rang~\cite{Rang2008} and requires the construction of an initial pressure $p^0$. 
\end{remark}

The following stability estimate is formulated in a generality only needed for section~\ref{sec:low-reg}. In this section only the cases $s = 1$ and $s = 2$ are used.

\begin{lemma}[Discrete stability]\label{th:st:disc-stab}
	Let $s \in \mathbb Z$. Let $v_k \in \scG1(\grid_k, \sV{s+1})$ and $r_k \in \sdG0(\grid_k, \sV{s-1})$ be such that 
	\begin{equation}\label{eq:disc-stab:diveq}
		\partial_t v_k - P \Delta v_k = r_k \; \text{in $\sdG0(\grid_k, \sV{1-s})'$} \quad \xLeftrightarrow{\;\eqref{eq:dual-incl}\;} \quad \partial_t v_k - P \Delta \ak v_k = r_k.
	\end{equation}
	Then we have the stability estimate
	\begin{equation} \label{eq:disc-stab:div}
		\nTS{v_k}_{\sL\infty \sV{s}} + \nTS{\partial_t v_k}_{\sL2\sV{s-1}} + \nTS{\ak v_k}_{\sL2\sV{s+1}} \le C \left( \nS{v_k^0}_{\sV{s}} + \nTS{r_k}_{\sL2 \sV{s-1}} \right).
	\end{equation}
\end{lemma}
\begin{proof}
	By assumption $(-P\Delta)^s v_k(t) \in \sV{1-s}$ for $t \in \bar I$. Hence \eqref{eq:disc-stab:diveq}, as element of $\sL2(I, \sV{1-s})'$, can be tested with $(-P\Delta)^s v_k \indicatorfunc_J$ where $J \coloneqq (0, t_n]$ for $n = 1, \ldots, N$:
	\begin{equation}\label{eq:th:st:disc-stab:tmp}
		\spLL{\partial_t v_k, (-P\Delta)^s v_k}_J + \spLL{
                  -P\Delta \ak v_k, (-P\Delta)^s v_k}_J = \spLL{r_k, (-P\Delta)^s v_k}_J.
	\end{equation}
	Partial integration in time of the continuous function $v_k$ implies
	\begin{equation*}
		\spLL{\partial_t v_k, (-P\Delta)^s v_k}_J = \tfrac 1 2 \left( \nS{v_k^n}_{\sV{s}}^2 - \nS{v_k^0}_{\sV{s}}^2 \right).
	\end{equation*}
	For the other term on the left of \eqref{eq:th:st:disc-stab:tmp} we use the projection property to insert $\ak$:
	\begin{equation*}
		\spLL{- P \Delta \ak v_k, (-P\Delta)^s v_k}_J = \spLL{- P \Delta \ak v_k, (-P\Delta)^s \ak v_k}_J = \nTS{\ak v_k}_{\sLJ2\sV{s+1}}^2. 
	\end{equation*}
	For the right-hand side of \eqref{eq:th:st:disc-stab:tmp} we again insert $\ak$ and estimate with Young's inequality 
	\begin{equation*}
		\spLL{r_k, (-P \Delta)^s v_k}_J = \spLL{r_k, (- P \Delta)^s \ak v_k}_J \leq \tfrac 12 \norm{r_k}_{\sL2 \sV{s-1}}^2 + \tfrac 1 2 \norm{\ak v_k}_{\sLJ2 \sV{s+1}}^2
	\end{equation*}
	and combination of the three previous estimates implies for \eqref{eq:th:st:disc-stab:tmp}:
	\begin{equation*}
		\norm{v_k^n}_{\sV{s}}^2 + \norm{\ak v_k}_{\sLJ2 \sV{s+1}}^2 \leq \norm{v_k^0}_{\sV{s}}^2 + \norm{r_k}_{\sL2 \sV{s-1}}^2.
	\end{equation*}
	Taking the maximum over $n = 1, \ldots, N$ we arrive at the estimates for $\nTS{v_k}_{\sL\infty \sV{s}}$ and $\nTS{\ak v_k}_{\sL2\sV{s+1}}$. For the remaining estimate of $\norm{\partial_t v_k}_{\sL2 \sV{s-1}}$ we have by \eqref{eq:disc-stab:diveq}
	\begin{align*}
		\norm{\partial_t v_k}_{\sL2 \sV{s-1}} &= \norm{r_k + P \Delta \ak v_k}_{\sL2 \sV{s-1}} \leq \norm{r_k}_{\sL2 \sV{s-1}} + \norm{P \Delta \ak v_k}_{\sL2 \sV{s-1}} \\
		&\leq \norm{r_k}_{\sL2 \sV{s-1}} + \norm{\ak v_k}_{\sL2 \sV{s+1}}
	\end{align*}
	and \eqref{eq:disc-stab:div} follows by our estimate for $\norm{\ak v_k}_{\sL2 \sV{s+1}}$.
\end{proof}

\begin{remark}
	We emphasize that in \eqref{eq:disc-stab:div} only $\norm{\ak v_k}_{\sL2 \sV{s+1}}$ can be controlled, i.e.\ the integral of the averages $\frac 12(v_k^{n-1} + v_k^n)$ for $n = 1, \ldots, N$, and not $\norm{v_k}_{\sL2 \sV{s+1}}$. This will be a central problem in section~\ref{sec:low-reg}.
\end{remark}

\begin{remark}\label{rm:cont-stab}
	Given a sufficiently regular continuous solution $v$ of
	\begin{equation*}	
		\partial_t v - P \Delta v = r \quad \text{in $\sL2(I, \sV{1-s})'$}
	\end{equation*}
	we can derive stability estimates by testing with $\phi = (-P\Delta)^s v \indicatorfunc_{(0, t]}$ for $t \in I$ a.e. For our time-discrete solution $v_k$ of 
	\begin{equation*}
		\partial_t v_k - P \Delta v_k = r_k \quad \text{in $\sdG0(\grid_k, \sV{1-s})'$},
	\end{equation*}
	we derived stability estimates by testing the equivalent formulation
	\begin{equation*}
		\partial_t v_k - P \Delta \ak v_k = r_k,
	\end{equation*}
	with $\phi_k = (-P\Delta)^s v_k \indicatorfunc_{(0, t_n]}$ for $n = 1, \ldots, N$. We note the similarity of the approaches and the fact that $\phi_k$ is not a valid test function in the first formulation of the time-discrete problem, but in the second one since $\sdG0(\grid_k, \sV{s-1}) \subset \sL2(I, \sV{s-1}) \cong \sL2(I, \sV{1-s})'$.
\end{remark}

With these preparations we can prove second order error estimates for both velocity and pressure based on the following error identities:

\begin{lemma}\label{th:st:error-id}
	For the solutions $(u, p)$ of the continuous and $(u_k, p_k)$ of the time-discrete Stokes problem we have the velocity error identity
	\begin{equation}\label{eq:st:error-u}
		\partial_t(u_k - \ik u) - P \Delta \ak (u_k - \ik u) = -P\Delta \ak(u - \ik u)
	\end{equation}
	and the pressure error identity
	\begin{equation}\label{eq:st:error-p} 
		\nabla(p_k - \ak p) = (P - \Id) \Delta \ak (u - u_k).
	\end{equation}
\end{lemma}
\begin{proof}
	By definition of $(u, p)$ and $(u_k, p_k)$ we have Galerkin orthogonality:
	\begin{equation}\label{eq:th:st:error:tmp}
		\partial_t (u_k - u) - \Delta (u_k - u) + \nabla (p_k - p) = 0 \quad \text{in $\sdG0(\grid_k, \sL2(\Omega))'$}
	\end{equation}
	and hence in particular 
	\begin{equation*}
		\ak \partial_t (u_k - u) - P \Delta \ak (u_k - u) = 0.
	\end{equation*}
	Adding and subtracting $\ik u$ we arrive at \eqref{eq:st:error-u} if we can show that $\ak \partial_t(u - \ik u) = 0$, noting that $\ak \partial_t (u_k - \ik u) = \partial_t (u_k - \ik u)$. The identity $\ak \partial_t(u - \ik u) = 0$ follows from
	\begin{equation*}
		\restr{\ak \partial_t(u - \ik u)}{I^n} = k_n^{-1} \int_{I^n} \partial_t (u - \ik u) \dif t = k_n^{-1} \left( (u - \ik u)(t_n) - (u - \ik u)(t_{n-1}) \right) = 0
	\end{equation*}
	for $I^n \in \grid_k$, since $(u - \ik u)(t_j) = 0$ for all $j = 0, \dots, N$. To prove the pressure identity, \eqref{eq:st:error-p}, the Galerkin orthogonality \eqref{eq:th:st:error:tmp} yields that
	\begin{equation*}
		\nabla(p_k - \ak p) = -\partial_t(u_k - \ik u) - \Delta \ak(u - u_k)
	\end{equation*}
	and by \eqref{eq:st:error-u} we have $\partial_t(u_k - \ik u) = - P \Delta \ak(u - u_k)$ which implies \eqref{eq:st:error-p}.
\end{proof}

\begin{theorem}\label{th:st:error-l2}
  Let assumption~\ref{assumption:stokes1} hold. Then,
	the Crank-Nicolson time discretization \eqref{eq:st:disc} of the Stokes equations \eqref{eq:st} satisfies the a priori error estimate
  \[ \nTS{u - u_k}_{\sL\infty \sV1} + \nTS{\ak(u - u_k)}_{\sL2\sV2} +
  \nTS{\ak p - p_k}_{\sL2\sQ1} \le \Cas{stokes1} k^2. \]
\end{theorem}
\begin{proof}
	We split the velocity error as follows:
	\begin{align*}
		\norm{u - u_k}_{\sL\infty \sV1} + \norm{\ak (u - u_k)}_{\sL2 \sV2} &\leq \norm{u - \ik u}_{\sL\infty \sV1} + \norm{\ak ( u - \ik u)}_{\sL2 \sV2} \\
		&\qquad+ \norm{u_k - \ik u}_{\sL\infty \sV1} + \norm{\ak (u_k - \ik u)}_{\sL2 \sV2}.
	\end{align*}
	We apply the discrete stability estimate \eqref{eq:disc-stab:div} with $s = 1$ to the error identity \eqref{eq:st:error-u} with $v_k \coloneqq u_k - \ik u$ and $r_k \coloneqq - P \Delta \ak(u - \ik u)$. Let us note that, here and in the following, it is easy to check that the the discrete solution has enough regularity to apply the stability estimate using the regularity theory for the stationary Stokes equations. For the last two terms above this yields
	\begin{align*}
		\MoveEqLeft\norm{u_k - \ik u}_{\sL\infty \sV1} + \norm{\ak (u_k - \ik u)}_{\sL2 \sV2}\\
		&\leq C \norm{{-} P\Delta \ak(u - \ik u)} = C \norm{\ak (u - \ik u)}_{\sL2 \sV2}.
	\end{align*}
	Combining this estimate with the previous one, using the stability of $\ak$ and the interpolation error estimate from lemma~\ref{th:nodal-int-properties} we arrive at
	\begin{align*}
		\MoveEqLeft\norm{u - u_k}_{\sL\infty \sV1} + \norm{\ak (u - u_k)}_{\sL2 \sV2} \leq \norm{u - \ik u}_{\sL\infty \sV1} + C \norm{\ak (u - \ik u)}_{\sL2 \sV2}\\
		&\leq C k^2 \left( \norm{\partial_{tt} u}_{\sL\infty \sV1} + \norm{\partial_{tt} u}_{\sL2 \sV2} \right) \leq \Cas{stokes1} k^2.
	\end{align*}
	For the pressure error we use \eqref{eq:st:error-p} and the validity of Poincar\'e's inequality on $Q^1$:
	\begin{align*}
		\nTS{p_k - \ak p}_{\sL2 \sQ1} &\le C \nTS{\nabla(p_k - \ak p)} = C \nTS{(P - \Id) \Delta \ak(u - u_k)}\\
		&\leq C \nTS{\ak(u - u_k)}_{\sL2 \sV2} \leq \Cas{stokes1} k^2.\qedhere
	\end{align*}
\end{proof}

The previous theorem considered the pressure error in an integral sense. For pointwise-in-time pressure errors we need to increase the regularity assumptions:

\begin{assumption}\label{assumption:stokes2}
	We assume $\Omega$ has a $C^6$-boundary, the right-hand side has regularity $f\in
  C^0(\bar I, \sH4(\Omega))\cap C^1(\bar I, \sH2(\Omega))\cap
	C^2(\bar I, \sL2(\Omega))\cap \sH3(I, \sV{-1})$ and the initial data $u^0 \in \sV6$ satisfies the compatibility conditions \eqref{comp:stokes1} and \eqref{comp:stokes2}. We write
  \[
	\Cas{stokes2} :=
	C \left( \nTS{u^0}_{\sV6}
  +\nTS{f}_{C^0\sH4}
  +\nTS{f}_{C^1\sH2}
  +\nTS{f}_{C^2\sL2}
	+\nTS{f}_{\sH3\sV{-1}} \right)
  \]
	where $C > 0$ is independent of the data and, by abuse of notation, may change with each occurrence of $\Cas{stokes2}$.
\end{assumption}
It is shown in~\cite[theorem 2.1 and theorem
  2.2]{Temam1982} that under these conditions the solution
satisfies
\begin{equation}\label{reg:higher}
	\nTS{u}_{\sC0\sV6} + \nTS{u}_{\sC1 \sV4} + \nTS{u}_{\sC2\sV2} + \nTS{u}_{\sH2\sV3} +
	\nTS{p}_{\sC2 \sQ1} \le \Cas{stokes2}. 
\end{equation}

\begin{theorem}\label{th:st:error-linf}
	Let assumption~\ref{assumption:stokes2} hold. Then, the Crank-Nicolson time discretization \eqref{eq:st:disc} of the Stokes equations \eqref{eq:st} satisfies the a priori error estimate
  \[ \nTS{u - u_k}_{\sL\infty \sV2} + \nTS{\ak (u - u_k)}_{\sL2 \sV3} + \nTS{\mk p - p_k}_{\sL\infty \sQ1} \leq \Cas{stokes2} k^2. \]
	In particular, the discrete pressure $p_k$ must be compared to the continuous pressure $p$ evaluated at interval midpoints for second order convergence.
\end{theorem}
\begin{proof}
	For the velocity error we can repeat the arguments from theorem~\ref{th:st:error-l2} with one order of regularity higher, yielding
	\begin{align*}
		\MoveEqLeft\norm{u - u_k}_{\sL\infty \sV2} + \nTS{\ak (u - u_k)}_{\sL2 \sV3} \leq \norm{u - \ik u}_{\sL\infty \sV2} + C \norm{\ak (u - \ik u)}_{\sL2 \sV3} \\
		&\leq C k^2 \left( \norm{\partial_{tt} u}_{\sL\infty \sV2} + \nTS{\partial_{tt} u}_{\sL2 \sV3} \right) \leq \Cas{stokes2} k^2.
	\end{align*}
	For the pressure error we use \eqref{eq:st:error-p} to get
	\begin{equation*}
		\nTS{\ak p - p_k}_{\sL\infty \sQ1} \leq C \nTS{(P - \Id) \Delta \ak (u - u_k)}_{\sL\infty \sL2} \leq C \nTS{u - u_k}_{\sL\infty \sV2}.
	\end{equation*}
	Using the already established estimate for $u - u_k$ and lemma~\ref{th:midpoint-quadratic-estimate} we conclude that
	\begin{align*}
		\MoveEqLeft\nTS{\mk p - p_k}_{\sL\infty \sQ1} \leq \nTS{\ak p - \mk p}_{\sL\infty \sQ1} + \nTS{\ak p - p_k}_{\sL\infty \sQ1}\\
		&\leq C k^2 \nTS{\partial_{tt} p}_{\sL\infty \sQ1} + C \nTS{u - u_k}_{\sL\infty \sV2} \leq \Cas{stokes2} k^2. \qedhere
	\end{align*}
\end{proof}

\begin{remark}
	The proof of theorem~\ref{th:st:error-linf} shows that $\mk p$ in the pressure estimate could be replaced with $\ak p$, the latter in fact being more natural in view of \eqref{eq:st:error-u}. We formulate all $\sL\infty$-in-time pressure estimates with $\mk p$ due to its simpler evaluation.
\end{remark}

\subsection{Navier-Stokes Equations}

In this section let $u$ and $p$ denote a solution to the Navier-Stokes equations
\begin{equation}\label{eq:ns}
	\partial_t u - \Delta u + u \cdot \nabla u + \nabla p = f, \quad \div u = 0 \qquad \text{in $\Omega$},
\end{equation}
with $u(0) = u^0$, $\restr[0]{u}{\partial\Omega} =
0$ and $\int_\Omega p = 0$. If $n = 3$ we require $T$ or the data to
be sufficiently small to guarantee existence of weak solutions. As in the
last section we will assume that the data is sufficiently regular to
allow the minimal required regularity for quadratic pressure estimates
without smoothing techniques. Throughout this section, the problem
data has to satisfy the 
following assumption:
\begin{assumption}\label{assumption:ns1}
  We assume that $\Omega$ has a $C^5$-boundary, the
  right hand side is given in $f\in C^0(\bar I, \sH3(\Omega))\cap
  C^1(\bar I, \sH1(\Omega))\cap 
  \sH2(I, \sL2(\Omega))$ and that the initial data is given in
  $u^0\in \sV5$ and satisfies the
  compatibility condition that $p^0$ and $q^0$ with 
  \begin{align}\label{comp:ns1}
		&\left\{
		\begin{aligned}
			\Delta p^0 &= \nabla\cdot \big(f(0)-u^0\cdot\nabla
			u^0\big),\\
			\Delta q^0 &= \nabla\cdot
			\partial_t f(0) - 2 \operatorname{tr}\left(\left(\nabla
			f(0)+\nabla\Delta u^0 - \nabla(u^0\cdot\nabla u^0) -
			\nabla^2 p^0 \right)\nabla
			u^0\right)
		\end{aligned}
		\right.
		\intertext{in $\Omega$ can be chosen such that}
		\label{comp:ns2}
		&\left\{
    \begin{aligned}
      \nabla p^0 &= f(0)+\Delta u^0,\\
			\nabla q^0 &= \partial_t f(0) +\Delta f(0) + \Delta^2 u^0 -
      2\Delta u^0\cdot\nabla u^0
      -f\cdot\nabla u^0\\
			&\qquad - \Delta \nabla p^0 + \nabla p^0\cdot\nabla u^0
    \end{aligned}
		\right.
  \end{align}
	on $\partial \Omega$. We write
  \[
  \Cas{ns1}:=
	C ( \nTS{u^0}_{\sV5}
  ,\nTS{f}_{C^0\sH3}
  ,\nTS{f}_{C^1\sH1}
	,\nTS{f}_{\sH2\sL2} )
  \]
	where $C(\cdots)$ is some function independent of the data and, by abuse of notation, may change with each occurrence of $\Cas{ns1}$.
\end{assumption}
In~\cite[theorem 3.1 and theorem
  3.2]{Temam1982} it is shown that exactly under these conditions
there exists a solution to the Navier-Stokes equations that satisfies
the bound
\begin{equation}\label{reg:ns}
	\nTS{u}_{\sC0 \sV5} + \nTS{u}_{\sC1 \sV3} + \nTS{u}_{\sC2 \sV1} + \nTS{u}_{\sH2\sV2} + \nTS{p}_{\sH2\sQ1} \le
  \Cas{ns1}. 
\end{equation}
Like in the case of the Stokes equations the compatibility
condition requires the solution of an
overdetermined problem. Again we assume the validity of
assumption~\ref{assumption:ns1} throughout this section. 
In section~\ref{sec:low-reg:ns} we will introduce a
smoothing version of all estimates that will allow us to derive
optimal order error bounds without relying on such non-local
compatibility conditions.

We call functions $u_k \in \scG1(\grid_k, \sV1)$ and $p_k \in \sdG0(\grid_k, \sQ0)$ a solution to the time-discrete Navier-Stokes equations if there holds
\begin{equation}\label{eq:ns:disc}
	\partial_t u_k - \Delta u_k + (\ak u_k) \cdot \nabla (\ak u_k) + \nabla p_k = f \quad\text{in $\sdG0(\grid_k, \sHz1(\Omega))'$}
\end{equation}
with $u_k(0) = u^0$. As a classical timestepping scheme this is equivalent to
\begin{equation*}
	u_k^n - u^{n-1}_k - \tfrac{k_n}{2} \Delta (u_k^n + u_k^{n-1}) + \tfrac{k_n}{4} (u_k^n + u_k^{n-1}) \cdot \nabla (u_k^n + u_k^{n-1}) + k_n \nabla p_k^n  = \int_{I^n} f \dif t,
\end{equation*}
for $n = 1, \ldots, N$ with $u_k^0 \coloneqq u^0$. 

\begin{remark}
	The quadratic form of the nonlinearity makes the approximation
	$(\ak u_k) \cdot \nabla (\ak u_k)$ instead of $u_k \cdot
	\nabla u_k$ feasible. This not only simplifies numerical quadrature, but also the analysis since some terms will cancel.
\end{remark}

Before we prove a priori estimates we collect some technical results, sometimes in greater generality than needed for this section.

\begin{lemma}[Gronwall's inequality]\label{th:gronwall}
	Let $\alpha_n > 0$, $\beta_n \ge 0$, $\gamma_n \ge 0$, $\delta \ge 0$ and $x_n \ge 0$ for $n = 0, \ldots, N$. If $x_n$ satisfies
	\[ \alpha_n x_n + \beta_n \le \sum_{k = 0}^{n-1} \gamma_k x_k + \delta \qquad \text{for $n = 0, \ldots, N$} \] 
	then there holds
	\[ \alpha_n x_n + \beta_n \le \delta \euler^{\sum_{k = 0}^{n-1} \frac{\gamma_k}{\alpha_k}} \qquad \text{for $n = 0, \ldots, N$}. \]
\end{lemma}

\begin{lemma}\label{th:ns:nonlin-props}
	Let $u, v \in \sV1$ with additional regularity if necessary.
	\begin{enumerate}
		\item 
		For $s = -3, \ldots, 0$ there holds
			\begin{subequations}\label{eq:nl:sl1}
				\begin{align}
					\nS{u \cdot \nabla v}_{\sH{s}} &\leq C \nS{u}_{\sL2} \nS{v}_{\sV{s+3}}, \label{eq:nl:1}\\
					\nS{u \cdot \nabla v}_{\sH{s}} &\leq C \nS{u}_{\sV2} \nS{v}_{\sV{s+1}}, \label{eq:nl:2}\\
					\nS{u \cdot \nabla v}_{\sH{s}} &\leq C \nS{u}_{\sV{s+1}} \nS{v}_{\sV2}, \label{eq:nl:3}
				\end{align}
			\end{subequations}
			where \eqref{eq:nl:2} and \eqref{eq:nl:3} can be sharpened for $s = 0$ to
			\begin{subequations}\label{eq:nl:s0}
				\begin{align}
					\nS{u \cdot \nabla v}_{\sL2} &\leq C \nS{\nabla u}_{\sL2}^{\frac 1 2} \nS{\Delta u}_{\sL2}^{\frac 12} \nS{v}_{\sV1}, \label{eq:nl:2b}\\
					\nS{u \cdot \nabla v}_{\sL2} &\leq C \nS{u}_{\sV1} \nS{\nabla v}_{\sL2}^{\frac 1 2} \nS{\Delta v}_{\sL2}^{\frac 12}. \label{eq:nl:3b}
				\end{align}
			\end{subequations}

		\item
			For $s \geq 1$ we have
			\begin{equation}\label{eq:nl:sge1}
				\nS{\nabla^s (u \cdot \nabla v)}_{\sL2} \leq C \sum_{i = 1}^s \nS{u}_{\sV{i+1}} \nS{v}_{\sV{s-i+2}} = C \sum_{i = 1}^s \nS{v}_{\sV{i+1}} \nS{u}_{\sV{s-i+2}}.
			\end{equation}
	\end{enumerate}	
\end{lemma}
\begin{proof}
	We use the embeddings $\sH1(\Omega) \hookrightarrow \sL6(\Omega) \hookrightarrow \sL3(\Omega)$ and $\sH2(\Omega) \hookrightarrow \sL\infty(\Omega)$.

	\begin{enumerate}
		\item[$s = 0$:]
			With the Gagliardo-Nirenberg inequality $\norm{u}_{\sL3} \leq C \norm{u}_{\sL2}^{\frac 12} \norm{\nabla u}_{\sL2}^{\frac 12}$ for $u \in \sH1(\Omega)$ and Agmon's inequality $\norm{u}_{\sL\infty} \leq C \norm{\nabla u}_{\sL2}^{\frac 12} \norm{\Delta u}_{\sL2}^{\frac 12}$ for $u \in \sV2$ we get \eqref{eq:nl:1}, \eqref{eq:nl:2b} and \eqref{eq:nl:3b}:
			\begin{alignat*}{3}
				\nS{u \cdot \nabla v}_{\sL2} &\leq C \nS{u}_{\sL2} \nS{\nabla v}_{\sL\infty} &&\leq C \nS{u}_{\sL2} \nS{v}_{\sV3},\\
				\nS{u \cdot \nabla v}_{\sL2} &\leq C \nS{u}_{\sL\infty} \nS{\nabla v}_{\sL2} &&\leq C \nS{\nabla u}_{\sL2}^{\frac 12} \nS{\Delta u}_{\sL2}^{\frac 12} \nS{v}_{\sV1},\\ 
				\nS{u \cdot \nabla v}_{\sL2} &\leq C \nS{u}_{\sL6} \nS{\nabla v}_{\sL3} &&\leq C \nS{u}_{\sV1} \nS{\nabla v}_{\sL2}^{\frac 12} \nS{\Delta v}_{\sL2}^{\frac 12}.
			\end{alignat*}

		\item[$s= -1$:]
			Let $w \in \sHz1(\Omega)$. Then $(u \cdot \nabla v, w) = -(u \otimes v, \nabla w)$ and thus \eqref{eq:nl:2} since
			\[ \nS{u \cdot \nabla v}_{\sH{-1}} \leq C \nS{u \otimes v}_{\sL2} \leq C \nS{u}_{\sL\infty} \nS{v}_{\sL2} \leq C \nS{u}_{\sV2} \nS{v}_{\sL2}. \]
			\eqref{eq:nl:3} follows by using $\nS{u \otimes v}_{\sL2} \leq C \nS{u}_{\sL2} \nS{v}_{\sL\infty}$ instead. \eqref{eq:nl:3} implies \eqref{eq:nl:1}. 
	
		\item[$s=-2$:]
			Let $w \in \sHz1(\Omega) \cap \sH2(\Omega)$. Then \eqref{eq:nl:1} follows from 
			\begin{equation*}
				(u \cdot \nabla v, w) \leq \nS{u \cdot \nabla v}_{\sL1} \nS{w}_{\sL\infty} \leq \nS{u}_{\sL2} \nS{\nabla v}_{\sL2} \nS{w}_{\sH2}.
			\end{equation*}
			For \eqref{eq:nl:2} we estimate
			\begin{equation}\label{eq:th:ns:nonlin-props:tmp}
				(u \cdot \nabla v, w) = - (u \cdot \nabla w, v) \leq C \nS{u \cdot \nabla w}_{\sH1} \nS{v}_{\sV{-1}}
			\end{equation}
			and \eqref{eq:nl:2} follows from $\nS{u \cdot \nabla w}_{\sL2} \leq C \nS{u}_{\sL\infty} \nS{\nabla w}_{\sL2}$ together with
			\begin{equation*}
				\nS{\nabla (u \cdot \nabla w)}_{\sL2} \leq C \left( \nS{\nabla u}_{\sL3} \nS{\nabla w}_{\sL6} + \nS{u}_{\sL\infty} \nS{\nabla^2 w}_{\sL2} \right) \leq C \nS{u}_{\sV2} \nS{w}_{\sH2}.
			\end{equation*}
			If in \eqref{eq:th:ns:nonlin-props:tmp} we use instead $\spL{u \cdot \nabla w, v} = \spL{u, \nabla w v} \leq \nS{u}_{\sV{-1}} \nS{\nabla w v}_{\sH1}$, similar arguments yield \eqref{eq:nl:3}.

		\item[$s = -3$:] 
			Let $w \in \sHz1(\Omega) \cap \sH3(\Omega)$. Then \eqref{eq:nl:1} follows from 
			\[ (u \cdot \nabla v, w) = - (u \otimes v, \nabla w) \leq \nS{u \otimes v}_{\sL1} \nS{\nabla w}_{\sL\infty} \leq C \nS{u}_{\sL2} \nS{v}_{\sL2} \nS{w}_{\sH3}. \] 
			For \eqref{eq:nl:2} we have $(u \cdot \nabla v, w) \leq C \nS{u \cdot \nabla w}_{\sH2} \nS{v}_{\sV{-2}}$ similarly as in \eqref{eq:th:ns:nonlin-props:tmp}. Using the estimates for $\norm{u \cdot \nabla w}_{\sL2}$ and $\norm{\nabla (u \cdot \nabla w)}_{\sL2}$ from $s = -2$ and
			\begin{align*}
				\MoveEqLeft\nS{\nabla^2 (u \cdot \nabla w)}_{\sL2}\\
				&\leq C \left( \nS{\nabla^2 u}_{\sL2} \nS{\nabla w}_{\sL\infty} + \nS{\nabla u}_{\sL3} \nS{\nabla^2 w}_{\sL6} + \nS{u}_{\sL\infty} \nS{\nabla^3 w}_{\sL2} \right)\\
				&\leq C \nS{u}_{\sV2} \nS{w}_{\sH3}
			\end{align*}
			we arrive at \eqref{eq:nl:2}. The estimate \eqref{eq:nl:3} follows by a similar argument.
			
		\item[$s \geq 1$:] We have by the product rule
			\[ \nS{\nabla^s (u \cdot \nabla v)} \leq C \sum_{i = 0}^s \nS{\nabla^{i} u}_{\sL{p_i}} \nS{\nabla^{s - i +1} v}_{\sL{q_i}} \] 
			with $\frac{1}{p_i} + \frac{1}{q_i} = \frac 1 2$ for $i = 0, \ldots, s$. For $i > 0$ we set $p_i = q_i = 4$ and use $\sH1(\Omega) \hookrightarrow \sL4(\Omega)$, for $i = 0$ we set $p_0 = \infty$, $q_0 = 2$ and use $\sH2(\Omega) \hookrightarrow \sL\infty(\Omega)$. This yields
			\[ \nS{\nabla^s (u \cdot \nabla v)} \leq C \sum_{i = 1}^s \nS{u}_{\sV{i+1}} \nS{v}_{\sV{s-i+2}} + C \nS{u}_{\sV2} \nS{v}_{\sV{s+1}} \] 
			and since the last term corresponds to $i = 1$, \eqref{eq:nl:sge1} follows.\qedhere
	\end{enumerate}
\end{proof}

\begin{lemma}
	With the operator
	\[ \oCdiff(u, v) \coloneqq (u - v) \cdot \nabla(u - v) + (u - v) \cdot \nabla v + v \cdot \nabla (u - v). \]
	we have the velocity error identity
	\begin{equation}\label{eq:ns:error-u}
		\begin{aligned}
			\MoveEqLeft\partial_t(u_k - \ik u) - P \Delta \ak (u_k - \ik u) \\
			&= -P \Delta \ak(u - \ik u) + \ak P \oCdiff(u, \ak \ik u) - P \oCdiff(\ak u_k, \ak \ik u)
		\end{aligned}
	\end{equation}
	and the pressure error identity
	\begin{equation}\label{eq:ns:error-p}
		\begin{aligned}
			\nabla(p_k - \ak p) = (P - \Id) \left( \Delta \ak (u - u_k) - \ak \oCdiff(u, \ak \ik u) + \oCdiff(\ak u_k, \ak \ik u) \right).
		\end{aligned}
	\end{equation}
\end{lemma}
\begin{proof}
	We proceed just as in the linear case, lemma~\ref{th:st:error-id}. The only exception is the nonlinear term in the Galerkin orthogonality, for which we use the identity
	\[ u \cdot \nabla u - \ak u_k \cdot \nabla \ak u_k = \oCdiff(u, \ak \ik u) - \oCdiff(\ak u_k, \ak \ik u) \]	
	which follows by elementary calculations.
\end{proof}

\begin{lemma}\label{th:ns:nonlin-approx-error}
	Under assumption~\ref{assumption:ns1} there holds
	\begin{align*}
		\nTS{\ak \oCdiff(u, \ak \ik u)}_{\sL2\sH{-1}} \leq \Cas{ns1} k^2, \quad \nTS{\ak\oCdiff(u, \ak \ik u)} \leq \Cas{ns1} k^2.
	\end{align*}
\end{lemma}
\begin{proof}
	By definition of $\oCdiff$ and the properties of $\ak$ we have that
	\begin{equation}\label{eq:ak-nonlin}
		\begin{aligned}
			\ak \oCdiff(u, \ak \ik u) &= \ak \left( (u - \ak \ik u) \cdot \nabla (u - \ak \ik u) \right) + \ak (u - \ik u) \cdot \nabla \ak \ik u \\
			&\qquad+ \ak \ik u \cdot \nabla \ak (u - \ik u).
		\end{aligned}
	\end{equation}
	Together with \eqref{eq:nl:2}, \eqref{eq:nl:3} and the results from lemmata \ref{th:nodal-int-properties} and~\ref{th:nodal-midpoint-estimates} this yields
	\begin{align*}
		\MoveEqLeft\nTS{\ak \oCdiff(u, \ak \ik u)}_{\sL2 \sH{s}} \\
		&\leq C \left( \nTS{u - \ak \ik u}_{\sL\infty \sV{s+1}} \nTS{u - \ak \ik u}_{\sL2 \sV{2}} + \nTS{u}_{\sC0 \sV{2}} \nTS{\ak (u - \ik u)}_{\sL2\sV{s+1}} \right)\\
		&\leq C k^2 \left( \nTS{\partial_t u}_{\sL\infty \sV{s+1}} \nTS{\partial_t u}_{\sL2\sV{2}} + \nTS{u}_{\sC0 \sV{2}} \nTS{\partial_{tt} u}_{\sL2\sV{s+1}} \right)
	\end{align*}
	if $s \in \{ -1, 0 \}$. The claim follows since by \eqref{reg:ns} the right side is bounded by $\Cas{ns1} k^2$.
\end{proof}

\begin{lemma}\label{th:ns:approx-error-h1}
	Under assumption~\ref{assumption:ns1} there exists $C_0 > 0$ such that there holds
	\[ \nTS{\ik u - u_k}_{\sL\infty \sL2} + \nHH{\ak (\ik u - u_k)}_{\sL2\sV1} \le \Cas{ns1} k^2 \]
	if the stepsize condition $k < C_0 \nTS{\nabla u}_{\sL\infty\sL2}^{-1} \nTS{\Delta u}_{\sL\infty\sL2}^{-1}$ is satisfied.
\end{lemma}
\begin{proof}
	Just as in the proof of linear stability, lemma~\ref{th:st:disc-stab} for $s = 0$, we test \eqref{eq:ns:error-u} with $v_k \indicatorfunc_J$ where $v_k \coloneqq \ik u - u_k$ and $J \coloneqq (0, t_n]$ for $n = 1, \dots, N$. This implies
	\begin{align*}
		\MoveEqLeft[1]\tfrac 1 2 \nS{v_k^n}_{\sL2}^2 + \nTS{\ak v_k}_{\sLJ2\sV1}^2 \\
		&= - \spLL{\nabla \ak(u - \ik u), \nabla \ak v_k}_J + \spLL{\ak P \oCdiff(u, \ak \ik u), v_k}_J - \spLL{P \oCdiff(\ak u_k, \ak \ik u), v_k}_J\\
		&\leq C \left(\nTS{u - \ik u}_{\sLJ2 \sV1} + \nTS{\ak P \oCdiff(u, \ak \ik u)}_{\sLJ2 \sV{-1}} \right) \nTS{\ak v_k}_{\sLJ2 \sV1} \\
		&\qquad - \spLL{P \oCdiff(\ak u_k, \ak \ik u), v_k}_J.
	\end{align*}
	Lemma~\ref{th:nodal-int-properties} for the interpolation error and lemma~\ref{th:ns:nonlin-approx-error} for $\oCdiff(u, \ak \ik u)$ imply
	\begin{equation*} 
		\nTS{P \Delta \ak(u - \ik u)}_{\sL2\sV{-1}} \le \Cas{ns1} k^2, \quad \nTS{\ak P \oCdiff(u, \ak \ik u)}_{\sL2 \sV{-1}} \le \Cas{ns1} k^2.
	\end{equation*}
	and hence by Young's inequality
	\begin{equation}\label{eq:th:ns:approx-error-h1:tmp}
		\tfrac 1 2 \nS{v_k^n}_{\sL2}^2 + \tfrac 1 2 \nTS{\ak v_k}_{\sLJ2\sV1}^2 \leq \Cas{ns1} k^2  - \spLL{P \oCdiff(\ak u_k, \ak \ik u), v_k}_J.
	\end{equation}
	For the last term we use the antisymmetry of the nonlinearity, hence implicitly its time discretization, and \eqref{eq:nl:2} with $s = 0$ to estimate 
	\begin{align}\label{eq:th:ns:approx-error-h1:tmp2}
		\MoveEqLeft\abs{\spLL{P \oCdiff(\ak u_k, \ak \ik u), v_k}_J} = \abs{\spLL{\ak v_k \cdot \nabla \ak \ik u, \ak v_k}_J}\\ 
		&\le C \nTS{\nabla u}_{\sL\infty \sL2}^\ahalf \nTS{\Delta u}_{\sL\infty \sL2}^\ahalf \nTS{\ak v_k}_{\sLJ2\sV1} \nTS{\ak v_k}_{\sLJ2\sL2}\nonumber\\
		&\le C \nTS{\nabla u}_{\sL\infty \sL2} \nTS{\Delta u}_{\sL\infty \sL2} \nTS{\ak v_k}_J^2 + \tfrac 14 \nTS{\ak v_k}^2_{\sLJ2\sV1}.\nonumber
	\end{align}
	The last term will be moved to the left of \eqref{eq:th:ns:approx-error-h1:tmp}. Since
	\[
		\begin{aligned}
			\nTS{\ak v_k}_J^2 = \frac 1 4 \sum_{j = 1}^n k_j\nS{v_k^j + v_k^{j-1}}_{\sL2}^2 \le C \sum_{j = 1}^n k_j \nS{v_k^j}_{\sL2}^2
		\end{aligned}
	\]
	we can use \eqref{eq:th:ns:approx-error-h1:tmp2} in \eqref{eq:th:ns:approx-error-h1:tmp} and conclude the proof with Gronwall's inequality from lemma~\ref{th:gronwall} if $k < C_0 \nTS{\nabla u}_{\sL\infty\sL2}^{-1} \nTS{\Delta u}_{\sL\infty\sL2}^{-1}$ for some $C_0 > 0$ independent of the data. 
\end{proof}

\begin{lemma}\label{th:ns:approx-error-l2}
	Under assumption~\ref{assumption:ns1} there exists $C_0 > 0$ such that there holds
	\begin{equation}\label{eq:ns:approx-error-l2}
		\nTS{\ik u - u_k}_{\sL\infty \sV1} + \nTS{\ak (\ik u - u_k)}_{\sL2\sV2} \le \Cas{ns1} k^2
	\end{equation}
	if the stepsize condition $k < C_0 \nTS{\nabla u}_{\sL\infty\sL2}^{-1} \nTS{\Delta u}_{\sL\infty\sL2}^{-1}$ is satisfied.
\end{lemma}
\begin{proof}
	We use the stability estimate \eqref{eq:disc-stab:div} with $s = 1$ for the error equation \eqref{eq:ns:error-u}: 
	\begin{equation}\label{eq:th:ns:approx-error-l2:tmp0}
		\begin{aligned}
			\MoveEqLeft\nTS{\ik u - u_k}_{\sL\infty \sV1} + \nTS{\ak (\ik u - u_k)}_{\sL2\sV2} \\
			&\le C \left( \nTS{P \Delta \ak (u - \ik u)} + \nTS{\ak P \oCdiff(u, \ak \ik u)} + \nTS{P \oCdiff(\ak u_k, \ak \ik u)} \right)
		\end{aligned}
	\end{equation}
	For the first and second term on the right we have by lemmata~\ref{th:nodal-int-properties} and \ref{th:ns:nonlin-approx-error}:
	\begin{equation}\label{eq:th:ns:approx-error-l2:tmp1}
		\nTS{P \Delta \ak (u - \ik u)} \le \Cas{ns1} k^2, \quad \nTS{\ak P \oCdiff(u, \ak \ik u)} \leq \Cas{ns1} k^2.
	\end{equation}
	For the third term on the right we abbreviate $v_k \coloneqq \ik u - u_k$ and have by \eqref{eq:nl:2b}, \eqref{eq:nl:3b}:
	\begin{align*}
		\nTS{\oCdiff(\ak u_k, \ak \ik u)} &\leq C \left( \nTS{\ak v_k}_{\sL\infty \sV1} \nTS{\ak v_k}^{\frac 12}_{\sL2 \sV1} \nTS{\ak v_k}^{\frac 12}_{\sL2 \sV2} + \nTS{u}_{\sC0 \sV2} \nTS{\ak v_k}_{\sL2 \sV1} \right)\\
		&\leq C(\delta) \nTS{\ak v_k}_{\sL2 \sV1} \left( \nTS{\ak v_k}^2_{\sL\infty \sV1} + \nTS{u}_{\sC0 \sV2} \right) + \delta \nTS{\ak v_k}_{\sL2 \sV2}
	\end{align*}
	for arbitrary $\delta > 0$. By lemma~\ref{th:ns:approx-error-h1} we have $\nTS{\ak v_k}_{\sL2 \sV1} \leq \Cas{ns1} k^2$ and with the inverse inequality also $\nTS{\ak v_k}_{\sL\infty \sV1} \leq C k^{-\frac12} \nTS{\ak v_k}_{\sL2 \sV1} \leq \Cas{ns1}$, hence
	\begin{equation}\label{eq:th:ns:approx-error-l2:nonlin}
		\nTS{\oCdiff(\ak u_k, \ak \ik u)} \leq \Cas{ns1}(\delta) k^2 + \delta \nTS{\ak v_k}_{\sL2 \sV2}.
	\end{equation}
	Using \eqref{eq:th:ns:approx-error-l2:tmp1} and \eqref{eq:th:ns:approx-error-l2:nonlin} in the stability estimate we arrive at \eqref{eq:ns:approx-error-l2} after choosing $\delta$ small enough such that $\nTS{\ak v_k}_{\sL2 \sV2}$ in \eqref{eq:th:ns:approx-error-l2:nonlin} can be moved to the left-hand side of \eqref{eq:th:ns:approx-error-l2:tmp0}.
\end{proof}

\begin{theorem}\label{th:ns:error-l2}
	Let assumption~\ref{assumption:ns1} hold. Then there exists $C_0 > 0$ such that the Crank-Nicolson time discretization \eqref{eq:ns:disc} of the Navier-Stokes equations \eqref{eq:ns} satisfies the a priori bound
  \[ \nTS{u - u_k}_{\sL\infty \sV1} + \nTS{\ak (u - u_k)}_{\sL2\sV2} +
  \nTS{\ak p - p_k}_{\sL2 \sQ1} \le  \Cas{ns1} k^2 \]
	if the stepsize condition $k < C_0 \nTS{\nabla u}_{\sL\infty\sL2}^{-1} \nTS{\Delta u}_{\sL\infty\sL2}^{-1}$ is satisfied.
\end{theorem}
\begin{proof}
	With $u - u_k = (u - \ik u) + (\ik u - u_k)$ the estimates for the velocity follow by lemma~\ref{th:ns:approx-error-l2} and the interpolation error estimate from lemma~\ref{th:nodal-int-properties}. To estimate the pressure error we use \eqref{eq:ns:error-p} and Poincar\'e's inequality to get
	\begin{equation*}
		\nTS{\ak p - p_k}_{\sL2 \sQ1} \leq C \left( \nTS{\ak (u - u_k)}_{\sL2 \sV2} + \nTS{\ak \oCdiff(u, \ak \ik u)} + \nTS{\oCdiff(\ak u_k, \ak \ik u)} \right).
	\end{equation*}
	The $k^2$-bound follows for the first term using the established velocity estimate and for the second by lemma~\ref{th:ns:nonlin-approx-error}. For the third term we use \eqref{eq:th:ns:approx-error-l2:nonlin} from the proof of lemma~\ref{th:ns:approx-error-l2}:
	\begin{equation}\label{eq:th:ns:error-l2:nonlin}
		\nTS{\oCdiff(\ak u_k, \ak \ik u)} \leq \Cas{ns1} k^2 + C \nTS{\ak (u_k - \ik u)}_{\sL2 \sV2} \leq \Cas{ns1} k^2.
	\end{equation}
	This concludes the proof that $\nTS{\ak p - p_k}_{\sL2 \sQ1} \leq \Cas{ns1} k^2$.
\end{proof}

We extend the assumptions on the data to obtain
$\sL\infty$-estimates for the pressure.
\begin{assumption}\label{assumption:ns2}
	We assume $\Omega$ has a $C^6$-boundary, the right-hand side has regularity
  $f\in
  C^0(\bar I, \sH4(\Omega))\cap C^1(\bar I, \sH2(\Omega))\cap
  C^2(\bar I, \sL2(\Omega))\cap \sH3(I, \sV{-1})$ and the initial
	 data $u^0 \in \sV6$ satisfies the compatibility conditions \eqref{comp:ns1} and \eqref{comp:ns2}. We write
  \[
  \Cas{ns2} \coloneqq
	C( \nTS{u^0}_{\sV6}, \nTS{f}_{C^0\sH4}, \nTS{f}_{C^1\sH2}, \nTS{f}_{C^2\sL2}, \nTS{f}_{\sH3\sV{-1}} )
  \]
	where $C(\cdots) > 0$ is a function independent of the data and, by abuse of notation, may change with each occurrence of $\Cas{ns2}$.
\end{assumption}
It is shown in~\cite[theorem 3.1 and theorem
  3.2]{Temam1982} that under exactly these conditions the solution
satisfies
\begin{equation}\label{reg:ns:higher}
	\nTS{u}_{\sC0 \sV6} + \nTS{u}_{\sC1 \sV4} + \nTS{u}_{\sC2\sV2} + \nTS{u}_{\sH2\sV3} 
	+ \nTS{p}_{\sC2\sQ1} + \nTS{p}_{\sH2 \sQ2}
\le \Cas{ns2}. 
\end{equation}

\begin{theorem}\label{th:ns:error-linf}
	Let assumption~\ref{assumption:ns2} hold. Then there exists $C_0 > 0$ such that the Crank-Nicolson time discretization \eqref{eq:ns:disc} of the Navier-Stokes equations \eqref{eq:ns} satisfies the a priori error estimate
  \[ \nTS{u - u_k}_{\sL\infty \sV2} + \nTS{\ak(u - u_k)}_{\sL2 \sV3} +
  \nTS{\mk p - p_k}_{\sL\infty \sQ1} \leq \Cas{ns2} k^2, \]
	if the stepsize condition $k < C_0 \nTS{\nabla u}_{\sL\infty\sL2}^{-1} \nTS{\Delta u}_{\sL\infty\sL2}^{-1}$ is satisfied.
	In particular, the discrete pressure $p_k$ must be compared to the continuous pressure $p$ evaluated at interval midpoints for second order convergence.
\end{theorem}
\begin{proof}
	We again split $u - u_k = (u - \ik u) + (\ik u - u_k)$. By lemma~\ref{th:nodal-int-properties} there holds
	\begin{equation*}
		\nTS{u - \ik u}_{\sL\infty \sV2} + \nTS{\ak(u - \ik u)}_{\sL2 \sV3} \leq C k^2 \left( \nTS{\partial_{tt} u}_{\sL\infty \sV2} + \nTS{\partial_{tt} u}_{\sL2 \sV3} \right) \leq \Cas{ns2} k^2.
	\end{equation*}
	For $\ik u - u_k$ we use the stability estimate \eqref{eq:disc-stab:div} from lemma~\ref{th:st:disc-stab} with $s = 2$ for \eqref{eq:ns:error-u}. This implies, abbreviating $v_k \coloneqq \ik u - u_k$:
	\begin{equation}\label{eq:th:ns:error-linf:tmpA}
		\begin{aligned}
			\nTS{v_k}_{\sL\infty \sV2} + \nTS{\ak v_k}_{\sL2\sV3} \leq C \bigl( &\nTS{\ak (u - \ik u)}_{\sL2 \sV3} + \nTS{\ak \oCdiff(u, \ak \ik u)}_{\sL2 \sH1} \\
			&\quad+ \nTS{\oCdiff(\ak u_k, \ak \ik u)}_{\sL2 \sH1} \bigr).
		\end{aligned}
	\end{equation}
	The first term on the right can be estimated as above. For the nonlinear terms we only have to derive estimates for the gradients due to lemma~\ref{th:ns:nonlin-approx-error} and \eqref{eq:th:ns:error-l2:nonlin}. Combining \eqref{eq:nl:sge1} and the representation \eqref{eq:ak-nonlin} of $\ak \oCdiff(u, \ak \ik u)$ we get
	\begin{align*}
		\MoveEqLeft\norm{\nabla \ak \oCdiff(u, \ak \ik u)} \\
		&\leq C \left( \nTS{u - \ak \ik u}_{\sL\infty \sV2} \nTS{u - \ak \ik u}_{\sL2 \sV2} + \nTS{\ak \ik u}_{\sL\infty \sV2} \nTS{\ak (u- \ik u)}_{\sL2 \sV2} \right)\\
		&\leq C k^2 \left( \nTS{\partial_t u}_{\sL\infty \sV2} \nTS{\partial_t u}_{\sL2 \sV2} + \nTS{u}_{\sC0 \sV2} \nTS{\partial_{tt} u}_{\sL2 \sV2} \right) \leq \Cas{ns2} k^2.
	\end{align*}
	For the second nonlinear term in \eqref{eq:th:ns:error-linf:tmpA} we proceed similarly, yielding
	\[ \nTS{\nabla \oCdiff(\ak u_k, \ak \ik u)} \leq C \nTS{\ak v_k}_{\sL2 \sV2} \left( \nTS{\ak v_k}_{\sL\infty \sV2} + \nTS{u}_{\sC0 \sV2} \right) \leq \Cas{ns2} k^2 \]
	using that $\nTS{\ak v_k}_{\sL2 \sV2} \leq \Cas{ns2} k^2$ by lemma~\ref{th:ns:approx-error-l2} and, by the inverse inequality, also $\nTS{\ak v_k}_{\sL\infty \sV2} \leq \Cas{ns2}$. Inserting the previous estimates into \eqref{eq:th:ns:error-linf:tmpA} we arrive at
	\begin{equation*}
		\nTS{v_k}_{\sL\infty \sV2} + \nTS{\ak v_k}_{\sL2\sV3} \leq \Cas{ns2} k^2
	\end{equation*}
	which together with the estimates for $u - \ik u$ concludes the proof for the velocity errors. For the pressure error we get from the error identity~\eqref{eq:ns:error-p} that
	\begin{equation}\label{eq:th:ns:error-linf:tmp1}
		\begin{aligned}
			\MoveEqLeft\nTS{\ak p - p_k}_{\sL\infty \sQ1} \\
			&\leq C \left( \nTS{u - u_k}_{\sL\infty \sV2} + \nTS{\oCdiff(u, \ak \ik u)}_{\sL\infty \sL2} + \nTS{\oCdiff(\ak u_k, \ak \ik u)}_{\sL\infty \sL2} \right).
		\end{aligned}
	\end{equation}
	For the first term on the right we use the established velocity estimate. For the second term on the right of \eqref{eq:th:ns:error-linf:tmp1} we repeat the arguments from the proof of lemma~\ref{th:ns:nonlin-approx-error} with $\sL\infty$-in-time estimates:
	\[ \nTS{\oCdiff(u, \ak \ik u)}_{\sL\infty\sL2} \leq C k^2 \left ( \nTS{\partial_t u}_{\sL\infty\sV1} \nTS{\partial_t u}_{\sL\infty\sV2} + \nTS{\partial_{tt} u}_{\sL\infty\sV1} \nTS{u}_{\sC0\sV2} \right). \]
	Using lemma~\ref{th:ns:nonlin-props} we get for the third term on the right of \eqref{eq:th:ns:error-linf:tmp1}:
	\[ \nTS{\oCdiff(\ak u_k, \ak \ik u)}_{\sL\infty \sL2} \leq C \nTS{\ak v_k}_{\sL\infty \sV1} \left( \nTS{\ak v_k}_{\sL\infty \sV2} + \nTS{u}_{\sC0 \sV2} \right) \leq \Cas{ns2} k^2 \]
	using $\nTS{v_k}_{\sL\infty \sV1} \leq \Cas{ns2} k^2$ by theorem~\ref{th:ns:error-l2} and $\nTS{\ak v_k}_{\sL\infty \sV2} \leq \Cas{ns2}$ by the inverse inequality. Combining these estimates for the right-hand side of \eqref{eq:th:ns:error-linf:tmp1} with lemma~\ref{th:midpoint-quadratic-estimate} we can hence conclude the proof since
	\begin{equation*}
		\nTS{\mk p - p_k}_{\sL\infty \sQ1} \leq \nTS{\ak p - \mk p}_{\sL\infty \sQ1} + \nTS{\ak p - p_k}_{\sL\infty \sQ1} \leq \Cas{ns2} k^2.\qedhere
	\end{equation*}
\end{proof}

\section{Estimates with Low Regularity} \label{sec:low-reg}

The previous estimates require a priori bounds which can 
only be obtained by strong assumptions on the initial data, in particular 
nonlocal compatibility
conditions~(\ref{comp:stokes1}),~(\ref{comp:stokes2}) for the Stokes
and~(\ref{comp:ns1}),~(\ref{comp:ns2}) for the Navier Stokes equations. These
conditions are hard to verify and do not hold in general.

To avoid this requirement we derive error estimates exploiting the smoothing of the solution operator, quantified through the smoothing function $\tau\colon I \to \mathbb R$ and its discrete counterpart $\tau_k \in \sdG0(\grid_k, \mathbb R)$ defined as
\[
\tau:=\min\{t,1\}, \quad \restr{\tau_k}{I^n} \coloneqq \min\{ t_{n-1},1\} \quad\text{for $n=1,\dots,N$}.
\]
Note that in particular $\tau_k\big|_{I^1} \equiv 0$ and $\tau_k \le \tau$. We remind of remark~\ref{rm:approx-weights}: The presence of the smoothing function does not affect the error estimates for the operators $\ak$, $\ik$ and $\mk$, this will be used without mention in the following. By weighting norms with powers of the smoothing function and by introducing some few initial steps with the backward Euler scheme we will derive optimal order error estimates that do not require nonlocal compatibility conditions. 

\subsection{Stokes Equations}

Throughout this section we make\ldots

\begin{assumption}\label{as:time1+}
	In addition to assumption~\ref{as:time1} of comparable timestep sizes there also exists $\rho \geq 1$ such that $k \leq \rho k_1$. The generic constants $C > 0$ may depend on $\rho$.
\end{assumption}

For reasons which will become clear later we first perform $n_0 \in \mathbb N$ implicit Euler steps with $n_0$ only depending on the regularity assumptions. We start by formulating this timestepping scheme: Let $u_k^n \in \sV1$ and $p_k^n \in \sQ0$ solve
\begin{subequations}\label{eq:st:disc:sm}
\begin{gather}
	u_k^n - u_k^{n-1} - k_n \Delta u_k^n + k_n \nabla p_k^n = \int_{I^n} f(t) \dif t\label{eq:st:disc:sm:ie}
	\intertext{for $n = 1, \ldots, n_0$ with $u_k^0 \coloneqq u^0$. Let $\grid_k^{n_0}$ denote the time discretization $\grid_k$ starting at $t_{n_0}$. Then $u_k \in \scG1(\grid_k^{n_0}, \sV1)$ and $p_k \in \sdG0(\grid_k^{n_0}, \sQ0)$ must satisfy}
	\partial_t u_k - \Delta u_k + \nabla p_k = f \quad \text{in $\sdG0(\grid_k^{n_0}, \sHz1(\Omega))'$}
\end{gather}
\end{subequations}
with initial value $u_k(t_{n_0}) = u_k^{n_0}$ from the last implicit Euler step. The results for the Crank-Nicolson scheme from section~\ref{sec:high-reg} carry over to this time-shifted variant. We omit a formulation of this scheme in $\sdG{}$-spaces since we are only interested in pointwise-in-time results. 

In contrast to section~\ref{sec:high-reg} we combine the assumptions required for $\sL2$- and $\sL\infty$-in-time pressure a priori estimates into one single assumption:

\begin{assumption}\label{assumption:st:low}
	We assume that $u^0 \in \sV{r}$ with $r \in \set{0, 1, 2}$. Let $s_0\in\set{1,2}$ encode the regularity of the a priori estimates, in the sense that $s_0 = 1$ will yield $\sL2$-in-time and $s_0 = 2$ will yield $\sL\infty$-in-time pressure estimates. With
	\begin{equation}\label{eq:L}
		L \coloneqq 4 + s_0 - r \in \mathbb N_0
	\end{equation}
	we can formulate the remaining assumptions: We assume that $\Omega$ has a
  $C^{s_0+L+1}$-boundary and the right-hand side has the regularity
	$f \in \sC2(\bar I, \sH{s_0+L-1}(\Omega))$. We write   
  \begin{equation*}
    \Cas{st:low} \coloneqq C\left(\norm{u^0}_{\sV{r}} + \norm{f}_{\sC2\sH{s_0+L-1}}\right)
  \end{equation*}
  where $C > 0$ is independent of the data and, by abuse of notation, may change with each occurrence of $\Cas{st:low}$. If this assumption is made, all generic constants $C > 0$ may depend on $r$ and $s_0$. 
\end{assumption}

Under assumption~\ref{assumption:st:low} for $s\in \mathbb{Z}$ with $s_0 - L \le s \le s_0+L$ we define $\alpha \coloneqq \frac{(2 m + s - r)^+}{2}$ for $m \in \set{0, 1, 2}$. The solution to the Stokes equations then satisfies the bounds
\begin{equation}\label{reg:stokes:sm1}
  \nTS{\tau^\alpha \partial_t^m u}_{\sC0\sV{s}} + \nTS{\tau^\alpha \partial_t^m u}_{\sL2\sV{s+1}}   < \Cas{st:low}. 
\end{equation}
Given $s\geq 2$ it further holds
\begin{equation}\label{reg:stokes:sm2}
  \nTS{\tau^\alpha \partial_t^m p}_{\sC0\sQ{s-1}} + \nTS{\tau^\alpha \partial_t^m p}_{\sL2\sQ{s}} < \Cas{st:low}.
\end{equation}
This follows by similar arguments as in theorems 2.3--2.5 from \cite{HeywoodRannacher1982}. By limiting the required regularity of the initial data to $r\le 2$, we avoid nonlocal compatibility conditions. While this, most critical, assumption is removed, assumption~\ref{assumption:st:low} is not strictly weaker than those from section~\ref{sec:high-reg} due to the stronger assumptions on $f$ and possibly $\Omega$, although some of these can be weakened with a more involved analysis. 

To use these results for the discrete error we need discrete stability estimates with smoothing:

\begin{lemma}[Discrete stability with smoothing]\label{th:st:sm:step}
	Let $s \in \mathbb Z$, $n_0 \in \mathbb N$ and $J \coloneqq (t_{n_0}, T]$. Let $v_k \in \scG1(\grid_k^{n_0}, \sV{s+1})$ and $r_k \in \sdG0(\grid_k^{n_0}, \sV{s-1})$ be such that 
	\begin{equation}\label{eq:sm:disc-stab}
		\partial_t v_k - P \Delta \ak v_k = r_k.
	\end{equation}
	Then we have for $\ell \in \mathbb N$ the stability estimate
	\begin{equation}\label{eq:st:sm:step}
		\begin{aligned}
			\MoveEqLeft[0.5] \nTS{\tau_k^{\frac{\ell}{2}} v_k}_{\sLJ\infty \sV{s}} + \nTS{\tau_k^{\frac{\ell}{2}} \ak v_k}_{\sLJ2 \sV{s+1}} + \nTS{\tau_k^{\frac{\ell}{2}} \partial_t v_k}_{\sLJ2 \sV{s-1}} \\
			&\leq C\bigl(k^{\frac{\ell}{2}} \nS{v_k^{n_0}}_{\sV{s}} + \nTS{\tau_k^{\frac{\ell}{2}} r_k}_{\sLJ2 \sV{s-1}} + \nTS{\tau_k^{\frac{\ell - 1}{2}} \ak v_k}_{\sLJ2 \sV{s}} + k \nTS{\tau_k^{\frac{\ell-1}{2}} \partial_t v_k}_{\sLJ2 \sV{s}} \bigr).
		\end{aligned}
	\end{equation}
	We will refer to \eqref{eq:st:sm:step} as an \emph{estimate of regularity (level) $s$ and smoothing (level) $\ell$}.
\end{lemma}
\begin{proof}
	Let $J' \coloneqq (t_{n_0}, t_n]$ for $n = n_0+1, \ldots, N$. We test \eqref{eq:sm:disc-stab} with $\tau_k^\ell (- P \Delta)^s v_k \indicatorfunc_{J'}$ and proceed as in the proof of lemma~\ref{th:st:disc-stab}, yielding
	\begin{equation}\label{eq:st:sm:step:tmp1}
		\spLL{\partial_t v_k, \tau_k^\ell (-P\Delta)^s v_k}_{J'} + \tfrac 1 2 \nTS{\tau_k^{\frac{\ell}{2}} \ak v_k}_{\sLJp2\sV{s+1}}^2 \leq \tfrac 1 2 \nTS{\tau_k^{\frac{\ell}{2}} r_k}_{\sLJp2 \sV{s-1}}^2.
	\end{equation}
	For the first term on the left we use the piecewise constantness of $\tau_k$ to get, with a factor $2$ for convenience 
	\begin{align*}
		\MoveEqLeft2 \spLL{\partial_t v_k, \tau_k^\ell (- P \Delta)^s v_k}_{J'} = \sum_{j = n_0 + 1}^n \int_{I^j} \frac{\dif}{\dif t} \left( \restr{\tau_k^\ell}{I^j} \nS{v_k}_{\sV{s}}^2 \right) \dif t\\
		&= \sum_{j = n_0 + 1}^n \restr{\tau_k^{\ell}}{I^j} \nS{v_k^j}_{\sV{s}}^2 - \restr{\tau_k^\ell}{I^j} \nS{v_k^{j-1}}_{\sV{s}}^2\\
		&= \restr{\tau_k^\ell}{I^n} \nS{v_k^n}_{\sV{s}}^2 - \restr{\tau_k^\ell}{I^{n_0+1}} \nS{v_k^{n_0}}_{\sV{s}}^2  + \sum_{j = n_0 + 1}^{n-1} \left( \restr{\tau_k^\ell}{I^j} - \restr{\tau_k^\ell}{I^{j+1}} \right) \nS{v_k^j}_{\sV{s}}^2.
	\end{align*}
	Moving all but the first term to the right-hand side of \eqref{eq:st:sm:step:tmp1} and using $\restr{\tau_k}{I^{n_0+1}} \leq C k$, we arrive at
	\begin{equation}\label{eq:st:sm:step:tmp2}
		\begin{aligned}
			\MoveEqLeft\restr{\tau_k^\ell}{I^n} \nS{v_k^n}_{\sV{s}}^2 + \nTS{\tau_k^{\frac{\ell}{2}} \ak v_k}_{\sLJp2\sV{s+1}}^2 \\
			&\leq C k^\ell \nS{v_k^{n_0}}_{\sV{s}}^2 + \nTS{\tau_k^{\frac{\ell}{2}} r_k}_{\sLJp2 \sV{s-1}}^2 + \sum_{j = n_0 + 1}^{n-1} \left( \restr{\tau_k^\ell}{I^{j+1}} - \restr{\tau_k^\ell}{I^j} \right) \nS{v_k^j}_{\sV{s}}^2.
		\end{aligned}
	\end{equation}
	We now want to estimate the last term on the right of \eqref{eq:st:sm:step:tmp2}. In the Crank-Nicolson scheme we have no control over $\nS{v_k^j}_{\sV{s}}$ itself, forcing us to use
	\begin{equation}\label{eq:st:sm:step:tmp3}
		\nS{v_k^j}_{\sV{s}} \leq \nS{\restr{\ak v_k}{I^j}}_{\sV{s}} + \tfrac k 2 \nS{\restr{\partial_t v_k}{I^j}}_{\sV{s}}.
	\end{equation}
	For the smoothing function jump, we have by elementary calculations that
	\begin{equation}\label{eq:st:sm:step:tmp4}
		\restr{\tau_k^\ell}{I^{j+1}} - \restr{\tau_k^\ell}{I^j} = \tau^\ell(t_j) - \tau^\ell(t_{j-1}) \leq t_j^\ell - t_{j-1}^\ell \leq \ell k_j t_j^{\ell-1} \leq C k_j \restr{\tau_k^{\ell-1}}{I^j}
	\end{equation}
	for $t_{j-1} \leq 1$, but the finale estimate also holds otherwise. From \eqref{eq:st:sm:step:tmp3} and \eqref{eq:st:sm:step:tmp4} we get
	\begin{align*}
		\sum_{j = n_0 + 1}^{n-1} \left( \restr{\tau_k^\ell}{I^{j+1}} - \restr{\tau_k^\ell}{I^j} \right) \nS{v_k^j}_{\sV{s}}^2 &\leq C \sum_{j = n_0 + 1}^{n-1} k_j \restr{\tau_k^{\ell-1}}{I^j} \left( \nS{\restr{\ak v_k}{I^j}}_{\sV{s}}^2 + k^2 \nS{\restr{\partial_t v_k}{I^j}}_{\sV{s}}^2 \right)\\
		&\leq C \left( \nTS{\tau_k^{\frac{\ell-1}{2}} \ak v_k}_{\sLJp{2} \sV{s}}^2 + k^2 \nTS{\tau_k^{\frac{\ell-1}{2}} \partial_t v_k}_{\sLJp{2} \sV{s}}^2 \right)
	\end{align*}
	and hence we can conclude for \eqref{eq:st:sm:step:tmp2} that
	\begin{align*}
		\restr{\tau_k^\ell}{I^n} \nS{v_k^n}_{\sV{s}}^2 + \nTS{\tau_k^{\frac{\ell}{2}} \ak v_k}_{\sLJp2\sV{s+1}}^2 &\leq C k^\ell \nS{v_k^{n_0}}_{\sV{s}}^2 + \nTS{\tau_k^{\frac{\ell}{2}} r_k}_{\sLJp2 \sV{s-1}}^2 \\
		&\quad+ C \nTS{\tau_k^{\frac{\ell-1}{2}} \ak v_k}_{\sLJp{2} \sV{s}}^2 + C k^2 \nTS{\tau_k^{\frac{\ell-1}{2}} \partial_t v_k}_{\sLJp{2} \sV{s}}^2.
	\end{align*}
	Taking the maximum over $n = n_0 + 1, \ldots, N$ and using 
	\begin{equation*}
		\nTS{\tau_k^{\frac{\ell}{2}} v_k}^2_{\sLJ\infty \sV{s}} \leq C \max_{n = n_0, \ldots, N} \restr{\tau_k^{\ell}}{I^n} \nS{v_k^n}^2_{\sV{s}}
	\end{equation*}
	we can conclude the claimed estimate for $\nTS{\tau_k^{\frac{\ell}{2}} v_k}_{\sLJ\infty \sV{s}}$ and $\nTS{\tau_k^{\frac{\ell}{2}} \ak v_k}_{\sLJ2 \sV{s+1}}$. Since
	\begin{equation*}
		\tau_k^{\frac{\ell}{2}} \partial_t v_k = \tau_k^{\frac{\ell}{2}} r_k + \tau_k^{\frac{\ell}{2}} P \Delta \ak v_k
	\end{equation*}
	the final estimate for $\nTS{\tau_k^{\frac{\ell}{2}} \partial_t v_k}_{\sLJ{2} \sV{s-1}}$ follows from the already established one.
\end{proof}

\begin{remark}\label{rm:sm-cascade}
	We emphasize the occurrence of $\ak v_k$ and $\partial_t v_k$ on the right-hand side of \eqref{eq:st:sm:step}. These terms can be estimated again by \eqref{eq:st:sm:step} with smoothing level $\ell - 1$ and regularity level $s - 1$ (for $\ak v_k$) and $s + 1$ (for $\partial_t v_k$). The unfavorable increase in regularity to estimate $\partial_t v_k$ is balanced by the occurrence of the factor $k$ in \eqref{eq:st:sm:step}. The procedure must be repeated until $\ell = 0$, where \eqref{eq:disc-stab:div} from lemma~\ref{th:st:disc-stab} can be used. The whole process is depicted as a cascade of estimates in figure~\ref{fig:st:sm:cascade}. The origin of the cascade is the lack of an a priori estimate for $\nTS{v_k}_{\sLJp2 \sV{s}}$ in the Crank-Nicolson scheme, circumvented by \eqref{eq:st:sm:step:tmp3}. 
\end{remark}

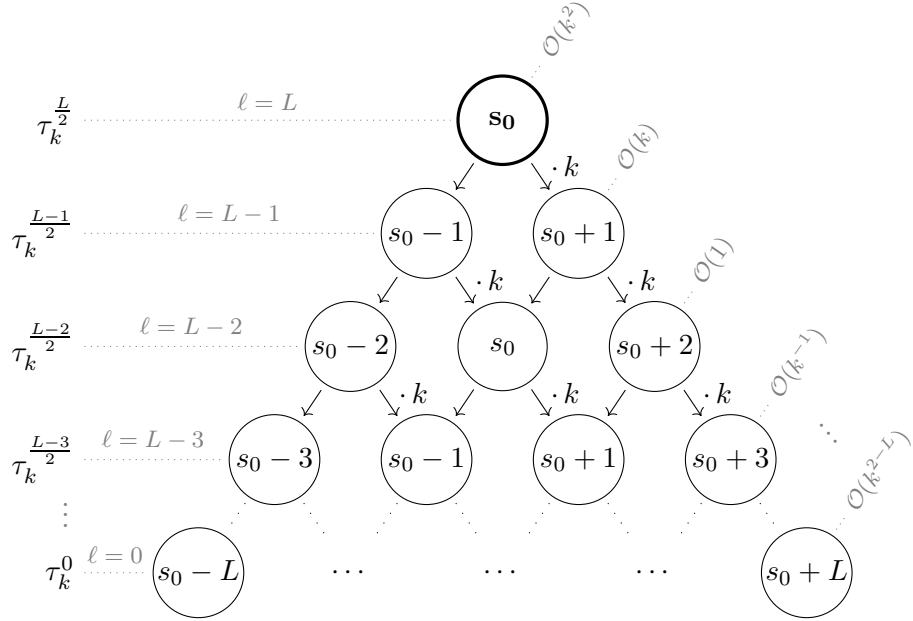
\begin{figure}
  \begin{center}
\pgfdeclarelayer{edgelayer}
\pgfdeclarelayer{nodelayer}
\pgfsetlayers{edgelayer,nodelayer,main}

\tikzstyle{none}=[inner sep=0pt]
\tikzstyle{smoothnode}=[circle,draw=black,align=center,text width=3em,inner sep=0pt,outer sep=0.25em]
\tikzstyle{topsmoothnode}=[smoothnode, very thick]
\tikzstyle{colnode}=[color=gray,rotate=55, anchor=west]
\tikzstyle{rownode}=[anchor=east]
\tikzstyle{conn}=[->]
\tikzstyle{anno}=[dotted,color=gray]
\tikzstyle{leftout}=[loosely dotted]

\begin{tikzpicture}
	\begin{pgfonlayer}{nodelayer}
		\node [style=topsmoothnode] (0) at (0, 0) {$\mathbf{s_0}$};
		\node [style=smoothnode] (1) at (-1, -1.5) {$s_0 - 1$};
		\node [style=smoothnode] (2) at (1, -1.5) {$s_0 + 1$};
		\node [style=smoothnode] (3) at (-2, -3) {$s_0 - 2$};
		\node [style=smoothnode] (4) at (2, -3) {$s_0 + 2$};
		\node [style=smoothnode] (5) at (0, -3) {$s_0$};
		\node [style=smoothnode] (6) at (-3, -4.5) {$s_0 - 3$};
		\node [style=smoothnode] (7) at (-1, -4.5) {$s_0 - 1$};
		\node [style=smoothnode] (8) at (3, -4.5) {$s_0 +3$};
		\node [style=smoothnode] (9) at (1, -4.5) {$s_0 + 1$};
		\node [style=colnode] (10) at (0.5, 0.75) {\footnotesize$\mathcal O(k^2)$};
		\node [style=colnode] (11) at (1.5, -0.75) {\footnotesize$\mathcal O(k)$};
		\node [style=colnode] (12) at (2.5, -2.25) {\footnotesize$\mathcal O(1)$};
		\node [style=colnode] (13) at (3.5, -3.75) {\footnotesize$\mathcal O(k^{-1})$};
		\node [style=rownode] (14) at (-5.5, 0) {$\tau_k^{\frac{L}2}$};
		\node [style=rownode] (15) at (-5.5, -1.5) {$\tau_k^{\frac{L-1}2}$};
		\node [style=rownode] (16) at (-5.5, -3) {$\tau_k^{\frac{L-2}2}$};
		\node [style=rownode] (17) at (-5.5, -4.5) {$\tau_k^{\frac{L-3}2}$};
		\node [style=smoothnode] (18) at (-4, -6) {$s_0 - L$};
		\node [style=smoothnode] (19) at (4, -6) {$s_0 + L$};
		\node [style=colnode] (20) at (4.5, -5.25) {\footnotesize $\mathcal O(k^{2-L})$};
		\node [style=rownode] (21) at (-5.5, -6) {$\tau_k^0$};
		\node [style=none, fill=white, inner sep=1em] (22) at (-2, -6) {$\cdots$};
		\node [style=none, fill=white, inner sep=1em] (23) at (0, -6) {$\cdots$};
		\node [style=none, fill=white, inner sep=1em] (24) at (2, -6) {$\cdots$};
		\node [style=none, color=gray, rotate=40, anchor=west, xshift=1em, yshift=0.5em] (25) at (4, -4.5) {$\vdots$};
		\node [style=none, color=gray] (26) at (-5.75, -5.125) {$\vdots$};
	\end{pgfonlayer}
	\begin{pgfonlayer}{edgelayer}
		\draw [style=conn] (0) to (1);
		\draw [style=conn] (1) to (3);
		\draw [style=conn] (3) to (6);
		\draw [style=conn] (3) to node[right, yshift=0.35em]{$\cdot \,k$} (7);
		\draw [style=conn] (1) to node[right, yshift=0.35em]{$\cdot \,k$} (5);
		\draw [style=conn] (0) to node[right, yshift=0.35em]{$\cdot \,k$} (2);
		\draw [style=conn] (2) to node[right, yshift=0.35em]{$\cdot \,k$} (4);
		\draw [style=conn] (2) to (5);
		\draw [style=conn] (5) to node[right, yshift=0.35em]{$\cdot \,k$} (9);
		\draw [style=conn] (4) to node[right, yshift=0.35em]{$\cdot \,k$} (8);
		\draw [style=conn] (4) to (9);
		\draw [style=conn] (5) to (7);
		\draw [style=anno] (0) to (10);
		\draw [style=anno] (2) to (11);
		\draw [style=anno] (4) to (12);
		\draw [style=anno] (8) to (13);
		\draw [style=anno] (14) to node[above]{\footnotesize $\ell=L$} (0);
		\draw [style=anno] (15) to node[above]{\footnotesize $\ell=L-1$} (1);
		\draw [style=anno] (16) to node[above]{\footnotesize $\ell=L-2$} (3);
		\draw [style=anno] (17)  to node[above]{\footnotesize $\ell=L-3$}(6);
		\draw [style=anno]  (21) to node[above]{\footnotesize $\ell=0$} (18);
		\draw [style=anno] (19) to (20);
		\draw [style=leftout] (6) to (18);
		\draw [style=leftout] (6) to (22);
		\draw [style=leftout] (7) to (22);
		\draw [style=leftout] (7) to (23);
		\draw [style=leftout] (9) to (23);
		\draw [style=leftout] (9) to (24);
		\draw [style=leftout] (8) to (24);
		\draw [style=leftout] (8) to (19);
	\end{pgfonlayer}
\end{tikzpicture}

  \end{center}
  \caption{Illustration of the cascade of estimates required to reduce \eqref{eq:st:sm:step} from lemma~\ref{th:st:sm:step} with regularity level $s_0$ and smoothing level $L$ (bold node) to an estimate with only the data on the right-hand side, see remark~\ref{rm:sm-cascade}. Each node  represents an application of estimate~\eqref{eq:st:sm:step} if $\ell > 0$ or \eqref{eq:disc-stab:div} if $\ell = 0$. From top-to-bottom the smoothing level decreases while the regularity decreases (estimates for $\ak v_k$) or increases with factor $k$ (estimates for $\partial_t v_k$).}
	\label{fig:st:sm:cascade}
\end{figure}

Let us now turn to the necessity of the implicit Euler steps. If we apply lemma~\ref{th:st:sm:step} to the discrete error identity \eqref{eq:st:error-u} we must control $\norm{u(t_{n_0}) - u_k^{n_0}}_{\sV{s}}$ in \eqref{eq:st:sm:step}. If $s > r$ this is impossible in the Crank-Nicolson scheme since the regularity of the initial value limits the regularity of the discrete solution due to the occurrence of $\Delta u_k^{n-1}$ in \eqref{eq:st:disc-classical}. The implicit Euler scheme has, in contrast, a smoothing property: From
\begin{equation}\label{eq:st:sm:reg}
	- k_n P \Delta u_k^n = \int_{I^n} P f(t) \dif t - u_k^n + u_k^{n-1}
\end{equation}
we can deduce that $u_k^n \in \sV{r+2n}$ for sufficiently regular $f$ and $\Omega$. The number of implicit Euler steps is consequently chosen such that $u_k^{n_0} \in \sV{s}$ for the largest $s$ in the cascade from figure~\ref{fig:st:sm:cascade}, which is the bottom-right node with $s = s_0 + L$. Therefore $n_0$ is such that $r + 2 n_0 = s_0 + L = 4 + 2 s_0 - r$, i.e.\
\begin{equation}\label{eq:n0}
	n_0 \coloneqq 2 + s_0 - r.
\end{equation}

The following lemma quantifies the control over $\norm{u(t_{n_0}) - u_k^{n_0}}_{\sV{s}}$ for the regularity levels $s$ which will appear in the smoothing cascade, cf.\ figure~\ref{fig:st:sm:cascade}. For sufficiently small $s$ we have the usual second order convergence for the local error in the implicit Euler scheme. For larger $s$ we loose convergence and ultimately stability, the loss of stability of $\norm{u_k^n}_{\sV{s}}$ for large $s$ can already be seen through the factor $k_n$ in \eqref{eq:st:sm:reg}.

\begin{lemma}\label{th:st:sm:euler}
	Let assumption~\ref{assumption:st:low} hold. Then for $s = s_0 - L, \ldots, s_0 + L$ we have
	\[ \nS{u(t_{n_0}) - u_k^{n_0}}_{\sV{s}} \leq \Cas{st:low} k^{\frac{r-s}{2}}. \]
\end{lemma}
\begin{proof}
	We first consider the case $s \geq r$, which yields non-positive powers of $k$. We have by assumption~\ref{as:time1+} that $t_{n_0} \geq C k$. The smoothing estimate of $u$ from assumption~\ref{assumption:st:low} hence implies
	\begin{equation}\label{eq:th:st:sm:euler:tmp1}
		\nS{u(t_{n_0})}_{\sV{s}} = (\tau(t_{n_0}))^{\frac{r-s}{2}} (\tau(t_{n_0}))^{\frac{s-r}{2}} \nS{u(t_{n_0})}_{\sV{s}} \leq \Cas{st:low} \, (t_{n_0})^{\frac{r-s}{2}} \leq \Cas{st:low} k^{\frac{r-s}{2}}.
	\end{equation}
	To estimate $u_k^{n_0}$ we test, for $n = 1, \ldots, n_0$,
	\begin{equation}\label{eq:th:st:sm:euler:tmp2}
		u_k^n - k_n P \Delta u_k^n = \int_{I^n} Pf(t) \dif t + u_k^{n-1}
	\end{equation}
	with $(-P \Delta)^\vars u_k^n$ where $r \leq \vars \leq s - 1$ is chosen later. Then
	\begin{equation*}
		\tfrac 1 2 \norm{u_k^n}^2_{\sV\vars} + k_n \norm{u_k^n}_{\sV{\vars+1}}^2 \leq \int_{I^n} \norm{Pf(t)}_{\sV\vars}^2 \dif t + \norm{u_k^{n-1}}_{\sV{\vars}}^2
	\end{equation*}
	and hence, again using assumptions~\ref{assumption:st:low} and~\ref{as:time1+},
	\begin{equation}\label{eq:th:st:sm:euler:tmp3}
		\norm{u_k^n}_{\sV{\vars}} \leq \Cas{st:low} + C \norm{u_k^{n-1}}_{\sV{\vars}}, \quad	\norm{u_k^n}_{\sV{\vars+1}} \leq \Cas{st:low} k^{-\frac 12} + C k^{-\frac 12} \norm{u_k^{n-1}}_{\sV{\vars}}.
	\end{equation}
	By regularity of the Stokes operator we also conclude from \eqref{eq:th:st:sm:euler:tmp2} that for $r \leq \vars \leq s - 2$:
	\begin{equation}\label{eq:th:st:sm:euler:tmp4}
		\begin{aligned}
			\norm{u_k^n}_{\sV{\vars+2}} &\leq C k^{-1} \left(\int_{I^n} \norm{Pf}_{\sV{\vars}} \dif t + \norm{u^{n-1}_k}_{\sV{\vars}} + \norm{u_k^n}_{\sV{\vars}} \right) \\
			&\leq \Cas{st:low} k^{-1} + C k^{-1} \norm{u_k^{n-1}}_{\sV{\vars}}.
		\end{aligned}
	\end{equation}
	Let now $s = r + 2m$ with $m \geq 0$. Then $s \leq s_0 + L$ implies $m \leq n_0$. By induction over $m = 0, \ldots, n_0$, application of \eqref{eq:th:st:sm:euler:tmp4} yields
	\begin{align*}
		\norm{u_k^m}_{\sV{s}} &= \norm{u_k^m}_{\sV{r+2m}} \leq \Cas{st:low} k^{-1} + C k^{-1} \norm{u_k^{m-1}}_{\sV{r+2(m-1)}} \leq \cdots \\
		&\leq \Cas{st:low} k^{-m} + C k^{-m} \norm{u_k^0}_{\sV{r}} \leq \Cas{st:low} k^{-m}
	\end{align*}
	with constants depending on $n_0$. For $m < n_0$ the estimate \eqref{eq:th:st:sm:euler:tmp3} implies
	\begin{equation*}
		\norm{u_k^{n_0}}_{\sV{s}} \leq \Cas{st:low} + C \norm{u_k^{n_0-1}}_{\sV{s}} \leq \ldots \leq \Cas{st:low} + C \norm{u_k^m}_{\sV{s}} \leq \Cas{st:low} k^{-m}
	\end{equation*}
	which yields for $s = r + 2m$ that
	\begin{equation}\label{eq:th:st:sm:euler:tmp5}
		\norm{u_k^{n_0}}_{\sV{s}} \leq \Cas{st:low} k^{-m} = \Cas{st:low} k^{\frac{r-s}{2}}.
	\end{equation}
	For $s = r + 2m - 1$ we have $m \leq n_0$ and use the second estimate in \eqref{eq:th:st:sm:euler:tmp3} once to get
	\begin{equation*}
		\norm{u_k^m}_{\sV{s}} = \norm{u_k^m}_{\sV{r + 2 (m-1) + 1}} \leq \Cas{st:low} k^{-\frac 1 2} + C k^{-\frac 12} \norm{u_k^{m-1}}_{\sV{r + 2(m-1)}}.
	\end{equation*}
	Applying the already established estimate for $u_k^{m-1}$ and using the stability in $\sV{s}$ from \eqref{eq:th:st:sm:euler:tmp3} if $m < n_0$ we arrive for $s = r + 2 m - 1$ at
	\begin{equation}\label{eq:th:st:sm:euler:tmp6}
		\norm{u_k^{n_0}}_{\sV{s}} \leq \Cas{st:low} k^{-(m-1) - \frac12} = \Cas{st:low} k^{\frac{r-s}{2}}.
	\end{equation}
	Combining \eqref{eq:th:st:sm:euler:tmp5} and \eqref{eq:th:st:sm:euler:tmp6} yields $\norm{u_k^{n_0}}_{\sV{s}} \leq \Cas{st:low} k^{\frac{r-s}{2}}$ for any $s = r, \ldots, s_0 + L$. Together with the estimate \eqref{eq:th:st:sm:euler:tmp1} for $u$ this concludes the proof for $s \geq r$.

	For $s < r$ we define $v_k^n \coloneqq u(t_n) - u_k^n$ and derive from \eqref{eq:st:disc:sm:ie} the error identity
	\[ v_k^n - v_k^{n-1} - k_n P \Delta v_k^n = P \int_{I^n} \Delta u(t) - \Delta  u^n \dif t. \]
	Testing this with $(- P \Delta)^s v_k^n$, yields for $n = 1, \ldots, n_0$ that
	\begin{equation}\label{eq:th:st:sm:euler:tmp7}
		\begin{aligned}
			\nS{v_k^n}^2_{\sV{s}} + k_n \nS{v_k^n}^2_{\sV{s+1}} &\leq C \nS{v_k^{n-1}}^2_{\sV{s}} + C \left( \int_{I^n} \nS{u(t) - u^n}_{\sV{s+2}} \dif t \right)^2\\
			&\leq C \nS{v_k^{n-1}}^2_{\sV{s}} + C \left( \int_{I^n} \int_t^{t_n} \nS{\partial_t u(\tilde t)}_{\sV{s+2}} \dif \tilde t \dif t \right)^2
		\end{aligned}
	\end{equation}
	By assumption~\ref{assumption:st:low} we have $\nS{\partial_t u(t)}^2_{\sV{s+2}} \leq \Cas{st:low} t^{- (4 + s- r)}$, thus
	\[ \left( \int_{I^n} \int_t^{t_n} \nS{\partial_t u(\tilde t)}_{\sV{s+2}} \dif \tilde t \dif t \right)^2 \leq \Cas{st:low} k^{r - s} \]
	and by induction of \eqref{eq:th:st:sm:euler:tmp7} over $n = 1, \ldots, n_0$ the claim for $s < r$ follows.
\end{proof}

For simplicity we always assume in the following that $t_{n_0} \leq 1$ such that $\tau(t_n) = t_n$ for all $n = 0, \ldots, n_0$. 

\begin{theorem}\label{th:st:sm:error}
	Let assumptions~\ref{assumption:st:low} hold. If $s_0 = 1$ the Crank-Nicolson time discretization with $n_0 = 3 - r$ initial implicit Euler steps \eqref{eq:st:disc:sm} of the Stokes equations \eqref{eq:st} satisfies on $J \coloneqq (t_{n_0}, T]$ the a priori error estimate
	\begin{equation}\label{eq:st:sm:error-l2}
		\nTS{\tau_k^{\frac{5-r}{2}} (u - u_k)}_{\sLJ\infty \sV1} + \nTS{\tau_k^{\frac{5-r}{2}} \ak (u - u_k)}_{\sLJ2 \sV2} + \nTS{\tau_k^{\frac{5-r}{2}} (\ak p - p_k)}_{\sLJ2 \sQ1} \leq \Cas{st:low} k^2.
	\end{equation}
	If $s_0 = 2$ we require $n_0 = 4 - r$ implicit Euler steps and there holds 
	\begin{equation}\label{eq:st:sm:error-linf}
		\nTS{\tau_k^{\frac{6-r}{2}} (u - u_k)}_{\sLJ\infty \sV2} + \nTS{\tau_k^{\frac{6-r}{2}} \ak (u - u_k)}_{\sLJ2 \sV3} + \nTS{\tau_k^{\frac{6-r}{2}} (\mk p - p_k)}_{\sLJ\infty \sQ1} \leq \Cas{st:low} k^2.
	\end{equation}
\end{theorem}
\begin{proof}
	We note that $n_0$ satisfies \eqref{eq:n0} and $\tau_k^{\frac L 2}$, with $L = 4 + s_0 - r$ from \eqref{eq:L}, matches the smoothing in \eqref{eq:st:sm:error-l2} and \eqref{eq:st:sm:error-linf}. We first prove, with $v_k \coloneqq \ik u - u_k$, that 
	\begin{equation}\label{eq:th:st:sm:error:tmp1}
		\nTS{\tau_k^{\frac L 2} v_k}_{\sLJ\infty \sV{s_0}} + \nTS{\tau_k^{\frac L 2} \ak v_k}_{\sLJ2 \sV{s_0+1}} + \nTS{\tau_k^{\frac L 2} \partial_t v_k}_{\sLJ2 \sV{s_0-1}} \leq \Cas{st:low} k^2.
	\end{equation}
	Specifically, we prove by induction over the levels $\ell = 0, \ldots, L$ of the cascade that
	\begin{equation}\label{eq:th:st:sm:error:tmp2}
		\nTS{\tau_k^{\frac \ell 2} v_k}_{\sLJ\infty \sV{s}} + \nTS{\tau_k^{\frac \ell 2} \ak v_k}_{\sLJ2 \sV{s+1}} + \nTS{\tau_k^{\frac \ell 2} \partial_t v_k}_{\sLJ2 \sV{s-1}} \leq \Cas{st:low} k^{2 - i}
	\end{equation}
	for $s = s_0 - L + \ell + 2i$ with $i = 0, \ldots, L - \ell$. Then \eqref{eq:th:st:sm:error:tmp1} corresponds to \eqref{eq:th:st:sm:error:tmp2} with $\ell = L$ (and $i = 0$). We remark that the order $k^{2-i}$ may be negative but is sufficient in \eqref{eq:th:st:sm:error:tmp2} since $i$ implicitly counts the number of $\partial_t v_k$-estimates in the cascade, each of which yields a factor $k$, see figure~\ref{fig:st:sm:cascade}. Before we proceed to the induction we prove that
	\begin{equation}\label{eq:th:st:sm:error:prep}
		k^{\frac{\ell}{2}} \nS{v_k(t_{n_0})}_{\sV{s}} \leq \Cas{st:low} k^{2- i}
		\quad\text{and}\quad
		\norm{\tau_k^{\frac \ell 2} (u - \ik u)}_{\sLJ2 \sV{s+1}} \leq \Cas{st:low} k^{2-i}
	\end{equation}
	for all nodes in the cascade. We first note that by definition of $s$ there holds
	\begin{equation}\label{eq:th:st:sm:error:vars}
		\ell + r - s = 4 - 2 i.
	\end{equation}
	The first estimate in \eqref{eq:th:st:sm:error:prep} follows then from lemma~\ref{th:st:sm:euler}:
	\[ k^{\frac \ell 2} \nS{v_k(t_{n_0})}_{\sV{s}} \leq \Cas{st:low} k^{\frac{\ell + r-s}{2}} = \Cas{st:low} k^{2- i}. \]
	To prove the second inequality in \eqref{eq:th:st:sm:error:prep} we use lemma~\ref{th:nodal-int-properties} to estimate
	\begin{equation}\label{eq:th:st:sm:error:tmpA}
		\norm{\tau_k^{\frac \ell 2} (u - \ik u)}_{\sLJ2 \sV{s+1}} \leq C k^2 \norm{\tau_k^{\frac \ell 2} \partial_{tt} u}_{\sLJ2 \sV{s+1}}.
	\end{equation}
	Using $\tau_k \leq \tau$, \eqref{eq:th:st:sm:error:vars} and $\tau(t_{n_0}) = t_{n_0} \geq C k$ we get
	\begin{align*} 
		\norm{\tau_k^{\frac \ell 2} \partial_{tt} u}_{\sLJ2 \sV{s+1}} \leq \norm{\tau^{\frac \ell 2} \partial_{tt} u}_{\sLJ2 \sV{s+1}} \leq \max_{t \in J} (\tau(t))^{-i} \norm{\tau^{\frac{4+s-r}{2}} \partial_{tt} u}_{\sLJ2 \sV{s+1}} \leq \Cas{st:low} k^{-i}
	\end{align*}
	which combined with \eqref{eq:th:st:sm:error:tmpA} concludes the proof of \eqref{eq:th:st:sm:error:prep}. The case $\ell = 0$ in \eqref{eq:th:st:sm:error:tmp2} now follows from the non-smoothing estimate \eqref{eq:disc-stab:div} from lemma~\ref{th:st:disc-stab} and \eqref{eq:th:st:sm:error:prep}:
	\begin{align*}
		\nTS{v_k}_{\sLJ\infty \sV{s}} + \nTS{\ak v_k}_{\sLJ2 \sV{s+1}} + \nTS{\partial_t v_k}_{\sLJ2 \sV{s-1}} &\leq C ( \nS{v_k^{n_0}}_{\sV{s}} + \nTS{u - \ik u}_{\sLJ2 \sV{s+1}} ) \\
		&\leq \Cas{st:low} k^{2 - i}.
	\end{align*}
	For $\ell = 1, \ldots, L$ and $i = 0, \ldots, L - \ell$ we use the smoothing estimate \eqref{eq:st:sm:step} from lemma~\ref{th:st:sm:step} and \eqref{eq:th:st:sm:error:prep}:
	\begin{align*}
		\MoveEqLeft \nTS{\tau_k^{\frac \ell 2} v_k}_{\sLJ\infty \sV{s}} + \nTS{\tau_k^{\frac \ell 2} \ak v_k}_{\sLJ2 \sV{s+1}} + \nTS{\tau_k^{\frac \ell 2} \partial_t v_k}_{\sLJ2 \sV{s-1}} \\
		&\leq C\Big(k^{\frac \ell 2} \nS{v_k^{n_0}}_{\sV{s}}+ \nTS{\tau_k^{\frac \ell 2} (u - \ik u)}_{\sLJ2 \sV{s+1}} \\
		&\quad\qquad + \nTS{\tau_k^{\frac{\ell-1}{2}} \ak v_k}_{\sLJ2 \sV{s}} + k \nTS{\tau_k^{\frac{\ell-1}{2}} \partial_t v_k}_{\sLJ2 \sV{s}} \Big)\\
		&\leq \Cas{st:low} k^{2-i} + C \Big( \nTS{\tau_k^{\frac{\ell-1}{2}} \ak v_k}_{\sLJ2 \sV{s}} + k \nTS{\tau_k^{\frac{\ell-1}{2}} \partial_t v_k}_{\sLJ2 \sV{s}} \Big).
	\end{align*}
	For the remaining terms on the right-hand side we have by induction hypothesis
	\[ \nTS{\tau_k^{\frac{\ell-1}{2}} \ak v_k}_{\sLJ2 \sV{s}} \leq \Cas{st:low} k^{2 - i} \quad\text{and}\quad \nTS{\tau_k^{\frac{\ell-1}{2}} \partial_t v_k}_{\sLJ2 \sV{s}} \leq \Cas{st:low} k^{2 - (i+1)} \]
	which concludes our induction and hence the proof of \eqref{eq:th:st:sm:error:tmp1}.
	
	Splitting $u - u_k = (u - \ik u) + (\ik u - u_k)$ we arrive at the velocity error estimates in \eqref{eq:st:sm:error-l2} and \eqref{eq:st:sm:error-linf} using \eqref{eq:th:st:sm:error:tmp1} and, for the norms of $\ik u - u_k$, the interpolation estimate for $\ik$ and assumption~\ref{assumption:st:low} which imply
	\begin{equation*}
		\norm{\tau_k^{\frac{L}{2}} (u - \ik u)}_{\sLJ\infty \sV{s_0}} \leq C k^2 \norm{\tau_k^{\frac{L}{2}} \partial_{tt} u}_{\sLJ\infty \sV{s_0}} \leq C k^2 \norm{\tau^{\frac{L}{2}} \partial_{tt} u}_{\sLJ\infty \sV{s_0}} \leq \Cas{st:low} k^2.
	\end{equation*}
	and similarly $\norm{\tau_k^{\frac L 2} \ak (u - \ik u)}_{\sLJ2\sV{s_0+1}} \leq \Cas{st:low} k^2$.

	For the pressure we use the error identity \eqref{eq:st:error-p}, which implies
	\begin{equation*}
		\tau_k^{\frac L 2} \nabla(p_k - \ak p) = (P - \Id) \tau_k^{\frac L 2} \ak \Delta(u - u_k)
	\end{equation*}
	and hence for $s_0 = 1$, by the already established velocity estimate from \eqref{eq:st:sm:error-l2},
	\begin{equation*}
		\nTS{\tau_k^{\frac{5-r}{2}} (\ak p - p_k)}_{\sLJ2 \sQ1} \leq C \nTS{\tau_k^{\frac{5-r}{2}} \ak (u - u_k)}_{\sLJ2 \sV2} \leq \Cas{st:low} k^2
	\end{equation*}
	and this is the pressure estimate in \eqref{eq:st:sm:error-l2}. Similarly for $s_0 = 2$ we have
	\begin{equation*}
		\nTS{\tau_k^{\frac{6-r}{2}} (\ak p - p_k)}_{\sLJ\infty \sQ1} \leq C \nTS{\tau_k^{\frac{6-r}{2}} (u - u_k)}_{\sLJ\infty \sV2} \leq \Cas{st:low} k^2
	\end{equation*}
	which implies the pressure estimate in \eqref{eq:st:sm:error-linf} using lemma~\ref{th:midpoint-quadratic-estimate}.
\end{proof}

\begin{remark}
	The techniques used in theorem~\ref{th:st:sm:error} can easily be generalized to other regularity levels $s_0 \in \mathbb N_0$ with appropriately modified assumptions on the data. The case $s_0 = 0$ and $r = 0$ yields $\norm{\tau_k^2 (u - u_k)}_{\sLJ\infty \sL2} \leq C k^2$ after $n_0 = 2$ implicit Euler steps, in agreement with the results in \cite{rannacher1984} for the fully discrete problem.
\end{remark}

\subsection{Navier-Stokes Equations}\label{sec:low-reg:ns}

We now consider the Navier-Stokes equations with smoothing and restrict
ourselves to the situation that $u^0 \in V^2$, which corresponds to $r = 2$ in our linear theory. 
\begin{assumption}\label{assumption:ns:low}
  Let assumption~\ref{assumption:st:low} hold for $r = 2$ and $s_0\in\set{1,2}$, in particular $L \coloneqq 2+s_0$. In addition, let  $\|\nabla u\|_{\sL\infty\sL2}<\infty$, which is satisfied for $d=2$  or for sufficiently small data if $d=3$. We write
  \[
  \Cas{ns:low}:=
  C(\nTS{u^0}_{V^r}, \nTS{f}_{C^2\sH{s_0+L-1}}, \norm{\nabla u}_{\sL\infty\sL2})
  \]
	where $C(\cdots) > 0$ is some function independent of the data and, by abuse of notation, may change with each occurrence of $\Cas{ns:low}$.
\end{assumption}
Under these assumptions for $s\in\mathbb{Z}$ with $-2\le s\le s_0+L$ we define $\alpha \coloneqq \frac{(2 m + s - 2)^+}{2}$ for $m \in \set{0, 1, 2}$. The solution to the Navier-Stokes equations then satisfies the bounds
\begin{equation}\label{reg:ns:sm1}
  \nTS{\tau^\alpha \partial_t^m u}_{\sC0\sV{s}} + \nTS{\tau^\alpha \partial_t^m u}_{\sL2\sV{s+1}} \leq \Cas{ns:low}. 
\end{equation}
Given $s\geq 2$ it further holds
\begin{equation}\label{reg:ns:sm2}
  \nTS{\tau^\alpha \partial_t^m p}_{\sC0\sQ{s-1}} + \nTS{\tau^\alpha \partial_t^m p}_{\sL2\sQ{s}} \leq \Cas{ns:low}.
\end{equation}
Again, this follows by similar arguments as in theorems 2.3--2.5 from \cite{HeywoodRannacher1982}. As in the linear case the required regularity is essentially reduced since we do not ask for nonlocal compatibility conditions to hold. 

In the following we apply the linear smoothing techniques to $v_k \coloneqq u_k - \ik u$ using the error identity~\eqref{eq:ns:error-u}. To bound the nonlinearity further error estimates are needed, resulting in a double-cascade which is sketched for $s_0 = 1$ in figure~\ref{fig:ns:sm:cascade}. The derivation of the necessary estimates for all nodes of this double-cascade will be the goal for much of this section. 

\begin{figure}\label{fig:ns:sm:cascade}
  \begin{center}
\pgfdeclarelayer{edgelayer}
\pgfdeclarelayer{nodelayer}
\pgfsetlayers{edgelayer,nodelayer,main}

\tikzstyle{none}=[inner sep=0pt]
\tikzstyle{smoothnode}=[circle,draw=black,align=center,text width=1.8em,inner sep=0pt,outer sep=0.25em]
\tikzstyle{topsmoothnode}=[smoothnode, very thick]
\tikzstyle{smoothnodesec}=[smoothnode, dashed, color=gray]
\tikzstyle{colnode}=[color=gray,rotate=55, anchor=west]
\tikzstyle{rownode}=[anchor=east]
\tikzstyle{conn}=[->]
\tikzstyle{connsec}=[dashed,conn,color=gray]
\tikzstyle{anno}=[dotted,color=gray]
\tikzstyle{leftout}=[loosely dotted]

\begin{tikzpicture}
	\begin{pgfonlayer}{nodelayer}
		\node [style=topsmoothnode] (0) at (0, 0) {$\mathbf 1$};
		\node [style=smoothnode] (1) at (-1, -1.5) {$0$};
		\node [style=smoothnode] (2) at (1, -1.5) {$2$};
		\node [style=smoothnode] (3) at (-2, -3) {$-1$};
		\node [style=smoothnode] (4) at (2, -3) {$3$};
		\node [style=smoothnode] (5) at (0, -3) {$1$};
		\node [style=smoothnode] (6) at (-3, -4.5) {$-2$};
		\node [style=smoothnode] (7) at (-1, -4.5) {$0$};
		\node [style=smoothnode] (8) at (3, -4.5) {$4$};
		\node [style=smoothnode] (9) at (1, -4.5) {$2$};
		\node [style=smoothnodesec] (10) at (0, -1.5) {$1$};
		\node [style=smoothnodesec] (11) at (-1, -3) {$0$};
		\node [style=smoothnodesec] (12) at (1, -3) {$2$};
		\node [style=smoothnodesec] (13) at (-2, -4.5) {$-1$};
		\node [style=smoothnodesec] (14) at (2, -4.5) {$3$};
		\node [style=smoothnodesec] (15) at (0, -4.5) {$1$};
		\node [style=rownode] (16) at (-4.0, 0) {$\tau_k^{\frac32}$};
		\node [style=rownode] (17) at (-4.0, -1.5) {$\tau_k^{1}$};
		\node [style=rownode] (18) at (-4.0, -3) {$\tau_k^{\frac12}$};
		\node [style=rownode] (19) at (-4.0, -4.5) {$\tau_k^0$};
		\node [style=colnode] (20) at (0.5, 0.75) {\footnotesize$\mathcal O(k^2)$};
		\node [style=colnode] (21) at (1.0, 0.00) {\footnotesize$\mathcal O(k^{1.5})$};
		\node [style=colnode] (22) at (1.5, -0.75) {\footnotesize$\mathcal O(k)$};
		\node [style=colnode] (23) at (2.0, -1.50) {\footnotesize$\mathcal O(k^{0.5})$};
		\node [style=colnode] (24) at (2.5, -2.25) {\footnotesize$\mathcal O(1)$};
		\node [style=colnode] (25) at (3.0, -3.00) {\footnotesize$\mathcal O(k^{-0.5})$};
		\node [style=colnode] (26) at (3.5, -3.75) {\footnotesize$\mathcal O(k^{-1})$};
	\end{pgfonlayer}
	\begin{pgfonlayer}{edgelayer}
		\draw [style=connsec] (10) to (11);
		\draw [style=connsec] (10) to (12);
		\draw [style=connsec] (11) to (13);
		\draw [style=connsec] (11) to (15);
		\draw [style=connsec] (12) to (15);
		\draw [style=connsec] (12) to (14);
		\draw [style=conn] (0) to (1);
		\draw [style=conn] (1) to (3);
		\draw [style=conn] (3) to (6);
		\draw [style=conn] (3) to (7);
		\draw [style=conn] (1) to (5);
		\draw [style=conn] (0) to (2);
		\draw [style=conn] (2) to (4);
		\draw [style=conn] (2) to (5);
		\draw [style=conn] (5) to (9);
		\draw [style=conn] (4) to (8);
		\draw [style=conn] (4) to (9);
		\draw [style=conn] (5) to (7);
		\draw [style=anno] (16) to (0);
		\draw [style=anno] (17) to (1);
		\draw [style=anno] (18) to (3);
		\draw [style=anno] (19)  to (6);
		\draw [style=anno] (0) to (20);
		\draw [style=anno] (10) to (21);
		\draw [style=anno] (2) to (22);
		\draw [style=anno] (12) to (23);
		\draw [style=anno] (4) to (24);
		\draw [style=anno] (14) to (25);
		\draw [style=anno] (8) to (26);
	\end{pgfonlayer}
\end{tikzpicture}
  \end{center}
  \caption{Double-cascade of estimates required to establish the $\sL2$-in-time pressure error estimate for the Navier-Stokes equations with smoothing in theorem~\ref{th:ns:sm:error}, corresponding to $s_0 = 1$ in assumption~\ref{assumption:ns:low} (bold node). The arrows only indicate dependencies due to the smoothing, c.f.\ figure~\ref{fig:st:sm:cascade}, those due to the nonlinearity are not shown but make the double cascade necessary. For readability the labels $\cdot k$ have been omitted.}
\end{figure}
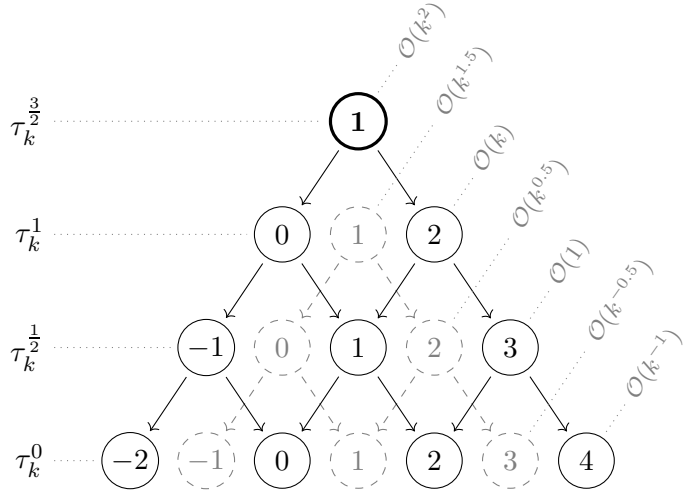

As in the linear case, we start with $n_0 = 2 + s_0 - r = s_0$ implicit Euler steps. The scheme has the form: Find $u_k^n \in \sV1$ and $q_k^n \in \sQ0$ such that
\begin{subequations}\label{eq:ns:disc:sm}
\begin{gather}
	u_k^n - u_k^{n-1} - k_n \Delta u_k^n + k_n (u_k^n \cdot \nabla u_k^n) - k_n \nabla p_k^n = \int_{I^n} f(t) \dif t\label{eq:ns:disc:sm:ie}
	\intertext{for $n = 1, \ldots, n_0$ with $u_k^0 \coloneqq u^0$. Afterwards we use the Crank-Nicolson scheme from section~\ref{sec:high-reg}, i.e.\ $u_k \in \scG1(\grid_k^{n_0}, \sV1)$ and $p_k \in \sdG0(\grid_k^{n_0}, \sQ0)$ must satisfy}
	\partial_t u_k - \Delta u_k + (\ak u_k) \cdot \nabla (\ak u_k) + \nabla p_k = f \quad \text{in $\sdG0(\grid_k^{n_0}, \sHz1(\Omega))'$}
\end{gather}
\end{subequations}
with initial value $u_k(t_{n_0}) = u_k^{n_0}$ from the last implicit Euler step.

\begin{lemma}\label{th:ns:sm:nonlin}
	Let assumption~\ref{assumption:ns:low} hold. With $J \coloneqq (t_{n_0}, T]$, we have for each $s = -2, \ldots, s_0 + L$:
	\begin{equation}\label{eq:ns:sm:nonlin}
		\nTS{\tau_k^{\frac{s+2}2} \ak \oCdiff(u, \ak \ik u)}_{\sLJ2 \sH{s-1}} \leq \Cas{ns:low} k^2.
	\end{equation}
\end{lemma}
\begin{proof}
	Combining the approximation estimates from lemmata~\ref{th:nodal-int-properties} and~\ref{th:nodal-midpoint-estimates} and the regularity results for our continuous solution, we have for $-2 \leq s \leq s_0 + L$:
	\begin{equation}\label{eq:th:ns:sm:nonlin:tmp}
		\left\{ \quad \begin{alignedat}{3}
			\nTS{\tau_k^{\frac{(s-2)^+}{2}} u}_{\sLJ\infty \sV{s}} &\leq \Cas{ns:low}, &&\;& \nTS{\tau_k^{\frac{(s)^+}{2}} (u - \ak \ik u)}_{\sLJ2\sV{s+1}} &\leq \Cas{ns:low} k, \\
			\nTS{\tau_k^{\frac{(s)^+}{2}} (u - \ak \ik u)}_{\sLJ\infty\sV{s}} &\leq \Cas{ns:low} k, &&&\nTS{\tau_k^{\frac{s+2}{2}} (u - \ik u)}_{\sLJ2\sV{s+1}} &\leq \Cas{ns:low} k^2.
		\end{alignedat} \right.
	\end{equation}
	Combining the estimates for the nonlinearity from \eqref{eq:nl:sl1} with \eqref{eq:th:ns:sm:nonlin:tmp} yields \eqref{eq:ns:sm:nonlin} for $s \leq 1$:
	\begin{align*}
		\MoveEqLeft\nTS{\tau_k^{\frac{s+2}{2}} \oCdiff(u, \ak \ik u)}_{\sLJ2\sH{s-1}} \\
		&\leq C \Big( \nTS{u - \ak \ik u}_{\sLJ\infty \sL2} \nTS{\tau_k^{\frac{(s+1)^+}{2}} (u - \ak \ik u)}_{\sLJ2\sV{s+2}} \\
		&\qquad\quad+ \nTS{\ak \ik u}_{\sLJ\infty \sV2} \nTS{\tau_k^{\frac{s+2}{2}} (u - \ik u)}_{\sLJ2 \sV{s+1}} \Big)\\
		&\leq C k^2 \Big( \nTS{\partial_t u}_{\sLJ\infty \sL2} \nTS{\tau_k^{\frac{(s+1)^+}{2}} \partial_t u}_{\sLJ2\sV{s+2}} + \nTS{u}_{\sCJ0 \sV2} \nTS{\tau_k^{\frac{s+2}{2}} \partial_{tt} u}_{\sLJ2 \sV{s+1}} \Big)\\
		&\leq \Cas{ns:low} k^2.
	\end{align*}
	For $s \geq 2$ we use \eqref{eq:nl:sge1} and \eqref{eq:th:ns:sm:nonlin:tmp} to get
	\begin{align*}
		\MoveEqLeft\nTS{\tau_k^{\frac{s+2}{2}} \ak \nabla^{s-1} \oCdiff(u, \ak \ik u)}_J\\
		&\leq C \sum_{i=1}^{s-1} \Big(\nTS{\tau_k^{\frac{i+1}{2}} (u - \ak \ik u)}_{\sLJ\infty \sV{i+1}} \nTS{\tau_k^{\frac{s-i}{2}} (u - \ak \ik u)}_{\sLJ2 \sV{s-i+1}} \\
		&\quad\qquad\qquad + \nTS{\tau_k^{\frac{i+2}{2}} (u - \ik u)}_{\sLJ2\sV{i+1}} \nTS{\tau_k^{\frac{s-i-1}{2}} u}_{\sLJ\infty \sV{s-i+1}} \Big)\\
		&\leq \Cas{ns:low} k^2
	\end{align*} 
	noting in particular that the smoothing on the left-hand side is sufficient for the smoothing used on the right. Combining these estimates and \eqref{eq:ns:sm:nonlin} with $s = 1$ for the $\sL2$-term, we conclude that \eqref{eq:ns:sm:nonlin} is true for the full $\sH{s-1}$-norm.
\end{proof}

\begin{lemma}\label{th:ns:sm:euler}
	Let assumption~\ref{assumption:ns:low} hold. Then there exists $C_0 > 0$ such that for each $s = -2, \ldots, s_0 + L$ we have 
	\begin{equation}\label{eq:ns:sm:ie}
		\nS{u(t_{n_0}) - u_k^{n_0}}_{\sV{s}} \leq \Cas{ns:low} k^{\frac{2 - s}{2}}
	\end{equation}
	if the stepsize condition $k < C_0 \nTS{\nabla u}_{\sL\infty\sL2}^{-1} \nTS{\Delta u}_{\sL\infty\sL2}^{-1}$ is satisfied.
\end{lemma}
\begin{proof}
	We first prove \eqref{eq:ns:sm:ie} for $s \leq 2$. With $v_k^n \coloneqq u^n - u_k^n$ we get from \eqref{eq:ns:disc:sm:ie} by elementary calculations the error identity 
	\begin{align*}
		v_k^n - v_k^{n-1} - k_n P \Delta v_k^n &= P \int_{I^n} \Delta (u(t) - u^n) - u(t) \cdot \nabla u(t) +  u_k^n \cdot \nabla u_k^n \dif t\\
		&= P \int_{I^n} \Delta (u(t) - u^n) - \oCdiff(u(t), u^n) + \oCdiff(u_k^n, u^n) \dif t.
	\end{align*}
	Testing this identity with $(-P \Delta)^s v_k^n$ for $s \leq 1$ yields the stability estimate
	\begin{equation}\label{eq:th:ns:sm:euler:stab}
		\tfrac 1 2 \nS{v_k^n}^2_{\sV{s}} + k_n \nS{v_k^n}^2_{\sV{s+1}} \leq \tfrac 1 2 \nS{v_k^{n-1}}^2_{\sV{s}} + (r_{k,1}^n - r_{k,2}^n + r_{k,3}^n, (-P\Delta)^s v_k^n)
	\end{equation}
	where 
	\begin{equation*}
		r_{k,1}^n \coloneqq \int_{I^n} P\Delta (u(t) - u^n) \dif t, \; r_{k,2}^n \coloneqq \int_{I^n} \oCdiff(u(t), u^n) \dif t, \; r_{k,3}^n \coloneqq \int_{I^n} \oCdiff(u_k^n, u^n) \dif t.
	\end{equation*}
	To estimate the $r_{k,1}^n$-term we use arguments as in the linear case, see \eqref{eq:th:st:sm:euler:tmp7} from lemma~\ref{th:st:sm:euler}. For $\delta  > 0$ there thus holds
	\begin{align*}
		\spL{r_{k,1}^n, (- P \Delta)^s v_k^n} &\leq C(\delta) \left( \int_{I^n} \nS{u(t) - u^n}_{\sV{s+2}} \dif t \right)^2 + \delta \nS{v_k^n}_{\sV{s}}^2 \\
		&\leq \Cas{ns:low}(\delta) k^{2-s} + \delta \nS{v_k^n}_{\sV{s}}^2.
	\end{align*}
	Using this estimate in \eqref{eq:th:ns:sm:euler:stab} and moving the $\delta$-term to the left-hand side, we get
	\begin{equation}\label{eq:th:ns:sm:euler:stab2}
		\begin{aligned}
			\tfrac 14 \nS{v_k^n}^2_{\sV{s}} + k_n \nS{v_k^n}^2_{\sV{s+1}} \leq \tfrac 1 2 \nS{v_k^{n-1}}^2_{\sV{s}} + \Cas{ns:low} k^{2-s} + (r_{k,2}^n + r_{k,3}^n, (-P\Delta)^s v_k^n).
		\end{aligned}
	\end{equation}
	To estimate the $r_{k,2}^n$-term in \eqref{eq:th:ns:sm:euler:stab2} we use two different splittings of the scalar product:
	\begin{equation}\label{eq:th:ns:sm:euler:tmp1}
		\spL{r_{k,2}^n, (-P\Delta)^s v_k^n} \leq \begin{cases} C(\delta) \nS{r_{k,2}}_{\sV{s}}^2 + \delta \nS{v_k^n}_{\sV{s}}^2 & \text{if $s = -2, 0$,}\\
		C(\delta) k^{-1} \nS{r_{k,2}^n}_{\sV{s-1}}^2 + \delta k_n \nS{v_k^n}_{\sV{s+1}}^2 &\text{if $ s= -1, 1$.}
		\end{cases}
	\end{equation}
	Note that $\nS{r_{k, 2}^n}_{\sV{s}}$ now only appears with $s = 0$ and $s = -2$. With $v(t) \coloneqq u(t) - u^n$, \eqref{eq:nl:2} and \eqref{eq:nl:3} we get for $s = 0$:
	\begin{align*}
		\int_{I^n} \nS{\oCdiff(u(t), u^n)}_{\sL2} \dif t \leq C \int_{I^n} \nS{v(t)}_{\sV1} \left( \nS{v(t)}_{\sV2} + \nS{u^n}_{\sV2} \right) \dif t \leq \Cas{ns:low} k
	\end{align*}
	since $\norm{v}_{\sC0 \sV2} \leq 2 \norm{u}_{\sC0 \sV2} \leq \Cas{ns:low}$. For $s = -2$ we have 
	\begin{align*}
		\int_{I^n} \nS{\oCdiff(u(t), u^n)}_{\sH{-2}} \dif t \leq C \int_{I^n} \nS{v}_{\sV{-1}} \left( \nS{v}_{\sV2} + \nS{u^n}_{\sV2} \right) \dif t \leq \Cas{ns:low} k^2,
	\end{align*}
	since $\nS{v(t)}_{\sV{-1}} \leq \int_{t}^{t_n} \nS{\partial_t v(\tilde t)}_{\sV{-1}} \dif \tilde t \leq \Cas{ns:low} k$. Combining these estimates with \eqref{eq:th:ns:sm:euler:tmp1} we get for \eqref{eq:th:ns:sm:euler:stab}, after choosing a suitable $\delta$, that
	\begin{equation}\label{eq:th:ns:sm:euler:stab3}
		\nS{v_k^n}^2_{\sV{s}} + k_n \nS{v_k^n}^2_{\sV{s+1}} \leq C \nS{v_k^{n-1}}^2_{\sV{s}} + \Cas{ns:low} k^{2- s} + C k_n (r_{k,3}^n, (-P\Delta)^s v_k^n).
	\end{equation}
	It remains to estimate the $r_{k,3}^n$-term. For $s = 0$ we use the cancellation property of the nonlinearity and \eqref{eq:nl:2b} to get
	\begin{align*}
		k_n (\oCdiff(u_k^n, u^n), v_k^n) &= k_n (v_k^n \cdot \nabla u^n, v_k^n) \\
		&\leq C k_n \nS{\nabla u^n}_{\sL2}^{\frac 1 2} \nS{\Delta u^n}_{\sL2}^{\frac 1 2} \nS{v_k^n}_{\sL2} \nS{v_k^n}_{\sV1}\\
		&\leq C(\delta) k_n \nS{\nabla u^n}_{\sL2} \nS{\Delta u^n}_{\sL2} \nS{v_k^n}_{\sL2}^2 + \delta k_n \nS{v_k^n}_{\sV1}^2.
	\end{align*}
	With suitable chosen $\delta > 0$ and using the stepsize condition we may move both terms to the left-hand side of \eqref{eq:th:ns:sm:euler:stab3}, implying that
	\begin{equation*}
		\nS{v_k^n}^2_{\sL2} + k_n \nS{v_k^n}^2_{\sV1} \leq C \nS{v_k^{n-1}}^2_{\sL2} + \Cas{ns:low} k^2.
	\end{equation*}
	Applying this estimate iteratively for $n = 1, \ldots, n_0$ we arrive at \eqref{eq:ns:sm:ie} for $s = 0$. Actually this also yields \eqref{eq:ns:sm:ie} for $s = 1$, but we still consider $s = 1$ in \eqref{eq:th:ns:sm:euler:stab3} to prove \eqref{eq:ns:sm:ie} for $s = 2$: We estimate  
	\begin{align*}
		k_n \spL{\oCdiff(u_k^n, u^n), (- P \Delta) v_k^n} \leq C(\delta) k \nS{\oCdiff(u_k^n, u_n)}_{\sL2}^2 + \delta k_n \nS{v_k^n}_{\sV2}^2
	\end{align*}
	and using the estimate for $s = 0$, $\nS{v_k^n}_{\sV1} \leq \Cas{ns:low}$, we get with \eqref{eq:nl:s0} and \eqref{eq:nl:sl1}:
	\begin{align*}
		\nS{\oCdiff(u_k^n, u^n)}_{\sL2} \leq C \left( \nS{v_k^n}_{\sV1}^{\frac 3 2} \nS{v_k^n}^{\frac12}_{\sV2} + \nS{v_k^n}_{\sV1} \nS{u^n}_{\sV2} \right) \leq \Cas{ns:low} \left(\nS{v_k^n}^{\frac12}_{\sV2}  + 1 \right).
	\end{align*}
	This implies
	\begin{align*}
		k_n \spL{\oCdiff(u_k^n, u^n), (- P \Delta) v_k^n} &\leq \Cas{ns:low}(\delta) k_n \left( \nS{v_k^n}_{\sV2} + 1 \right) + \delta k_n \nS{v_k^n}_{\sV2}^2\\
		&\leq \Cas{ns:low}(\delta) k + 2 \delta k_n \nS{v_k^n}_{\sV2}^2
	\end{align*}
	which inserted into \eqref{eq:th:ns:sm:euler:stab3}, moving the $\delta$-term to the left-hand side, yields
	\begin{equation*}
		\nS{v_k^n}^2_{\sV1} + k_n \nS{v_k^n}^2_{\sV2} \leq C \nS{v_k^{n-1}}^2_{\sV1} + \Cas{ns:low} k
	\end{equation*}
	and \eqref{eq:ns:sm:ie} for $s = 2$ follows. For $s = -1, -2$ we estimate the $r_{k,3}^n$-term in \eqref{eq:th:ns:sm:euler:stab3} by 
	\begin{equation*}
		k_n \spL{ \oCdiff(u_k^n, u^n), (-P\Delta)^s v_k^n} \leq \begin{cases} C(\delta) k \nS{\oCdiff(u_k^n, u^n)}^2_{\sH{-2}} + \delta k_n \nS{v_k^n}_{\sV{0}}^2 & \text{if $s = -1$},\\
		C(\delta) k^2 \nS{\oCdiff(u_k^n, u^n)}^2_{\sH{-2}} + \delta \nS{v_k^n}_{\sV{-2}}^2 & \text{if $s = -2$} \end{cases}
	\end{equation*}
	and combined with \eqref{eq:th:ns:sm:euler:stab3} this proves \eqref{eq:ns:sm:ie} for $s = -1, -2$ since we have
	\[ \nS{\oCdiff(u_k^n, u^n)}_{\sH{-2}} \leq C \left( \nS{v_k^n}_{\sL2} \nS{v_k^n}_{\sV1} + \nS{u^n}_{\sV2} \nS{v_k^n}_{\sV{-1}} \right) \leq \Cas{ns:low} k \]
	which follows from \eqref{eq:nl:sl1} and the estimate \eqref{eq:ns:sm:ie} for $s = 0$ and $s = 1$. For $s > 2$ the continuous smoothing result~\eqref{reg:ns:sm1} implies
	\begin{equation*}
		\nS{u(t_{n_0})}_{\sV{s}} = (\tau(t_{n_0}))^{\frac{r-s}{2}} (\tau(t_{n_0}))^{\frac{s-r}{2}} \nS{u(t_{n_0})}_{\sV{s}} \leq \Cas{ns:low} \, (t_{n_0})^{\frac{r-s}{2}} \leq \Cas{ns:low} k^{\frac{r-s}{2}}.
	\end{equation*}
	For the corresponding bound for $u_k$ we prove that for $m = 1, \ldots, n_0$ there holds
	\begin{equation}\label{eq:sm:initial:pf}
		\nS{u_k^n}_{\sV{2m+1}} \leq \Cas{ns:low} k^{-m + \frac 12}, \quad
		\nS{u_k^n}_{\sV{2m+2}} \leq \Cas{ns:low} k^{-m}
	\end{equation}
	for all $n = m, \ldots, n_0$. It is easy to see that \eqref{eq:sm:initial:pf} implies $\nS{u_k^{n_0}}_{\sV{s}} \leq \Cas{ns:low} k^{\frac{2-s}{2}}$ for $3 \leq s \leq s_0 + L$, e.g.\ for $s = s_0 + L = 2 + 2 s_0$ we have by \eqref{eq:sm:initial:pf} for $m \coloneqq n_0 = s_0$ that
	\begin{equation*}
		\nS{u_k^{n_0}}_{\sV{s_0 + L}} \leq \nS{u_k^{n_0}}_{\sV{2 m + 2}} \leq \Cas{ns:low} k^{-m} = \Cas{ns:low} k^{\frac{2 - s}{2}},
	\end{equation*}
	in particular the number of implicit Euler steps $n_0$ is sufficient. We prove \eqref{eq:sm:initial:pf} by induction over $m$, thus let $m = 1, \ldots, n_0$ and \eqref{eq:sm:initial:pf} be valid for smaller $m$ if $m > 1$. For the first estimate in \eqref{eq:sm:initial:pf} we use the a stability estimate, following from standard arguments:
	\begin{equation*}
		\nS{u_k^n}_{\sV{2m}}^2 + k_n \nS{u_k^n}^2_{\sV{2m+1}} \leq 2 \int_{I^n} \nS{P f(t)}^2_{\sV{2m}} \dif t + 2 \nS{u_k^{n-1}}^2_{\sV{2m}} + k_n \nS{u_k^n \cdot \nabla u_k^n}^2_{\sH{2m-1}}
	\end{equation*}
	for $n = m, \ldots, n_0$. Either by induction or, for $m = 1$, by the already proven estimate \eqref{eq:ns:sm:ie} with $s = 2$, we can bound $\nS{u_k^{n-1}}_{\sV{2m}}$ and trivially the $f$-term, yielding
	\begin{equation}\label{eq:th:ns:sm:euler:tmpA1}
		\nS{u_k^n}_{\sV{2m}}^2 + k_n \nS{u_k^n}^2_{\sV{2m+1}} \leq \Cas{ns:low} k^{2(1-m)} + k_n \nS{u_k^n \cdot \nabla u_k^n}^2_{\sH{2m-1}}.
	\end{equation}
	For the nonlinearity we combine \eqref{eq:nl:sge1} for all derivatives in the $\sH{2m-1}$-norm and \eqref{eq:nl:2} for the $\sL2$-term, to get
	\begin{equation}\label{eq:th:ns:sm:euler:tmpA1b}
		\nS{u_k^n \cdot \nabla u_k^n}_{\sH{2m-1}} \leq C \nS{u_k^n}_{\sV1} \nS{u_k^n}_{\sV2} + C \sum_{i = 1}^{2m-1} \nS{u_k^n}_{\sV{i+1}} \nS{u_k^n}_{\sV{2m-i+1}}.
	\end{equation}
	Noting that $2 \leq i + 1 \leq 2m$ and $2 \leq 2m - i + 1 \leq 2m$ for $i = 1, \ldots, 2m-1$, we can bound all terms on the right-hand side either by induction or using \eqref{eq:ns:sm:ie} for $s = 1, 2$:
	\begin{equation*}
		\nS{u_k^n \cdot \nabla u_k^n}_{\sH{2m-1}} \leq \Cas{ns:low} + \Cas{ns:low} \sum_{i = 1}^{2m-1} k^{\frac{2 - (i+1)}{2}} k^{\frac{2 - (2m + 1 -i)}{2}} \leq \Cas{ns:low} k^{\frac{2 - 2m}{2}} = \Cas{ns:low} k^{1-m}.
	\end{equation*}
	Combining this inequality with the stability estimate \eqref{eq:th:ns:sm:euler:tmpA1} we arrive at
	\begin{equation*}
		k_n \nS{u_k^n}^2_{\sV{2m+1}} \leq \nS{u_k^n}_{\sV{2m}}^2 + k_n \nS{u_k^n}^2_{\sV{2m+1}} \leq \Cas{ns:low} k^{2(1-m)}
	\end{equation*}
	which yields the first result in \eqref{eq:sm:initial:pf}. For the second estimate we employ the regularity results for the Stokes equations to get
	\begin{equation*}
		\nS{u_k^n}_{\sV{2m+2}} \leq C k^{-1} \left( \int_{I^n} \nS{P f(t)}_{\sV{2m}} \dif t + \nS{u_k^n}_{\sV{2m}} + \nS{u_k^{n-1}}_{\sV{2m}} \right) + C \nS{u_k^n \cdot \nabla u_k^n}_{\sV{2m}}.
	\end{equation*}
	Estimating the terms in brackets using \eqref{eq:sm:initial:pf} by induction, or \eqref{eq:ns:sm:ie} with $s = 2$, yields:
	\begin{equation}\label{eq:th:ns:sm:euler:tmpA2}
		\nS{u_k^n}_{\sV{2m+2}} \leq \Cas{ns:low} k^{-m} + C \nS{u_k^n \cdot \nabla u_k^n}_{\sV{2m}}.
	\end{equation}
	For the nonlinear term we proceed just as in \eqref{eq:th:ns:sm:euler:tmpA1b}, yielding
	\begin{equation*}
		\nS{u_k^n \cdot \nabla u_k^n}_{\sH{2m}} \leq C \nS{u_k^n}_{\sV1} \nS{u_k^n}_{\sV2} + C \sum_{i = 1}^{2m} \nS{u_k^n}_{\sV{i+1}} \nS{u_k^n}_{\sV{2m-i+2}}.
	\end{equation*}
	The norms on the right are at most of order $2m+1$. Hence using \eqref{eq:sm:initial:pf} by induction, the already proven first estimate in \eqref{eq:sm:initial:pf} for $m$, and estimates for $s = 1, 2$ we arrive at
	\begin{align*}
		\nS{u_k^n \cdot \nabla u_k^n}_{\sH{2m}} &\leq \Cas{ns:low} + \Cas{ns:low} \sum_{i = 1}^{2m} k^{\frac{2-(i+1)}{2}} k^{\frac{2-(2m - i + 2)}{2}} \\
		&\leq \Cas{ns:low} k^{\frac{1 - 2m}{2}} \leq \Cas{ns:low} k^{-m}.
	\end{align*}
	Combining this result with \eqref{eq:th:ns:sm:euler:tmpA2} implies $\nS{u_k^n}_{\sV{2m+2}} \leq \Cas{ns:low} k^{-m}$, i.e.\ the second estimate in \eqref{eq:sm:initial:pf}, completing the induction.
\end{proof}

\begin{theorem}\label{th:ns:sm:error}
	Let assumption~\ref{assumption:ns:low} hold. If $s_0 = 1$ there exists $C_0 > 0$ such that the Crank-Nicolson time discretization with $n_0 = 1$ implicit Euler steps \eqref{eq:ns:disc:sm} of the Navier-Stokes equations \eqref{eq:ns} satisfies on $J \coloneqq (t_1, T]$ the a priori error estimate
	\begin{equation}\label{eq:ns:sm:error-l2}
		\nTS{\tau_k^{\frac{3}{2}} (u - u_k)}_{\sLJ\infty \sV1} + \nTS{\tau_k^{\frac{3}{2}} \ak (u - u_k)}_{\sLJ2 \sV2} + \nTS{\tau_k^{\frac{3}{2}} (\ak p - p_k)}_{\sLJ2 \sQ1} \leq \Cas{ns:low} k^2
	\end{equation}
	if the stepsize condition $k < C_0 \nTS{u}_{\sL\infty\sV2}^{-2}$ is satisfied. If $s = 2$ there exists another $C_0 > 0$ such that after $n_0 = 2$ implicit Euler steps there holds with $J \coloneqq (t_2, T]$:
	\begin{equation}\label{eq:ns:sm:error-linf}
		\nTS{\tau_k^2 (u - u_k)}_{\sLJ\infty \sV2} + \nTS{\tau_k^2 \ak (u - u_k)}_{\sLJ2 \sV3} + \nTS{\tau_k^2 (\mk p - p_k)}_{\sLJ\infty \sQ1} \leq \Cas{ns:low} k^2
	\end{equation}
	if $k < C_0 \nTS{u}_{\sL\infty\sV2}^{-2}$ is satisfied. 
\end{theorem}
\begin{proof}
	\newcommand\eqpre{eq:th:ns:sm:error}
	We note that $n_0 = s_0$ and the smoothing level in \eqref{eq:ns:sm:error-l2} and \eqref{eq:ns:sm:error-linf} is $L = 2 + s_0$. The procedure is just as in the linear case, i.e.\ theorem~\ref{th:st:sm:error}: For both $s_0 = 1$ and $s_0 = 2$ we first bound $v_k \coloneqq \ik u - u_k$ by induction over the smoothing cascade, use this to prove the velocity error estimates and then arrive at pressure estimates using \eqref{eq:ns:error-p}. We first prove, under the stepsize restriction not mentioned furthermore, that
	\begin{equation}\label{\eqpre:vk}
		\nTS{\tau_k^{\frac L 2} v_k}_{\sLJ\infty \sV{s_0}} + \nTS{\tau_k^{\frac L 2} \ak v_k}_{\sLJ2 \sV{s_0+1}} + \nTS{\tau_k^{\frac L 2} \partial_t v_k}_{\sLJ2 \sV{s_0-1}} \leq \Cas{ns:low} k^2.
	\end{equation}
	We proceed by induction over the levels and nodes of the (double) smoothing cascade: For $\ell = 0, \ldots, L$ and $s = \ell - 2, \ldots, s_0 + L - \ell$ we prove that
	\begin{equation}\label{\eqpre:tmp1}
		\nTS{\tau_k^{\frac \ell 2} v_k}^2_{\sLJ\infty \sV{s}} + \nTS{\tau_k^{\frac \ell 2} \ak v_k}^2_{\sLJ2 \sV{s+1}} + \nTS{\tau_k^{\frac \ell 2} \partial_t v_k}^2_{\sLJ2 \sV{s-1}} \leq \Cas{ns:low} k^{2 + \ell - s}.
	\end{equation}
	Then \eqref{\eqpre:vk} corresponds to $\ell = L$ and $s = s_0$. We prepare some estimates first: Just as in the linear case, or as in the proof of lemma~\ref{th:ns:sm:nonlin}, we get
	\begin{equation}\label{\eqpre:tmp2}
		\norm{\tau_k^{\frac \ell 2} (u - \ik u)}_{\sLJ2 \sV{s+1}} \leq C k^{\frac{2+\ell - s}{2}} \nTS{\tau^{\frac{2+s}{2}} \partial_{tt} u}_{\sLJ2 \sV{s+1}} \leq \Cas{ns:low} k^{\frac{2 + \ell - s}{2}}
	\end{equation}
	and from lemma~\ref{th:ns:sm:euler} that
	\begin{equation}\label{\eqpre:tmp3}
		k^{\frac \ell 2} \nS{v_k(t_{n_0})}_{\sV{s}} \leq \Cas{ns:low} k^{\frac{2 + \ell - s}{2}}.
	\end{equation}
	Furthermore, $t_{n_0} \geq C k$ and lemma~\ref{th:ns:sm:nonlin} implies, with $\ell - s - s \leq 0$, that
	\begin{equation}\label{\eqpre:tmp3b}
		\begin{aligned}
			\nTS{\tau_k^{\frac{\ell}{2}} \ak \oCdiff(u, \ak \ik u)}_{\sLJ2 \sH{s-1}} &\leq C \max_{t \in J} (\tau_k(t))^{\frac{\ell - 2 - s}{2}} \nTS{\tau_k^{\frac{2+s}{2}} \ak \oCdiff(u, \ak \ik u)}_{\sLJ2 \sH{s-1}} \\
			&\leq \Cas{ns:low} \restr[1]{\tau_k^{\frac{\ell - 2 -s}{2}}}{I^{n_0+1}} k^2 \leq \Cas{ns:low} k^{\frac{2 + \ell - s}{2}}.
		\end{aligned}
	\end{equation}
	The main difficulty in the proof of \eqref{\eqpre:tmp1} is the estimation of $\oCdiff(\ak u_k, \ak \ik u)$ in the error identity \eqref{eq:ns:error-u}. To resolve the dependencies we prove \eqref{\eqpre:tmp1} for $\ell = 0$ in the order $s = 0, -1, -2, 1, \ldots, s_0 + L$, for $\ell = 1$ in the order $s = 0, -1, 1, \ldots, s_0 + L - 1$ and for $\ell \geq 2$ in the order $s = -2 + \ell, \ldots, s_0 + L - \ell$. If we consider some $\ell = 0, \ldots, L$ we implicitly assume in the following that \eqref{\eqpre:tmp1} was proven for smaller $\ell$ if $\ell > 0$.

	If $\ell = 1, 2$ then $s = 0$ is a node in the smoothing cascade at level $\ell$. Proceeding for the error identity \eqref{eq:ns:error-u} just as in the smoothing stability estimate from lemma~\ref{th:st:sm:step} we have for $\ell > 0$ and $n = n_0 + 1, \ldots, N$ with $J' \coloneqq (t_{n_0}, t_n]$ that
	\begin{align*}
		\MoveEqLeft \nS{\restr[1]{\tau_k^{\frac{\ell}{2}}}{I^n} v_k(t_n)}^2_{\sL2} + \nTS{\tau_k^{\frac{\ell}{2}} \ak v_k}^2_{\sLJp2 \sV1}\\
		&\leq C \bigl( k^\ell \norm{v_k(t_{n_0})}_{\sL2}^2 + \norm{\tau_k^{\frac \ell 2}(u - \ik u)}_{\sLJp2\sV1}^2 + \norm{\tau_k^{\frac \ell 2} \ak \oCdiff(u, \ak \ik u)}_{\sLJp2 \sL2}^2 \\
		&\qquad+ \abs{\spLL{\oCdiff(\ak u_k, \ak \ik u), \tau_k^\ell \ak v_k}_{J'}} +\nTS{\tau_k^{\frac{\ell - 1}{2}} \ak v_k}_{\sLJp2\sL2}^2 + k^2 \nTS{\tau_k^{\frac{\ell -1}{2}} \partial_t v_k}_{\sLJp2\sL2}^2 \bigr).
	\end{align*}
	For the first three terms on the right-hand side we use \eqref{\eqpre:tmp2}, \eqref{\eqpre:tmp3} and \eqref{\eqpre:tmp3b} with $s = 0$. For the two last terms we use the validity of \eqref{\eqpre:tmp1} for $\ell-1$, just as in the linear case. This yields
	\begin{equation}\label{\eqpre:s0}
		\begin{aligned}
			\MoveEqLeft\nS{\restr[1]{\tau_k^{\frac{\ell}{2}}}{I^n} v_k(t_n)}^2_{\sL2} + \nTS{\tau_k^{\frac{\ell}{2}} \ak v_k}^2_{\sLJp2 \sV1} \\
			&\leq \Cas{ns:low} k^{2+\ell} + C \abs{\spLL{\oCdiff(\ak u_k, \ak \ik u), \tau_k^\ell \ak v_k}_{J'}}.
		\end{aligned}
	\end{equation}
	If $\ell = 0$ we can proceed similarly, without the terms $\ak v_k$ and $\partial_t v_k$ on the right-hand side, and also arrive at \eqref{\eqpre:s0}. For the nonlinear term in \eqref{\eqpre:s0} we use the cancellation property to get
	\begin{align*}
		\MoveEqLeft\spLL{\oCdiff(\ak u_k, \ak \ik u), \tau_k^\ell \ak v_k}_{J'} = \spLL{\ak v_k \cdot \nabla \ak \ik u, \tau_k^\ell \ak v_k}_{J'}\\
		&\leq C \nTS{u}_{\sCJp0\sV2} \nTS{\tau_k^{\frac \ell 2} \ak v_k}_{J'} \nTS{\tau_k^{\frac \ell 2} \ak v_k}_{\sLJp2 \sV1}.
	\end{align*}
	Using Young's inequality to move the last term to the left-hand side we get
	\begin{equation}\label{\eqpre:s0-1}
		\nS{\restr[1]{\tau_k^{\frac{\ell}{2}}}{I^n} v_k(t_n)}^2_{\sL2} + \nTS{\tau_k^{\frac{\ell}{2}} \ak v_k}^2_{\sLJp2 \sV1} \leq \Cas{ns:low} k^{2+\ell} + C \nTS{u}_{\sCJp0 \sV2}^2 \nTS{\tau_k^{\frac \ell 2} \ak v_k}^2_{J'}.
	\end{equation}
	With the elementary estimate
	\begin{equation}\label{\eqpre:gronprep}
		\nTS{\tau_k^{\frac \ell 2} \ak v_k}^2_{\sLJp2\sL2} = \tfrac 1 4 \sum_{j = n_0+1}^n k_j \restr{\tau_k^\ell}{I^j} \nS{v_k^j + v_k^{j-1}}_{\sL2}^2 \leq C \sum_{j = n_0}^{n} k_j \left(\restr[1]{\tau_k^{\frac \ell 2}}{I^j} \nS{v_k^j}_{\sL2}\right)^2
	\end{equation} 
	we can employ Gronwall's inequality in \eqref{\eqpre:s0-1} and conclude that
	\begin{equation}\label{\eqpre:s0-2}
		\nTS{\tau_k^{\frac \ell 2} v_k}^2_{\sLJ\infty \sL2} + \nTS{\tau_k^{\frac \ell 2} \ak v_k}^2_{\sLJ2 \sV1} \leq \Cas{ns:low} k^{2 + \ell}
	\end{equation}
	holds under the stepsize condition. This is \eqref{\eqpre:tmp1} for $s = 0$ except for the $\partial_t v_k$-term on the left-hand side, the proof of which must be postponed until $s = 1$ has been considered. Any level containing $s = 0$ also contains $s = 1$ and the $\partial_t v_k$-estimate is only used in level $\ell + 1$, so the arguments are not disturbed by this postponement.

	For $s = -1$ and $\ell = 0, 1$ or $s = -2$ and $\ell = 0$ we use arguments similar to those leading to \eqref{\eqpre:s0}, but now also norm estimates for the remaining nonlinear term, to get
	\begin{equation}\label{\eqpre:sl0}
		\begin{aligned}
			\MoveEqLeft\nS{\restr[1]{\tau_k^{\frac{\ell}{2}}}{I^n} v_k(t_n)}^2_{\sV{s}} + \nTS{\tau_k^{\frac{\ell}{2}} \ak v_k}^2_{\sLJp2 \sV{s+1}} \\
			&\leq \Cas{ns:low} k^{2+\ell-s} + C \nTS{\tau_k^{\frac{\ell}{2}} \oCdiff(\ak u_k, \ak \ik u)}_{\sLJp2 \sH{s-1}}^2.
		\end{aligned}
	\end{equation}
	Using \eqref{eq:nl:sl1} we can estimate
	\begin{equation}\label{\eqpre:pregron-pre}
		\begin{aligned}
			\MoveEqLeft\nTS{\tau_k^{\frac \ell 2} \oCdiff(\ak u_k, \ak \ik u)}_{\sLJp2\sH{s-1}}^2 \\
			&\leq C \left( \nTS{\ak v_k}^2_{\sLJp\infty\sL2} \nTS{\tau_k^{\frac \ell 2} \ak v_k}^2_{\sLJp2 \sV{s+2}} + \nTS{u}_{\sLJp\infty\sV2}^2 \nTS{\tau_k^{\frac \ell 2} \ak v_k}_{\sLJp2\sV{s}}^2 \right).
		\end{aligned}
	\end{equation}
	By \eqref{\eqpre:s0-2}, for the current $\ell$ and $\ell = 0$, we have $\nTS{\ak v_k}^2_{\sLJp\infty\sL2} \nTS{\tau_k^{\frac \ell 2} \ak v_k}^2_{\sLJp2 \sV1} \leq \Cas{ns:low} k^{4+\ell}$ from which we can conclude for $s = -1$ that
	\begin{equation}\label{\eqpre:pregron}
		\nTS{\tau_k^{\frac \ell 2} \oCdiff(\ak u_k, \ak \ik u)}_{\sLJp2\sH{s-1}}^2 \leq \Cas{ns:low} k^{2 + \ell - s} + C \nTS{u}_{\sLJp\infty\sV2}^2 \nTS{\tau_k^{\frac \ell 2} \ak v_k}_{\sLJp2\sV{s}}^2 .
	\end{equation}
	Using this estimate in \eqref{\eqpre:sl0} and proceeding just as in \eqref{\eqpre:gronprep} we can apply Gronwall's inequality to arrive at 
	\begin{equation*}
		\nTS{\tau_k^{\frac \ell 2} v_k}^2_{\sLJ\infty \sV{s}} + \nTS{\tau_k^{\frac \ell 2} \ak v_k}^2_{\sLJ2 \sV{s+1}} \leq \Cas{ns:low} k^{2 + \ell - s}
	\end{equation*}
	for $s = -1$. From \eqref{\eqpre:pregron} we get $\nTS{\tau_k^{\frac \ell 2} \oCdiff(\ak u_k, \ak \ik u)}_{\sLJp2\sH{s-1}}^2 \leq \Cas{ns:low} k^{2+\ell-s}$. This allows us to apply the	linear stability estimates from lemma~\ref{th:st:disc-stab} for $\ell = 0$ and lemma~\ref{th:st:sm:step} for $\ell = 1$ to arrive at the estimates for $\partial_t v_k$ as well, concluding the proof of \eqref{\eqpre:tmp1} for $s = -1$. For $s = -2$ we start at \eqref{\eqpre:sl0} and now use $\nTS{\ak v_k}^2_{\sLJp\infty\sL2} \nTS{\ak v_k}^2_{\sLJp2 \sL2} \leq \Cas{ns:low} k^5$ in \eqref{\eqpre:pregron-pre} from the established results for $s = 0$ and $s = -1$. Proceeding just as for $s = -1$ we arrive at \eqref{\eqpre:tmp1} for $s = -2$.
	
	For $s \geq 1$ we directly apply the linear stability estimates, from lemma~\ref{th:st:disc-stab} if $\ell = 0$ and lemma~\ref{th:st:sm:step} if $\ell > 0$, to the error identity \eqref{eq:ns:error-u}. Estimating the initial error, the linear residual $P \Delta (u - \ik u)$ and using the induction hypothesis for $\ell > 0$ just as for $s \leq 0$, we arrive at
	\begin{equation}\label{\eqpre:sge1}
		\begin{aligned}
			\MoveEqLeft\nTS{\tau_k^{\frac \ell 2} v_k}^2_{\sLJ\infty \sV{s}} + \nTS{\tau_k^{\frac \ell 2} \ak v_k}^2_{\sLJ2 \sV{s+1}} + \nTS{\tau_k^{\frac \ell 2} \partial_t v_k}^2_{\sLJ2 \sV{s-1}} \\
			&\leq \Cas{ns:low} k^{2 + \ell - s} + C \nTS{\tau_k^{\frac \ell 2} \oCdiff(\ak u_k, \ak \ik u)}^2_{\sLJ2 \sH{s-1}}.
		\end{aligned}
	\end{equation}
	For $s = 1$ and hence $0 \leq \ell \leq 3$ we have, using \eqref{eq:nl:2b} and \eqref{eq:nl:3b}, that
	\begin{align*}
		\MoveEqLeft\nTS{\tau_k^{\frac \ell 2} \oCdiff(\ak u_k, \ak \ik u)}_{\sLJ2\sL2}^2\\
		&\leq C \int_J \tau_k^\ell \nS{\ak v_k}_{\sV1}^3 \nS{\ak v_k}_{\sV2} + \tau_k^\ell \nS{\ak v_k}^2_{\sV1} \nS{\ak \ik u}^2_{\sV2} \dif t\\
		&\leq C(\delta) ( \nTS{\ak v_k}_{\sLJ\infty\sV1}^4 + \nTS{u}_{\sC0 \sV1}^2) \nTS{\tau_k^{\frac \ell 2} \ak v_k}^2_{\sLJ2\sV1} + \delta \nTS{\tau_k^{\frac \ell 2} \ak v_k}^2_{\sLJ2 \sV2}
	\end{align*}
	with $\delta > 0$ to be chosen later. Estimate \eqref{\eqpre:s0-2} on the level $0 \leq (\ell - 1)^+ \leq 2$ yields
	\begin{equation*}
		\nTS{\tau_k^{\frac \ell 2} \ak v_k}^2_{\sLJ2 \sV1} \leq \nTS{\tau_k^{\frac{(\ell-1)^+}{2}} \ak v_k}^2_{\sLJ2 \sV1} \leq \Cas{ns:low} k^{2 + (\ell - 1)^+} \leq \Cas{ns:low} k^{1 + \ell}.
	\end{equation*}
	The inverse inequality implies that $\nTS{\ak v_k}_{\sLJ\infty \sV1} \leq C k^{-\frac 12} \nTS{\ak v_k}_{\sLJ2 \sV1} \leq \Cas{ns:low}$ and thus 
	\begin{equation*}
		\nTS{\tau_k^{\frac \ell 2} \oCdiff(\ak u_k, \ak \ik u)}_{\sLJ2\sL2}^2 \leq \Cas{ns:low}(\delta) k^{1+\ell} + \delta \nTS{\tau_k^{\frac \ell 2} \ak v_k}^2_{\sLJ2 \sV2}.
	\end{equation*}
	Combining this with \eqref{\eqpre:sge1} and moving the last term to the left-hand side for $\delta$ small enough, we arrive at \eqref{\eqpre:tmp1} for $s = 1$. For $s \geq 2$ we use \eqref{eq:nl:sge1} to get 
	\begin{equation}\label{\eqpre:tmp4a}
		\begin{aligned}
			\MoveEqLeft\nS{\tau_k^{\frac \ell 2} \oCdiff(\ak u_k, \ak \ik u)}_{\sLJ2\sH{s-1}}^2 \\
			&\leq C \nS{\tau_k^{\frac \ell 2} \oCdiff(\ak u_k, \ak \ik u)}_{\sLJ2\sL2}^2 \\
			&\quad+ C \sum_{i = 1}^{s-1} \nS{\tau_k^{\frac \ell 2} \ak v_k}^2_{\sLJ2\sV{i+1}} \left( \nS{\ak v_k}^2_{\sLJ\infty\sV{s-i+1}} + \nS{\ak \ik u}_{\sLJ\infty\sV{s-i+1}}^2 \right).
		\end{aligned}
	\end{equation}
	By similar arguments as for the case $s = 1$ we have
	\begin{equation}\label{\eqpre:tmp4}
		\nTS{\tau_k^{\frac \ell 2} \oCdiff(\ak u_k, \ak \ik u)}_{\sLJ2\sL2}^2 \leq \Cas{ns:low} \nTS{\tau_k^{\frac \ell 2} \ak v_k}^2_{\sLJ2\sV1} + C \nTS{\tau_k^{\frac \ell 2} \ak v_k}^2_{\sLJ2 \sV2}.
	\end{equation}
	Using \eqref{\eqpre:s0-2} for the first term we get, since $0 \leq (\ell - s)^+ \leq 2$, that
	\begin{equation*}
		\nTS{\tau_k^{\frac \ell 2} \ak v_k}^2_{\sLJ2\sV1} \leq \nTS{\tau_k^{\frac{(\ell - s)^+}{2}} \ak v_k}^2_{\sLJ2\sV1} \leq \Cas{ns:low} k^{2+(\ell-s)^+} \leq \Cas{ns:low} k^{2 + \ell -s}
	\end{equation*}
	and using \eqref{\eqpre:tmp1} for $s = 1$, since $0 \leq (\ell - s + 1)^+ \leq 3$, that 
	\begin{equation*}
		\nTS{\tau_k^{\frac \ell 2} \ak v_k}^2_{\sLJ2 \sV2} \leq \nTS{\tau_k^{\frac{(\ell - s + 1)^+}{2}} \ak v_k}^2_{\sLJ2 \sV2} \leq \Cas{ns:low} k^{2+(\ell - s + 1)^+ -1} \leq \Cas{ns:low} k^{2 + \ell - s}.
	\end{equation*}
	From \eqref{\eqpre:tmp4} we hence conclude the $\sL2$-estimate for the nonlinear term:
	\begin{equation}\label{\eqpre:tmp5}
		\nTS{\tau_k^{\frac \ell 2} \oCdiff(\ak u_k, \ak \ik u)}_{\sLJ2\sL2}^2 \leq \Cas{ns:low} k^{2+\ell-s}.
	\end{equation}
	For the remaining terms in \eqref{\eqpre:tmp4a} we want to prove that, for $i = 1, \ldots, s -1$,
	\begin{equation}\label{\eqpre:tmp6}
		\nTS{\tau_k^{\frac{\ell}{2}} \ak v_k}^2_{\sLJ2\sV{i+1}} \nTS{\ak w_k}^2_{\sLJ\infty\sV{s-i+1}} \leq \Cas{ns:low} k^{2+\ell-s} 
	\end{equation}
	with either $w_k = v_k$ or $w_k = \ik u$. Combined with \eqref{\eqpre:tmp5} this would imply
	\begin{equation}\label{\eqpre:tmp7}
		\nS{\tau_k^{\frac \ell 2} \oCdiff(\ak u_k, \ak \ik u)}_{\sLJ2\sH{s-1}}^2 \leq \Cas{ns:low} k^{2+\ell-s}
	\end{equation}
	which together with \eqref{\eqpre:sge1} would conclude our proof of \eqref{\eqpre:tmp1} for $s \geq 2$. For the first term in \eqref{\eqpre:tmp6} we use that
	\begin{equation}\label{\eqpre:tmp8}
		\norm{\tau_k^{\frac{(\ell-1)^+}{2}} \ak v_k}^2_{\sLJ2\sV{i+1}} \leq \Cas{ns:low} k^{2+(\ell -1)^+-i}
	\end{equation}
	for $i = 1, \ldots, s-1$ by \eqref{\eqpre:tmp1}. These $i$ belong to the cascade at level $(\ell-1)^+$ since $s_0 \leq 2$ implies $(\ell - 1)^+ -2 \leq L - 3 \leq s_0 - 1 \leq 1 \leq i$ and from $s \leq s_0 + L - \ell$ it follows that $i \leq s - 1 \leq s_0 + L - (\ell - 1) \leq s_0 + L - (\ell-1)^+$. For the second factor in \eqref{\eqpre:tmp6} we combine in case $w_k = v_k$ the inverse inequality with \eqref{\eqpre:tmp1} for level $0$:
	\begin{equation*}
		\norm{\ak v_k}_{\sLJ\infty \sV{s-i+1}}^2 \leq C k^{-1} \norm{\ak v_k}_{\sLJ2 \sV{s-i+1}}^2 \leq \Cas{ns:low} k^{-1} k^{2-(s-i)} \leq \Cas{ns:low} k^{2 - (s-i+1)}.
	\end{equation*}
	For $w_k = \ik u$ we use the smoothing estimate of $u$ to get
	\begin{equation*}
		\norm{\ak \ik u}_{\sLJ\infty \sV{s-i+1}}^2 \leq C k^{-(s-i-1)^+} \norm{\tau^{\frac{(s-i-1)^+}{2}} u}_{\sCJ0 \sV{s-i+1}}^2 \leq \Cas{ns:low} k^{-(s-i-1)^+}.
	\end{equation*}
	Combining these two estimates for $w_k$ with \eqref{\eqpre:tmp8} we get for \eqref{\eqpre:tmp6} that
	\begin{equation*}
		\nTS{\tau_k^{\frac{\ell}{2}} \ak v_k}^2_{\sLJ2\sV{i+1}} \nTS{\ak w_k}^2_{\sLJ\infty\sV{s-i+1}} \leq \Cas{ns:low} k^{2+(\ell-1)^+ - i - (s-i-1)^+}
	\end{equation*}
	and a simple, but tedious, examination of all cases yields that \eqref{\eqpre:tmp6} indeed holds, implying \eqref{\eqpre:tmp7} and hence the claimed \eqref{\eqpre:tmp1} for $s \geq 2$.

	To finish the proof of \eqref{\eqpre:tmp1} we must show $\norm{\tau_k^{\frac \ell 2} \partial_t v_k}^2_{\sLJ2 \sV{-1}} \leq \Cas{ns:low} k^{2+\ell}$ from the case $s = 0$. A procedure similar to e.g.\ $s = -1$, i.e.\ using the linear stability, yields
	\begin{equation}\label{\eqpre:tmp9}
		\norm{\tau_k^{\frac \ell 2} \partial_t v_k}^2_{\sLJ2 \sV{-1}} \leq \Cas{ns:low} k^{2+\ell} + \nTS{\tau_k^{\frac \ell 2} \oCdiff(\ak u_k, \ak \ik u)}_{\sLJ2\sH{-1}}^2.
	\end{equation}
	From \eqref{\eqpre:pregron-pre} with $s = 0$ we get
	\begin{equation*}
		\nTS{\tau_k^{\frac \ell 2} \oCdiff(\ak u_k, \ak \ik u)}_{\sLJ2\sH{-1}}^2 \leq C \nTS{\ak v_k}^2_{\sLJ\infty\sL2} \nTS{\tau_k^{\frac \ell 2} \ak v_k}^2_{\sLJ2 \sV2} + \Cas{ns:low} \nTS{\tau_k^{\frac \ell 2} \ak v_k}_{\sLJ2\sL2}^2.
	\end{equation*}
	Using the estimate \eqref{\eqpre:s0-2} together with \eqref{\eqpre:tmp1} for $s = -1$ and $s = 1$ we get
	\begin{equation*}
		\nTS{\tau_k^{\frac \ell 2} \oCdiff(\ak u_k, \ak \ik u)}_{\sLJ2\sH{-1}}^2 \leq \Cas{ns:low} ( k^2 k^{2+\ell-1} + k^{2+\ell+1} ) \leq \Cas{ns:low} k^{2 + \ell}
	\end{equation*}
	which, inserted into \eqref{\eqpre:tmp9}, concludes the case $s = 0$ and hence the proof of \eqref{\eqpre:tmp1}. Splitting $u - u_k = (u - \ik u) + (\ik u - u_k)$ and using \eqref{\eqpre:tmp1} and
	\begin{equation*}
		\norm{\tau_k^{\frac{L}{2}} (u - \ik u)}_{\sLJ\infty \sV{s_0}} \leq C k^2 \norm{\tau_k^{\frac{L}{2}} \partial_{tt} u}_{\sLJ\infty \sV{s_0}} \leq C k^2 \norm{\tau^{\frac{L}{2}} \partial_{tt} u}_{\sLJ\infty \sV{s_0}} \leq \Cas{ns:low} k^2.
	\end{equation*}
	and similarly $\norm{\tau_k^{\frac L 2} \ak (u - \ik u)}_{\sLJ2\sV{s_0+1}} \leq \Cas{ns:low} k^2$, we arrive at the velocity error estimates in \eqref{eq:ns:sm:error-l2} and \eqref{eq:ns:sm:error-linf}:
	\begin{equation}\label{\eqpre:v}
		\norm{\tau_k^{\frac{L}{2}} (u - u_k)}_{\sLJ\infty \sV{s_0}} + \norm{\tau_k^{\frac{L}{2}} \ak (u - u_k)}_{\sLJ2 \sV{s_0+1}} \leq \Cas{ns:low} k^2.
	\end{equation}

	For the pressure error we use the error identity~\eqref{eq:ns:error-p} and Poincar\'e's inequality. For $s_0 = 1$ we get
	\begin{align*}
		\nTS{\tau_k^{\frac 32} (\ak p - p_k)}_{\sLJ2 \sQ1} \leq C \Bigl( &\nTS{\tau_k^{\frac 32} \ak (u - u_k)}_{\sLJ2 \sV2} \\&+ \nTS{\tau_k^{\frac 32} \ak \oCdiff(u, \ak \ik u)}_{\sLJ2 \sL2} + \nTS{\tau_k^{\frac 32} \oCdiff(\ak u_k, \ak \ik u)}_{\sLJ2 \sL2} \Bigr).
	\end{align*}
	For the terms on the right we use the velocity estimate \eqref{\eqpre:v}, lemma~\ref{th:ns:sm:nonlin} and \eqref{\eqpre:tmp5} for $s = s_0$ and $\ell = L$. This yields as claimed
	\begin{align*}
		\nTS{\tau_k^{\frac 32} (\ak p - p_k)}_{\sLJ2 \sQ1} \leq \Cas{ns:low} k^2
	\end{align*}
	which finishes the proof of \eqref{eq:ns:sm:error-l2}. For $s_0 = 2$ we use lemma~\ref{th:midpoint-quadratic-estimate} to estimate
	\begin{align*}
		\nTS{\tau_k^2 (\mk p - p_k)}_{\sLJ\infty \sQ1} &\leq \nTS{\tau_k^2 (\mk p - \ak p)}_{\sLJ\infty \sQ1} + \nTS{\tau_k^2 (\ak p - p_k)}_{\sLJ\infty \sQ1} \\
		&\leq C k^2 \nTS{\tau^2 \partial_{tt} p}_{\sLJ\infty \sQ1} + \nTS{\tau_k^2 (\ak p - p_k)}_{\sLJ\infty \sQ1}
	\end{align*}
	and by \eqref{reg:ns:sm2} there holds $\nTS{\tau^2 \partial_{tt} p}_{\sLJ\infty \sQ1} \leq \Cas{ns:low}$. It hence remains to estimate $\ak p - p_k$. The pressure error identity implies
	\begin{equation}\label{\eqpre:pinf}
		\begin{aligned}
			\nTS{\tau_k^2 (\ak p - p_k)}_{\sLJ\infty \sQ1} \leq C \Bigl( &\nTS{\tau_k^2 \ak (u - u_k)}_{\sLJ\infty\sV2} + \nTS{\tau_k^2 \ak \oCdiff(u, \ak \ik u)}_{\sLJ\infty \sL2} \\
			&+ \nTS{\tau_k^2 \oCdiff(\ak u_k, \ak \ik u)}_{\sLJ\infty \sL2} \Bigr).
		\end{aligned}
	\end{equation}
	For the first term on the right we again use the velocity estimate \eqref{\eqpre:v}. For the nonlinear terms no temporal $\sL\infty$-estimates have been derived so far. Modifying the proof leading to \eqref{eq:ns:sm:nonlin} for $s = 1$ in lemma~\ref{th:ns:sm:nonlin} we get
	\begin{align*}
		\nTS{\tau_k^2 \oCdiff(u, \ak \ik u)}_{\sLJ\infty \sL2} &\leq C k^2 (\nTS{\partial_t u}_{\sLJ\infty \sL2} \norm{\tau_k^2 \partial_t u}_{\sLJ\infty\sV3} + \nTS{u}_{\sCJ0 \sV2} \norm{\tau_k^2 \partial_{tt} u}_{\sLJ\infty \sV2}).
	\end{align*}
	Using \eqref{reg:ns:sm1} this implies that
	\begin{equation}\label{\eqpre:tmp10}
		\nTS{\tau_k^2 \oCdiff(u, \ak \ik u)}_{\sLJ\infty \sL2} \leq \Cas{ns:low} k^2.
	\end{equation}
	For the second nonlinear term we can proceed as above for $s = 1$, after \eqref{\eqpre:sge1}, to get
	\begin{equation}\label{\eqpre:tmp11}
		\begin{aligned}
			\MoveEqLeft\nTS{\tau_k^2 \oCdiff(\ak u_k, \ak \ik u)}_{\sLJ\infty \sL2} \\
			&\leq C \left( \nTS{v_k}_{\sLJ\infty \sV1} \nTS{\tau_k^{\frac 32} v_k}_{\sLJ\infty \sV1}^{\frac 12} \nTS{\tau_k^2 v_k}_{\sLJ\infty \sV2}^{\frac 12} + \nTS{\tau_k^{\frac 32} v_k}_{\sLJ\infty \sV1} \nTS{u}_{\sCJ0 \sV2} \right).
		\end{aligned}
	\end{equation}
	By \eqref{\eqpre:tmp1} for $\ell = 0$ and $s = 1$ we get $\nTS{v_k}_{\sLJ\infty \sV1} \leq \Cas{ns:low} k^{\frac 12}$, for $\ell = 3$ and $s = 1$ we get $\nTS{\tau_k^{\frac 32} v_k}_{\sLJ\infty \sV1} \leq \Cas{ns:low} k^2$ and for $\ell = 4$ and $s = 2$ we get $\nTS{\tau_k^2 v_k}_{\sLJ\infty \sV2} \leq \Cas{ns:low} k^2$. Using these estimates in \eqref{\eqpre:tmp11} implies
	\begin{equation*}
		\nTS{\tau_k^2 \oCdiff(\ak u_k, \ak \ik u)}_{\sLJ\infty \sL2} \leq \Cas{ns:low} k^2.
	\end{equation*}
	Using this estimate, the velocity estimate \eqref{\eqpre:v} and \eqref{\eqpre:tmp10} in \eqref{\eqpre:pinf} we conclude that
	\begin{equation*}
		\nTS{\tau_k^2 (\ak p - p_k)}_{\sLJ\infty \sQ1} \leq \Cas{ns:low} k^2
	\end{equation*}
	which finishes the proof of \eqref{eq:ns:sm:error-linf}.
\end{proof}

\begin{remark}
	The stepsize condition in theorem~\ref{th:ns:sm:error} is stronger than in section~\ref{sec:high-reg} without smoothing. This is due to the application of Gronwall's inequality in \eqref{eq:th:ns:sm:error:vk} for $s = -2$ where sharper estimates for the nonlinearity, like \eqref{eq:nl:s0}, could not be proven.
\end{remark}

\section{Numerical Study}\label{sec:num}

We present a numerical study illustrating the optimality of the error estimates and the necessity to consider both a weighted norm and initial Euler steps, if the initial data does not satisfy the compatibility conditions. On the unit disk $\Omega=\{(x,y)\in\mathbb{R}^2\,:\, x^2+y^2<1\}$ and the temporal interval $I=[0,2]$ we study the Navier-Stokes equations with homogeneous Dirichlet data
\[
\partial_t u - 0.01 \Delta u + u \cdot \nabla u + \nabla p = f, \quad \div u = 0 \text{ in $\Omega$},\quad 
u=0 \text{ on $\partial\Omega$}.
\]
Two different configurations are considered. First, \emph{(i)}, we
prescribe homogeneous initial data $u^0=0$ and the smooth right hand side 
\[
f(x,y,t) = 0.2 t^2\exp(-t) \big(-\sin(4x+y)y,\cos(x-4y)x\big). 
\]
It holds $f(\cdot,0)=0$ and $\partial_t f(\cdot,0)=0$ such that the data satisfies the compatibility conditions. Second, \emph{(ii)}, we consider the homogeneous right hand side $f=0$ and determine the initial condition $u^0$  as solution to the stationary Navier-Stokes problem
\[
- 0.01 \Delta u^0 + u^0 \cdot \nabla u^0 + \nabla p^0 =
0.2 \operatorname{sgn}(x)\operatorname{sgn}(y)\big(-\sin(4x+y)y,\cos(x-4y)x\big)\text{ in $\Omega$}
\]
with $u^0=0$ on the boundary. Since the domain is regular and the initial right hand side is just in $L^2(\Omega)$, it holds $u^0\in V^2$. The compatibility condition is not satisfied with the right hand side $f=0$.

Spatial discretization is accomplished with quadratic finite elements for velocity and pressure on a mesh with mesh size $h\approx 0.0025$. To cope with the missing inf-sup stability we employ the local projection stabilization, see~\cite{BeckerBraack2001}. For temporal discretization we use the Crank-Nicolson scheme as discussed in this paper. To avoid superconvergence effects by symmetry, we consider a base step size $k\in \{0.02,0.01,0.005,0.0025\}$ and add an alternating variation $0.8k,1.2k,0.8k,1.2k,\dots$ in both test cases. Without this modification, no reduction in convergence could be observed in case \emph{(ii)}, even if no initial Euler steps where performed.  The nonlinear problems are approximated with a Newton scheme, the resulting linear systems are solved with a geometric multigrid solver. For details on the implementation in \textsc{Gascoigne 3D}~\cite{Gascoigne3D} see~\cite[chapters 4, 7]{Richter2017}. A reference solution $p_{k_0,h}$ is computed on a uniform time mesh with $M_0=\frac{2}{k_0}$ steps of size $k_0=0.0005$ on the same spatial mesh. Pressure errors are evaluated in the $L^2$- and the $L^\infty$-norm, approximated by the midpoint rule on the reference subdivision with stepsize $k_0$ in time and by the Euclidean $l^2$-norm on the fixed discretization in space. With $t_m=(m-\ahalf)k_0$
\[
\begin{aligned}
  \|p_{k,h}-p_{k_0,h}\|_{L^2 l^2}&:= \left(\sum_{m=1}^{M_0}k_0 \sum_{i=1}^N|p_{k,h}(t_m,x_i)-p_{k_0,h}(t_m,x_i)|^2\right)^\frac{1}{2},\\
  \|p_{k,h}-p_{k_0,h}\|_{L^\infty l^2}&:= \max_{m=1,\dots,M_0} \left( \sum_{i=1}^N|p_{k,h}(t_m,x_i)-p_{k_0,h}(t_m,x_i)|^2\right)^\frac{1}{2},
\end{aligned}
\]
where we denote by $x_i$ for $i=1,\dots,N$ the nodes of the spatial mesh.

\begin{figure}[t]
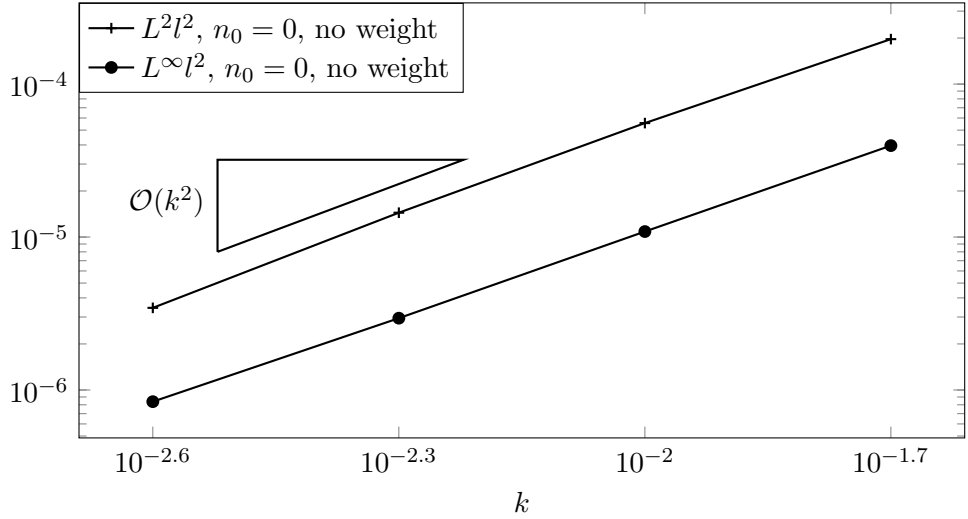

  \begin{center}
    \tikloglogpic{width=0.9\textwidth,height=0.5\textwidth,xlabel=$k$,legend style={at={(0,1)},anchor=north
        west,nodes=right},xtick={0.0025,0.005,0.01,0.02}}{
      \addplot[solid,mark=+,thick,color=black] table[row sep=crcr]{
        0.02 0.00019672953143202722\\
        0.01 5.55740877829025e-05\\
        0.005 1.4422860728274087e-05\\
        0.0025 3.4440003926816766e-06\\
      };
      \addlegendentry{$L^2l^2$, $n_0=0$, no weight}
      \addplot[solid,mark=*,thick,color=black] table[row sep=crcr]{
        0.02 3.959932059479635e-05\\
        0.01 1.0860960288764174e-05\\
        0.005 2.947917119028932e-06\\
        0.0025 8.388585709110509e-07\\
      };
      \addlegendentry{$L^\infty l^2$, $n_0=0$, no weight}
      \addplot[solid,thick,color=black] table[row sep=crcr]{
        0.003 8.e-6\\
        0.003 32.e-6\\
        0.006 32.e-6\\
        0.003 8.e-6\\
      };
      \node at (axis cs: 0.0026,17.e-6){${\cal O}(k^2)$};
    }
  \end{center}
  \caption{Results for configuration \emph{(i)} satisfying the
    compatibility condition. We observe optimal second order
    convergence in both norms. No initial Euler steps are required.} 
  \label{num:i}
\end{figure}

\begin{figure}[t]
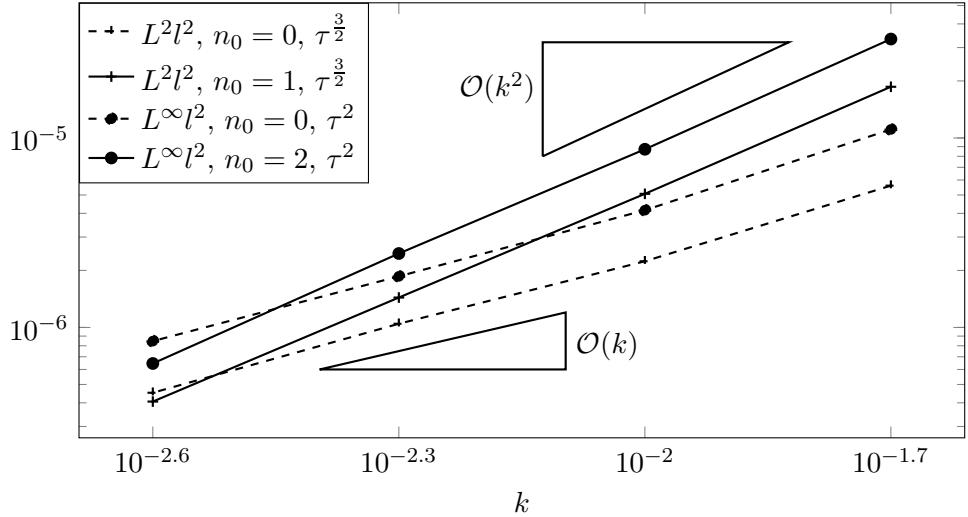

  \begin{center}
\tikloglogpic{width=0.9\textwidth,height=0.5\textwidth,xlabel=$k$,legend
  style={at={(0,1)},anchor=north
    west,nodes=right},xtick={0.0025,0.005,0.01,0.02}}{
  \addplot[dashed,mark=+,thick,color=black] table[row sep=crcr]{
    0.02 5.602058379894919e-06\\
    0.01 2.234497427972823e-06\\
    0.005 1.0455313858449752e-06\\
    0.0025 4.5201940148665017e-07\\
  };  \addlegendentry{$L^2l^2$, $n_0=0$, $\tau^\frac{3}{2}$}
  \addplot[solid,mark=+,thick,color=black] table[row sep=crcr]{
    0.02 1.8631356324343365e-05\\
    0.01 5.070903977473569e-06\\
    0.005 1.4378724265941399e-06\\
    0.0025 4.061031940478063e-07\\
  };
  \addlegendentry{$L^2l^2$, $n_0=1$, $\tau^\frac{3}{2}$}
  \addplot[dashed,mark=*,thick,color=black] table[row sep=crcr]{
    0.02 1.1093268300353136e-05\\
    0.01 4.156438299121408e-06\\
    0.005 1.8607606521684717e-06\\
    0.0025 8.466292392011256e-07\\
  };
  \addlegendentry{$L^\infty l^2$, $n_0=0$, $\tau^2$}
  \addplot[solid,mark=*,thick,color=black] table[row sep=crcr]{
    0.02 3.3272300776590364e-05\\
    0.01 8.716040443588668e-06\\
    0.005 2.455710657479571e-06\\
    0.0025 6.451721359237861e-07\\
  };
  \addlegendentry{$L^\infty l^2$, $n_0=2$, $\tau^2$}
  \addplot[solid,thick,color=black] table[row sep=crcr]{
    0.004  0.6e-6\\
    0.008  0.6e-6\\
    0.008  1.2e-6\\
    0.004  0.6e-6\\
    };
  \addplot[solid,thick,color=black] table[row sep=crcr]{
    0.0075 8.e-6\\
    0.0075 32.e-6\\
    0.015  32.e-6\\
    0.0075 8.e-6\\
  };
  \node at (axis cs: 0.009,0.8e-6){${\cal O}(k)$};
  \node at (axis cs: 0.0066,18.4e-6){${\cal O}(k^2)$};
}
  \end{center}
  \caption{Results for configuration \emph{(ii)}. Without initial Euler
    steps, $n_0=0$, only first order convergence is observed for the
    pressure 
    in both norms. By adding $n_0=1$ Euler step, we recover second
    order convergence in the $L^2l^2$ norm and by starting with
    $n_0=2$ Euler steps, we obtain second order in the $L^\infty l^2$
    norm. } 
  \label{num:ii}
\end{figure}

The resulting convergence behavior is shown in figure~\ref{num:i} for case \emph{(i)} and figure~\ref{num:ii} for configuration \emph{(ii)}. In configuration \emph{(i)} we observe optimal second order convergence without any weighting of the norms and without adding Euler steps. Configuration \emph{(ii)} shows the expected loss of optimality, as the problem regularity is not sufficient. Optimal order convergence is recovered if we add weights to the norms and if we start the procedure with implicit Euler steps according to theorem~\ref{th:ns:sm:error}. Adding a proper amount of Euler steps increases the convergence from first to second order. Without weighting the norms, convergence rates drop to approximately $\sqrt{k}$.



\section*{Acknowledgements}
TR acknowledges the financial support by the Federal Ministry of Education and Research of Germany, grant number 05M16NMA as well as the GRK 2297 MathCoRe, funded by the Deutsche Forschungsgemeinschaft, grant number 314838170. FS acknowledges the financial support by the GRK 2339 IntComSin, funded by the Deutsche Forschungsgemeinschaft.


\begin{thebibliography}{10}

\bibitem{AzizMonk1989}
A.K. Aziz and P.~Monk.
\newblock Continuous finite elements in space and time for the heat equation.
\newblock {\em Mathematics of Computation}, 186:255--274, 1989.

\bibitem{BeckerBraack2001}
R.~Becker and M.~Braack.
\newblock A finite element pressure gradient stabilization for the {S}tokes
  equations based on local projections.
\newblock {\em Calcolo}, 38(4):173--199, 2001.

\bibitem{Gascoigne3D}
R.~Becker, M.~Braack, D.~Meidner, T.~Richter, and B.~Vexler.
\newblock The finite element toolkit \textsc{Gascoigne}.
\newblock \textsc{http://www.gascoigne.uni-hd.de}.

\bibitem{ChrysafinosKaratzas2015}
K.~Chrysafinos and E.N. Karatzas.
\newblock Symmetric error estimates for discontinuous galerkin time-stepping
  schemes for optimal control problems constrained to evolutionary stokes
  equations.
\newblock {\em Comput. Optim. Appl.}, 60, 2015.

\bibitem{deFrutosArchillaJohnNovo2016}
J.~{de Frutos}, B.~Garc\'ia-Archilla, V.~John, and J.~Novo.
\newblock Grad-div stabilization for the evolutionary oseen problem with
  inf-sup stable finite elements.
\newblock {\em J. of Scientific Computing}, 66(3):991--1024, 2016.

\bibitem{deFrutosArchillaJohnNovo2018}
J.~{de Frutos}, B.~Garc\'ia-Archilla, V.~John, and J.~Novo.
\newblock Error analysis of non inf-sup stable discretizations of the
  time-dependent navier-stokes equations with local projection stabilization.
\newblock {\em IMA J. of Numer. Analysis}, 2018.

\bibitem{HeywoodRannacher1982}
J.~G. Heywood and R.~Rannacher.
\newblock {Finite Element Approximation of the Nonstationary Navier–Stokes
  Problem. I. Regularity of Solutions and Second-Order Error Estimates for
  Spatial Discretization}.
\newblock {\em SIAM Journal on Numerical Analysis}, 19(2):275--311, apr 1982.

\bibitem{HeywoodRannacher1990}
J.~G. Heywood and R.~Rannacher.
\newblock {Finite-Element Approximation of the Nonstationary Navier–Stokes
  Problem. Part IV: Error Analysis for Second-Order Time Discretization}.
\newblock {\em SIAM Journal on Numerical Analysis}, 27(2):353--384, 1990.

\bibitem{HussainSchieweckTurek2013}
S.~Hussain, F.~Schieweck, and S.~Turek.
\newblock An efficient and stable finite element solver of higher order in
  space and time for nonstationary incompressible flow.
\newblock {\em Int. J. Num. Meth. Fluids}, 73(11):927--952, 2013.

\bibitem{MeidnerVexler2011}
D.~Meidner and B.~Vexler.
\newblock A priori error analysis of the petrov-galerkin crank-nicolson scheme
  for parabolic optimal control problems.
\newblock {\em SIAM J. on Control and Optimization}, 49(5):2183--2211, 2011.

\bibitem{Rang2008}
J.~Rang.
\newblock Pressure corrected implicit $\theta$-schemes for the incompressible
  navie-–stokes equations.
\newblock {\em Applied Mathematics and Computation}, 201:747--761, 2008.

\bibitem{rannacher1984}
R.~Rannacher.
\newblock {Finite Element Solution of Diffusion Problems with Irregular Data}.
\newblock {\em Numerische Mathematik}, 43(2):309--327, 1984.

\bibitem{ReuskenEsser2013}
A.~Reusken and P.~Esser.
\newblock Analysis of time discretization methods for a stokes equations with a
  nonsmooth forcing term.
\newblock {\em Numerische Mathematik}, 126(2):293--319, 2013.

\bibitem{Richter2017}
T.~Richter.
\newblock {\em Fluid-structure Interactions. Models, Analysis and Finite
  Elements}, volume 118 of {\em Lecture Notes in Computational Science and
  Engineering}.
\newblock Springer, 2017.

\bibitem{Schieweck2010}
F.~Schieweck.
\newblock A-stable discontinuous galerkin-petrov time discretization of higher
  order.
\newblock {\em Numer. Math.}, 18:25--57, 2010.

\bibitem{sohr2001}
H.~Sohr.
\newblock {\em The Navier-Stokes Equations}.
\newblock Birkh{\"a}user Verlag, 2001.

\bibitem{Temam1982}
R.~Temam.
\newblock {Behaviour at time t = 0 of the solutions of semi-linear evolution
  equations}.
\newblock {\em Journal of Differential Equations}, 43(1):73--92, 1982.

\end{thebibliography}

\end{document}